\documentclass{amsart}
\usepackage{amsmath,amscd,xypic,amssymb,combelow,color,enumitem,graphicx,float,comment,tikz,listings,array}
\makeatletter
\def\squiggly{\bgroup \markoverwith{\textcolor{black}{\lower3.5\p@\hbox{\sixly \char58}}}\ULon}

\usepackage{xr-hyper}
\usepackage[
pdftex,
bookmarks=false,
colorlinks=true,
debug=true,
pdfnewwindow=true]{hyperref}

\definecolor{codegreen}{rgb}{0,0.6,0}
\definecolor{codegray}{rgb}{0.5,0.5,0.5}
\definecolor{codepurple}{rgb}{0.58,0,0.82}
\definecolor{backcolour}{rgb}{0.95,0.95,0.92}

\lstdefinestyle{mystyle}{
  backgroundcolor=\color{backcolour}, commentstyle=\color{codegreen},
  keywordstyle=\color{magenta},
  numberstyle=\tiny\color{codegray},
  stringstyle=\color{codepurple},
  basicstyle=\ttfamily\footnotesize,
  breakatwhitespace=false,
  breaklines=true,
  captionpos=b,
  keepspaces=true,
  numbers=left,
  numbersep=5pt,
  showspaces=false,
  showstringspaces=false,
  showtabs=false,
  tabsize=2
}

\makeatletter
  \@addtoreset{equation}{section}
\makeatother

\newtheorem{theorem}[subsection]{Theorem}

\newtheorem{proposition}[subsection]{Proposition}
\newtheorem{lemma}[subsection]{Lemma}
\newtheorem{corollary}[subsection]{Corollary}
\newtheorem{conjecture}[subsection]{Conjecture}
\newtheorem{definition}[subsection]{Definition}
\newtheorem{notation}[subsection]{Notation}

\theoremstyle{remark}
\newtheorem{claim}[subsection]{Claim}
\newtheorem{example}[subsection]{Example}
\newtheorem{remark}[subsection]{Remark}

\def\fa{{\mathfrak{a}}}
\def\fg{{\mathfrak{g}}}
\def\fh{{\mathfrak{h}}}
\def\fn{{\mathfrak{n}}}

\def\BC{{\mathbb{C}}}

\def\BZ{{\mathbb{Z}}}

\def\slaws{\text{standard Lyndon words}}
\def\aslaw{\text{affine standard Lyndon word}}
\def\aslaws{\text{affine standard Lyndon words}}

\def\SL{\mathrm{SL}}

\def\rk{\mathrm{rank}}
\def\wI{\widehat{I}}
\def\wQ{\widehat{Q}}
\def\wDelta{\widehat{\Delta}}
\def\imx{\wDelta^{+,\mathrm{imx}}}
\def\sb{\mathsf{b}}
\def\spanset{\mathcal{S}}
\def\convexsetim{C}
\def\convexsetre{O}
\def\leftfactorsset{\mathcal{L}}
\def\rightfactorsset{\mathcal{R}}
\newcommand{\ub}[2]{\underbrace{\SL_{#1} (\delta)}_{#2 \text{ times}}}

\def\hgt{\text{ht}}

\newcommand\iso{\,\vphantom{j^{X^2}}\smash{\overset{\sim}{\vphantom{\rule{0pt}{0.20em}}\smash{\longrightarrow}}}\,}

\DeclareMathOperator{\spn}{span}
\DeclareMathOperator{\chain}{ch}

\newcommand{\ol}{\overline}

\def\LL{{\mathrm{L}}}

\def\sb{{\mathsf{b}}}
\def\osb{\overline{\mathsf{b}}}

\def\re{\mathrm{re}}
\def\im{\mathrm{im}}
\def\ext{\mathrm{ext}}

\begin{document}

\title[Affine standard Lyndon words]
      {\Large{\textbf{Affine standard Lyndon words}}}
	
\author[Corbet Elkins and Alexander Tsymbaliuk]{Corbet Elkins and Alexander Tsymbaliuk}

\address{C.E.: Purdue University, Department of Computer Science, West Lafayette, IN, USA}
\email{cdelkins@purdue.edu}

\address{A.T.: Purdue University, Department of Mathematics, West Lafayette, IN, USA}
\email{sashikts@gmail.com}

\begin{abstract}
In this note, we establish the convexity and monotonicity for $\aslaws$ in all types, generalizing the $A$-type results
of \cite{AT}. We also derive partial results on the structure of imaginary standard Lyndon words and present a conjecture
for their general form. Additionally, we provide computer code in Appendix that, in particular, allows to
efficiently compute $\aslaws$ in exceptional types for all orders.
\end{abstract}

\maketitle


\section{Introduction}\label{sec:intro}


\subsection{Summary}\label{ssec:summary}
\

The free Lie algebras generated by a finite set $\{e_i\}_{i\in I}$ are known to have bases parametrized by
Lyndon words (see Definition~\ref{def:lyndon}, Theorem~\ref{thm:Lyndon.theorem}) for each order on $I$.
This was generalized to finitely generated Lie algebras $\fa$ in~\cite{LR}. Explicitly, if $\fa$ is generated
by $\{e_i\}_{i\in I}$, then any order on $I$ gives rise to the combinatorial basis parametrized by
standard Lyndon words (see Definition~\ref{def:standard}, Theorem~\ref{thm:standard Lyndon theorem}).

The key application of~\cite{LR} was to simple finite-dimensional $\fg$, or more precisely,
to their maximal nilpotent subalgebras $\fn^+$. Invoking the root space decomposition:
\begin{equation*}
  \fn^+ = \bigoplus_{\alpha \in \Delta^+} \BC \cdot e_\alpha \,, \qquad
  \Delta^+=\big\{\mathrm{positive\ roots}\big\},
\end{equation*}
it can be easily shown that there is a natural Lalonde-Ram bijection
\begin{equation}\label{eqn:1-to-1 intro}
  \ell \colon \Delta^+ \iso \big\{\text{standard Lyndon words}\big\}.
\end{equation}
In this context, the bracketing $\sb[\ell]$ (Definition~\ref{bracketing}) is on par with the general rule
\begin{equation*}
  [e_{\alpha} , e_{\beta}] = e_{\alpha} e_{\beta} - e_{\beta} e_{\alpha} \in \BC^\times \cdot e_{\alpha + \beta}
  \qquad \forall\, \alpha,\beta, \alpha+\beta\in \Delta^+.
\end{equation*}

A decade later, \cite{L} established an iterative Leclerc algorithm for~\eqref{eqn:1-to-1 intro}.
Moreover, the induced total order on $\Delta^+$ is convex (see Proposition~\ref{prop:fin.convex})
by earlier work~\cite{R}. This played the key role in~\cite{L}, where it was shown that natural $q$-deformations
of $\sb[\ell]$ give rise to a basis of the corresponding positive half $U_q(\fn^+)$ of Drinfeld-Jimbo quantum group,
recovering Lusztig's root generators, up to~scalars.

In the recent work of Avdieiev and the second author~\cite{AT}, the generalization to (untwisted) affine Lie algebras was initiated.
Let $\widehat{\fg}$ be the affinization of $\fg$, whose Dynkin diagram is obtained by extending the Dynkin diagram of $\fg$ with
a vertex~$0$. Thus, on the combinatorial side, we consider the alphabet $\wI=I\sqcup \{0\}$. The corresponding positive
subalgebra $\widehat{\fn}^+\subset \widehat{\fg}$ still admits the root space decomposition
$\widehat{\fn}^+=\bigoplus_{\alpha\in \wDelta^+} \widehat{\fn}^+_{\alpha}$ with $\wDelta^+=\{\mathrm{positive\ affine\ roots}\}$.
The key difference is that:
\begin{equation*}
  \dim \widehat{\fn}^+_{\alpha}=1 \quad \forall\, \alpha\in \wDelta^{+,\re} \,, \qquad
  \dim \widehat{\fn}^+_{\alpha}=|I| \quad \forall\, \alpha\in \wDelta^{+,\im}.
\end{equation*}
Here, $\wDelta=\wDelta^{+,\re}\sqcup \wDelta^{+,\im}$ is the decomposition into real and imaginary
affine roots, with $\wDelta^{+,\im}=\{k\delta | k\geq 1\}$. It is therefore natural to consider an extended set
$\widehat{\Delta}^{+,\ext}$ of~\eqref{eq:extended-affine-roots}. Then, the degree reasoning as in~\cite{LR}
provides a natural analogue of~\eqref{eqn:1-to-1 intro}:
\begin{equation}\label{eqn:affine 1-to-1 intro}
  \SL \colon \wDelta^{+,\ext} \iso \big\{ \text{affine standard Lyndon words} \big\}.
\end{equation}
In~\cite{AT}, we established a generalized Leclerc algorithm describing this bijection.
As the key application, we then used it to inductively derive formulas for all $\aslaws$ in the affine types $A^{(1)}_n$ with any order on $\wI$.

The above explicit formulas in the affine types $A^{(1)}_n$ illustrated a stunning periodicity for $\aslaws$, expressing
all of them through $\SL(\alpha)$ with $|\alpha|<|\delta|$. Using the explicit formulas, we also established
the pre-convexity:
\begin{equation}\label{eq:pre-convex}
  \alpha<\alpha+\beta<\beta \quad \mathrm{or} \quad \beta<\alpha+\beta<\alpha
  \qquad \forall\ \alpha,\beta,\alpha+\beta\in \wDelta^{+,\re},
\end{equation}
as well as the monotonicity:
\begin{equation}\label{eq:monot}
  \alpha<\alpha+\delta<\alpha+2\delta<\cdots \quad \mathrm{or} \quad \alpha>\alpha+\delta>\alpha+2\delta>\cdots
  \qquad \forall\ \alpha\in \wDelta^{+,\re}.
\end{equation}

The present note arose from an attempt to generalize the above results of~\cite{AT} to all types. In particular,
our key results are the convexity (see Theorem~\ref{thm:convexity} and Remark~\ref{rem:convexity-rephrased})
and monotonicity (see Proposition~\ref{prop:monotonicity}), generalizing~\eqref{eq:pre-convex} and
improving~\eqref{eq:monot}. We also propose a conjecture on the structure of all imaginary $\aslaws$, proving it
for all orders in any type with $0\in \wI$ being the smallest letter (see Conjecture~\ref{conj:imaginary-SL} and
Theorem~\ref{thm:imaginary-SL-orderI}). We note that our approach is completely
opposite to that of~\cite{AT}, as we establish convexity and monotonicity without having explicit formulas, and
then use them to get information on the explicit form of $\aslaws$.
The~present results play the crucial role in the follow-up paper~\cite{ET}, where we establish a periodicity among
all real $\aslaws$, generalizing~\cite{AT} (and rederiving~\cite{AT}).


\subsection{Outline}\label{ssec:outline}
\

\noindent
The structure of the present paper is the following:
\begin{itemize}[leftmargin=0.5cm]

\item[$\bullet$]
In Section~\ref{sec:setup}, we recall classical results on Lyndon and standard Lyndon words,
as well as the generalization to the case of affine root systems from~\cite{AT}.

\item[$\bullet$]
In Section~\ref{sec:Aux}, we establish some basic properties of the standard, costandard, and other factorizations
of standard Lyndon words, instrumental for this note.

\item[$\bullet$]
In Section~\ref{sec:flag}, we investigate the behavior of the bracketing with respect to different splittings of words.
For imaginary words, we replace their individual bracketings rather with the consideration of the complete flags $\spanset_\bullet^k$ from~\eqref{eq:flag}.
The key results are Propositions~\ref{prop:im.bracketing} and~\ref{prop:general.bracketing}, which also imply
that the $\aslaws$ are the same whether using standard or costandard factorizations, see Remark~\ref{rem:stan=costar-words}.
We also introduce the auxiliary sets $\convexsetim(\alpha)$ for imaginary $\alpha$ and $\convexsetre(\alpha)$ for real $\alpha$,
which are key for Section~\ref{sec:keyresults}, and establish some basic properties of the min/max elements of $\convexsetre(\alpha)$.

\item[$\bullet$]
In Section~\ref{sec:keyresults}, we prove the key results of this note. The convexity of Theorem~\ref{thm:convexity} is stated
using the above sets $\convexsetim(\alpha)$ and $\convexsetre(\alpha)$, see Definition~\ref{def:left.and.right.factor.sets},
and generalizes the pre-convexity of~\eqref{eq:pre-convex}, see Remark~\ref{rem:convexity-rephrased}. Our second main result is
Proposition~\ref{prop:monotonicity} establishing~\eqref{eq:monot} and specifying which monotonicity occurs. We also establish
some properties of the resulting increasing and decreasing chains.

\item[$\bullet$]
In Section~\ref{sec:imaginary.words}, we explore the structure of imaginary $\aslaws$ and establish relations among those.
We start by showing the compatibility of complete flags~\eqref{eq:flag} in Proposition~\ref{prop:spanset.equiv} and use it to
deduce the monotonicity of Lemma~\ref{cor:imaginary.words.decreasing}. The general structure of all imaginary $\aslaws$
is proposed in Conjecture~\ref{conj:imaginary-SL}. In Theorem~\ref{thm:imaginary-SL-orderI} we prove this result for the cases
when the smallest simple root appears only once in $\delta$ (which includes any order for any type with $0\in \wI$ being the smallest
element and recovers the result of~\cite{AT} for any order in the affine types $A^{(1)}_n$), in Proposition~\ref{prop:SL.1.form}
we prove it for $i=1$ and any order, and finally in Proposition~\ref{lem:conj_deduction} we reduce it to the $k=2$ case.
Using the latter result and computer code, we fully verified the Conjecture~\ref{conj:imaginary-SL} for exceptional types,
see Remark~\ref{rem:conj-exceptional}

\item[$\bullet$]
In Appendix~\ref{sec:app_code}, we provide the computer code that we heavily used to find the correct patterns for $\aslaws$
as well as to verify Conjecture~\ref{conj:imaginary-SL} for all orders on $\wI$ whenever $|I|\leq 8$.

\item[$\bullet$]
In Appendix~\ref{sec:app-G}, we present the tables of all $\aslaws$ in affine type $G_2^{(1)}$ for all orders
on $\wI$, which were computed using the above code.

\end{itemize}


\subsection{Acknowledgement}\label{ssec:acknowl}
\

This note represents a part of the REU project at Purdue University. We are grateful to Purdue University for
support and for the opportunity to present these results at REU math conference in April~2025. C.E.\ is grateful for the
Joel Spira Undergraduate Summer Research Award. A.T.\ is deeply indebted to Andrei Negu\c{t} for numerous inspiring
discussions over the years. A.T.\ is extremely grateful to IHES (Bures-sur-Yvette, France) for the hospitality and
wonderful working conditions in the summer 2025, where the final version of this paper was prepared. We are
grateful to the referees for their useful suggestions that improved the exposition.

Both authors were partially supported by NSF Grant DMS-$2302661$.
C.E.\ acknowledges support by Purdue Joel Spira Undergraduate Summer Research Award.
A.T.\ acknowledges support by Simons Fellowship MPS-SFM-00020373.


\section{Setup and Notations}\label{sec:setup}

In this section, we recall the classical results of~\cite{LR,L} that provide a combinatorial construction of an important
basis of finitely generated Lie algebras, with the main application to the maximal nilpotent subalgebras of simple Lie algebras.
We further invoke the appropriate generalization of this to affine cases, following~\cite{AT}.


\subsection{Lyndon words}\label{ssec:L-words}
\

Let $I$ be a finite ordered alphabet, and let $I^*$ be the set of all finite length words in the alphabet $I$.
For $u = [i_1\ldots i_k] \in I^*$, we define its length by $|u| = k$. Moreover, we consider the lexicographic order
on $I^*$ defined as follows:
$$
    [i_1\ldots i_k] < [j_1\ldots j_l]\quad \text{if } \begin{cases}
        i_1 = j_1,\ldots,i_a=j_a,i_{a+1}<j_{a+1} \text{ for some }a \geq 0, \\
          \ \  \text{or}\\
        i_1=j_1,\ldots,i_k=j_k \text{ and } k < l.
    \end{cases}
$$

\begin{definition}\label{def:lyndon}
A word $\ell=[i_1\dots i_k]$ is called \textbf{Lyndon} if it is smaller than all of its cyclic permutations:
\begin{equation*}
  [i_1 \dots i_{a-1} i_a \dots i_k] < [i_a \dots i_k i_1 \dots i_{a-1}] \quad \textit{for all} \quad a \in \{2,\dots,k\}.
\end{equation*}
\end{definition}

For a word $w = [i_1 \dots i_k]\in I^*$, the subwords:
\begin{equation*}
  w_{a|} =  [i_1 \dots i_a] \qquad \text{and} \qquad w_{|a} = [i_{k-a+1} \dots i_k]
\end{equation*}
with $0\leq a\leq k$ are usually called a \textbf{prefix} and a \textbf{suffix} of $w$, respectively. We call such a prefix or
a suffix proper if $0<a<k$. It is well-known that Definition~\ref{def:lyndon} is equivalent to the following one:

\begin{definition}\label{def:lyndon-2}
A word $w$ is Lyndon if it is smaller than all of its proper suffixes:
\begin{gather*}
   w < w_{| a} \quad \textit{for all} \quad 0 < a < |w|.
\end{gather*}
\end{definition}

As a corollary, we record the following basic property:

\begin{lemma}\label{lemma:lyndon}
If $\ell_1 < \ell_2$ are Lyndon, then $\ell_1\ell_2$ is also Lyndon, and so $\ell_1\ell_2 < \ell_2\ell_1$.
\end{lemma}

Let us now recall several basic facts from the theory of Lyndon words (cf.~\cite{Lo}).

\begin{proposition}\label{prop:cost.factor}
Any Lyndon word $\ell$ with $|\ell|>1$ has a factorization
\begin{equation}\label{eqn:cost.factor}
  \ell = \ell_1 \ell_2
\end{equation}
defined by the property that $\ell_2$ is the longest proper suffix of $\ell$ which is also a Lyndon word.
Then, $\ell_1$ is also a Lyndon word.
\end{proposition}

Henceforth, we denote the longest proper Lyndon suffix of $\ell$ as $\ell^r$ and the remaining prefix as $\ell^l$.
The factorization $\ell = \ell^l\ell^r$ is called the \textbf{costandard factorization}.

\begin{notation}
In the next section, we shall be using iterated superscripts of $l,r$ that are chained left-to-right, e.g.\
$\ell^{lr} = (\ell^l)^r$, $\ell^{lrl} = ((\ell^{l})^r)^{l}$, $\ell^{llrr} = (((\ell^l)^l)^r)^r$.
\end{notation}

We have an analogous factorization with the longest proper Lyndon prefix:

\begin{proposition}\label{prop:stand.factor}
Any Lyndon word $\ell$ with $|\ell|>1$ has a factorization
\begin{equation}\label{eqn:stand.factor}
   \ell = \ell_1\ell_2
\end{equation}
defined by the property that $\ell_1$ is the longest proper prefix of $\ell$ which is also a Lyndon word.
Then, $\ell_2$ is also a Lyndon word.
\end{proposition}

We shall denote such longest proper Lyndon prefix of $\ell$ by $\ell^{ls}$ and the remaining suffix by $\ell^{rs}$.
The factorization $\ell = \ell^{ls}\ell^{rs}$ is called the \textbf{standard factorization}. We note that our choice
of terminology is opposite to the one used in~\cite{L}.

Let $\LL$ be the set of all Lyndon words. Any word can be canonically built from~$\LL$:

\begin{proposition}\label{prop:canon.factor}
Any word $w\in I^*$ has a unique factorization
\begin{equation}\label{eqn:canon.factor}
  w = \ell_1 \dots \ell_k  \quad \text{with} \quad \ell_1 \geq \dots \geq \ell_k \in \LL.
\end{equation}
\end{proposition}

The factorization~\eqref{eqn:canon.factor} is called the \textbf{canonical factorization}.
The following result is well-known (see~\cite{M}):

\begin{lemma}\label{lem:Lyndon-vs-canonical}
If $\ell$ is Lyndon and $w\in I^*$ has the canonical factorization~\eqref{eqn:canon.factor}, then
  $$ \ell>w \ \ \textit{if and only if} \ \ \ell>\ell_1. $$
\end{lemma}

\begin{proof}
The ``only if'' direction is clear.

Assume now that $\ell>\ell_1$. If $\ell_1$ is not a prefix of $\ell$,
then $\ell>w$. If $\ell_1$ is a prefix of $\ell$, then $\ell=\ell_1 \ell^{(1)}$ with $\ell^{(1)}\ne \emptyset$.
As $\ell$ is Lyndon, we get $\ell^{(1)}>\ell>\ell_1\geq \ell_2$. This implies $\ell^{(1)}>\ell_2\dots \ell_k$ unless
$\ell_2$ is a prefix of $\ell^{(1)}$, i.e.\ $\ell^{(1)}=\ell_2\ell^{(2)}$. We note that $\ell^{(2)}\ne \emptyset$
as $\ell=\ell_1\ell_2$ would contradict the uniqueness of the canonical factorization.
Repeating this argument $k$ times we obtain $\ell>w$.
\end{proof}


\subsection{Bracketing}
\

Let $\fa$ be a Lie algebra generated by a finite set $\{e_i\}_{i \in I}$ labeled by the alphabet~$I$.

\begin{definition}\label{bracketing}
The \textbf{bracketing} of $\ell\in \LL$ is given inductively by:
\begin{itemize}[leftmargin=0.7cm]

\item[$\bullet$]
$\sb[i] = e_i \in \fa$ for $i \in I$,

\item[$\bullet$]
$\sb[\ell] = [\sb[\ell^l],\sb[\ell^r]] \in \fa$ if $|\ell|>1$.

\end{itemize}
\end{definition}

The major importance of this definition is due to the following result of Lyndon:

\begin{theorem}\label{thm:Lyndon.theorem}(\cite[Theorem 5.3.1]{Lo})
If $\fa$ is a free Lie algebra in the generators $\{e_i\}_{i\in I}$, then the set
$\big\{\sb[\ell] \,|\, \ell\in \LL \big\}$ provides a basis of $\fa$.
\end{theorem}


\subsection{Standard Lyndon words}\label{ssec:SL-words}
\

A generalization of Theorem~\ref{thm:Lyndon.theorem} to Lie algebras $\fa$ generated by $\{e_i\}_{i\in I}$
was provided in~\cite{LR}. To state the result, define ${_we}\in U(\fa)$ for any word $w\in I^*$:
\begin{equation}\label{eqn:word}
  _{[i_1 \dots i_k]}e = e_{i_1} \dots e_{i_k} \in U(\fa).
\end{equation}
Consider the following new order on $I^*$:
\begin{equation}\label{eq:LR-order}
    v\succeq w \quad \text{if }
    \begin{cases}
        |v|<|w| , \\
          \  \text{or}\\
        |v|=|w| \text{ and } v \geq w .
    \end{cases}
\end{equation}
The following definition is due to \cite{LR}:

\begin{definition}\label{def:standard}
(a) A word $w$ is called \textbf{standard} if $_we\in U(\fa)$ cannot be expressed as a linear combination of
$_ve$ for various $v \succ w$, with $_we$ as in~\eqref{eqn:word}.

\medskip
\noindent
(b) A Lyndon word $\ell$ is called \textbf{standard Lyndon} if $\sb[\ell]\in \fa$ cannot be expressed as
a linear combination of $\sb[v]$ for various Lyndon words $v \succ \ell$.
\end{definition}

The following result is nontrivial and justifies the above terminology:

\begin{proposition}\label{prop:standard}(\cite{LR})
A Lyndon word is standard iff it is standard Lyndon.
\end{proposition}

We shall use $\SL$ to denote the set of all standard Lyndon words. The major importance of this definition
is due to the following result of Lalonde-Ram:

\begin{theorem}\label{thm:standard Lyndon theorem}(\cite[Theorem 2.1]{LR})
For any Lie algebra $\fa$ generated by a finite collection $\{e_i\}_{i\in I}$, the set
$\big\{\sb[\ell] \,|\, \ell\in \SL \big\}$ provides a basis of $\fa$.
\end{theorem}


\subsection{Application to simple Lie algebras}\label{ssec:LR-bijection}
\

Let $\fg$ be a simple Lie algebra with a root system $\Delta=\Delta^+ \sqcup \Delta^-$ and simple roots
$\{\alpha_i\}_{i\in I}$. We endow the root lattice $Q=\bigoplus_{i\in I} \BZ\alpha_i$ with the symmetric
pairing $(\cdot,\cdot)$ so that the Cartan matrix $(a_{ij})_{i,j\in I}$ of $\fg$ is given by
$a_{ij} = 2(\alpha_i,\alpha_j)/(\alpha_i,\alpha_i)$. The Lie algebra $\fg$ admits the standard
\textbf{root space decomposition}:
\begin{equation}\label{eq:root.decomp}
  \fg=\fh \oplus \bigoplus_{\alpha \in \Delta} \fg_{\alpha},
  \qquad \fh\subset\fg -\mathrm{Cartan\ subalgebra},
\end{equation}
with $\dim(\fg_{\alpha})=1$ for all $\alpha\in \Delta$.
We pick root vectors $e_\alpha\in \fg_\alpha$ so that $\fg_\alpha=\BC\cdot e_\alpha$.

Consider the \emph{positive} Lie subalgebra $\fn^+=\bigoplus_{\alpha \in \Delta^+} \fg_{\alpha}$ of $\fg$. Explicitly,
$\fn^+$ is generated by $\{e_i\}_{i\in I}$ (where $e_i=e_{\alpha_i}$) subject to the classical Serre relations:
\begin{equation}\label{eqn:Serre}
  \underbrace{[e_i,[e_i,\cdots,[e_i,e_j]\cdots]]}_{1-a_{ij} \text{ Lie brackets}}\, =\, 0 \qquad \forall\ i\neq j.
\end{equation}
Let $Q^+=\bigoplus_{i\in I} \BZ_{\geq 0}\alpha_i$. The Lie algebra $\fn^+$ is naturally $Q^+$-graded via $\deg\, e_i=\alpha_i$.

Fix any order on the set $I$. According to Theorem~\ref{thm:standard Lyndon theorem}, $\fn^+$ has a basis consisting
of $\{\sb[\ell] \,|\, \ell\in \SL\}$. Invoking the above $Q^+$-grading of $\fn^+$, it is natural to define the grading
of words via $\deg [i_1 \dots i_k] = \alpha_{i_1} + \dots + \alpha_{i_k} \in Q^+$. Due to the
decomposition~\eqref{eq:root.decomp} and the fact that the root vectors $\{e_\alpha\}_{\alpha\in \Delta^+} \subset \fn^+$
all live in distinct degrees $\alpha \in Q^+$, we conclude that there exists a \textbf{Lalonde-Ram bijection}~\cite{LR}:
\begin{equation}\label{eqn:associated word}
  \ell \colon \Delta^+ \,\iso\, \big\{\slaws \big\}
  \quad \mathrm{with} \quad \deg \ell(\alpha) = \alpha.
\end{equation}
For degree reasons, we also note that one can presently replace $v \succ w$ of~\eqref{eq:LR-order} simply with $v>w$
subject to $\deg v=\deg w$ (so that $|v|=|w|$) in Definition~\ref{def:standard}.


\subsection{Results of Leclerc and Rosso}\label{ssec:L-and-R}
\

The Lalonde-Ram's bijection~\eqref{eqn:associated word} was described explicitly in~\cite{L}. We recall that for a root
$\gamma=\sum_{i\in I} n_i\alpha_i\in \Delta^+$, its height is $|\gamma|=\hgt(\gamma)=\sum_{i\in I} n_i\in \BZ_{>0}$.

\begin{proposition}\label{prop:Leclerc algorithm}(\cite[Proposition 25]{L})
The bijection $\ell$ is inductively given by:
\begin{itemize}[leftmargin=0.7cm]

\item[$\bullet$]
for simple roots $\ell(\alpha_i)=[i]$,

\item[$\bullet$]
for other positive roots, we have the following \textbf{Leclerc algorithm}:
\begin{equation}\label{eqn:inductively}
  \ell(\alpha) =
  \max\left\{ \ell(\gamma_1)\ell(\gamma_2) \,\Big|\,
               \alpha=\gamma_1+\gamma_2 \,,\, \gamma_1,\gamma_2\in \Delta^+ \,,\, \ell(\gamma_1) < \ell(\gamma_2) \right\}.
\end{equation}

\end{itemize}
\end{proposition}

\noindent
The formula~\eqref{eqn:inductively} recovers $\ell(\alpha)$ once we know $\ell(\gamma)$ for all
$\{\gamma\in \Delta^+ \,|\, \mathrm{ht}(\gamma)<\mathrm{ht}(\alpha)\}$.

Let us also recall another fundamental property of $\ell$.

\begin{definition}\label{def:convex}
A total order on the set $\Delta^+$ of positive roots is called \textbf{convex} if:
\begin{equation}\label{eqn:convex}
  \alpha < \alpha+\beta < \beta
\end{equation}
for all $\alpha < \beta \in \Delta^+$ such that $\alpha+\beta$ is also a root.
\end{definition}

The following result is~\cite[Proposition 26]{L} (where it is attributed to Rosso~\cite{R}):

\begin{proposition}\label{prop:fin.convex}(\cite{L,R})
Consider a total order $<$ on $\Delta^+$ induced from the lexicographic order on $\slaws$:
\begin{equation}\label{eqn:induces}
  \alpha < \beta \ \ \textit{if and only if} \ \ \ell(\alpha) < \ell(\beta)  \ \ \mathrm{ lexicographically}.
\end{equation}
This order is convex.
\end{proposition}


\subsection{Affine Lie algebras}\label{ssec:affineLie}
\

We now consider untwisted affine Kac-Moody algebras. Let $\fg$ be a simple finite dimensional Lie algebra,
$\{\alpha_i\}_{i\in I}$ be the simple roots, and $\theta\in \Delta^+$ be the highest root. We define
$\wI = I \sqcup \{0\}$. Consider the affine root lattice $\wQ=Q \times \BZ$ with the generators
$\{(\alpha_i,0)\}_{i\in I}$ and $\alpha_0:=(-\theta,1)$. We endow $\wQ$ with the symmetric pairing
\begin{equation*}
  \big((\alpha,n),(\beta,m)\big)=(\alpha,\beta) \qquad \forall\ \alpha,\beta\in Q \,,\, n,m \in \BZ.
\end{equation*}
This leads to the affine Cartan matrix $(a_{ij})_{i,j\in \wI}$ and the \textbf{affine Lie algebra} $\widehat{\fg}$.
The associated affine root system $\wDelta=\wDelta^+ \sqcup \wDelta^-$ has the following explicit description:
\begin{equation}\label{eq:affine-roots}
  \wDelta^+ = \big\{ \Delta^+ \times \BZ_{\geq 0} \big\}
  \sqcup \big\{ 0 \times \BZ_{>0} \big\}
  \sqcup \big\{ \Delta^- \times \BZ_{>0} \big\}.
\end{equation}
Here, $\delta=\alpha_0+\theta=(0,1) \in Q \times \BZ$ is the \textbf{minimal imaginary root} of the
affine root system $\wDelta$. With this notation, we have the following root space decomposition:
\begin{equation}\label{eq:aff.root.decomp}
  \widehat{\fg}=\widehat{\fh} \oplus \bigoplus_{\alpha \in \wDelta} \widehat{\fg}_{\alpha} ,
  \qquad \widehat{\fh}\subset \widehat{\fg} -\mathrm{Cartan\ subalgebra}.
\end{equation}

Let us now recall another realization of $\widehat{\fg}$. To this end, consider the Lie algebra
\begin{equation}\label{eq:loops}
\begin{split}
  & \widetilde{\fg}=\fg\otimes \BC[t,t^{-1}]\oplus \BC\cdot \mathsf{c}
    \quad \mathrm{with\ a\ Lie\ bracket\ given\ by}\\
  & [x\otimes t^n,y\otimes t^m]=[x,y]\otimes t^{n+m}+n\delta_{n,-m}(x,y)\cdot \mathsf{c}
    \quad \mathrm{and} \quad [\mathsf{c},x\otimes t^n]=0,
\end{split}
\end{equation}
where $x,y\in \fg$, $m,n\in \BZ$, and $(\cdot,\cdot)\colon \fg\times\fg\to \BC$ is
a non-degenerate invariant pairing.

The rich theory of affine Lie algebras is mainly based on the following key result:

\begin{claim}
There exists a Lie algebra isomorphism:
\begin{equation*}
  \widehat{\fg}\ \iso\ \widetilde{\fg}
\end{equation*}
determined on the generators by the following formulas:
\begin{align*}
  & e_i \mapsto e_i \otimes t^0, & & f_i \mapsto f_i \otimes t^0, & & h_i \mapsto h_i \otimes t^0, \quad \mathrm{for\ all} \ i\in I \\
  & e_0 \mapsto e_{-\theta} \otimes t^1, & & f_0 \mapsto e_\theta \otimes t^{-1}, &
  & h_0 \mapsto [e_{-\theta},e_\theta]\otimes t^0 + (e_{-\theta},e_{\theta}) \mathsf{c}.
\end{align*}
\end{claim}

In view of this result, we can explicitly describe the root subspaces from~\eqref{eq:aff.root.decomp}:
\begin{align*}
  & \widehat{\fg}_{(\alpha,k)}=\fg_\alpha\otimes t^k \quad \forall\
    (\alpha,k)\in \wDelta^{+,\re}:=\big\{\Delta^+ \times \BZ_{\geq 0} \big\} \sqcup \big\{ \Delta^- \times \BZ_{>0} \big\},\\
  & \widehat{\fg}_{k\delta}=\fh \otimes t^k \quad \mathrm{for} \
    k\delta\in \wDelta^{+,\im}:=\big\{ 0 \times \BZ_{>0} \big\}.
\end{align*}
As $\dim(\fg_\alpha)=1$ for any $\alpha\in \Delta$ and $\dim(\fh)=\rk(\fg)=|I|$, we thus obtain:
\begin{equation}\label{eq:aff-dim}
  \dim(\widehat{\fg}_\alpha)=1 \quad \forall\ \alpha\in \wDelta^{+,\re} \,, \qquad
  \dim(\widehat{\fg}_\alpha)=|I| \quad \forall\ \alpha\in \wDelta^{+,\im}.
\end{equation}

In what follows, we shall always write $xt^n$ instead of $x\otimes t^n$.


\subsection{Affine standard Lyndon words}\label{ssec:aslaws}
\

Replacing $\fg$ of Subsection~\ref{ssec:LR-bijection} with $\widehat{\fg}$ of Subsection~\ref{ssec:affineLie}, we likewise
consider only the positive subalgebra $\widehat{\fn}^+=\bigoplus_{\alpha\in \wDelta^+} \widehat{\fg}_{\alpha}$,
which is generated by $\{e_i\}_{i\in \wI}$ subject to the Serre relations~\eqref{eqn:Serre} for $i\ne j\in \wI$. Endowing
$\wI$ with any order allows us to introduce Lyndon and standard Lyndon words (with respect to $\widehat{\fn}^+$).
Henceforth, we shall often use the term \textbf{$\aslaws$} in the present setup.

The key difference is that some root subspaces are higher dimensional, see~\eqref{eq:aff-dim}. Thus, we no longer
have a bijection~\eqref{eqn:associated word}. However, a similar degree reasoning implies that there is a unique
$\aslaw$ in each real degree $\alpha\in \wDelta^{+,\re}$, denoted by $\SL(\alpha)$, and $|I|$ $\aslaws$ in each
imaginary degree $\alpha\in \wDelta^{+,\im}$, denoted by $\SL_1(\alpha)>\dots>\SL_{|I|}(\alpha)$. These words can
be computed through the following \textbf{generalized Leclerc algorithm}:

\begin{proposition}\label{prop:generalized.Leclerc.algo} (\cite[Proposition 3.4]{AT})
The $\aslaws$ (with respect to $\widehat{\fn}^+$) are determined inductively by the following rules:

\medskip
\noindent
(a) For simple roots, we have $\SL(\alpha_i)=[i]$. For other real $\alpha\in \wDelta^{+,\re}$, we have:
\begin{equation}\label{eq:generalized Leclerc}
  \SL(\alpha) =
  \max\left\{\SL_*(\gamma_1)\SL_*(\gamma_2) \,\Big|\,
   \substack{\alpha=\gamma_1+\gamma_2,\, \gamma_1,\gamma_2\in \wDelta^+\\ \SL_*(\gamma_1)<\SL_*(\gamma_2)\\
             [\sb[\SL_*(\gamma_1)],\sb[\SL_*(\gamma_2)]]\neq 0} \right\},
\end{equation}
where $\SL_*(\gamma)$ denotes $\SL(\gamma)$ for $\gamma\in \wDelta^{+,\re}$
and any of $\{\SL_k(\gamma)\}_{k=1}^{|I|}$ for $\gamma\in \wDelta^{+,\im}$.

\medskip
\noindent
(b) For imaginary $\alpha\in \wDelta^{+,\im}$, the corresponding $|I|$ $\aslaws$ $\{\SL_k(\alpha)\}_{k=1}^{|I|}$ are
the $|I|$ lexicographically largest words from the list as in the right-hand side of~\eqref{eq:generalized Leclerc}
whose bracketings are linearly independent.
\end{proposition}

We shall call $u=\SL_{*}(\gamma)$ \emph{real} (resp.\ \emph{imaginary}) if $\gamma\in \wDelta^{+,\re}$
(resp.\ $\gamma\in \wDelta^{+,\im}$).

\begin{remark}\label{rem:bracketing-irrelevance}
We note that the condition $[\sb[\SL_*(\gamma_1)],\sb[\SL_*(\gamma_2)]]\neq 0$ implies that $\gamma_1,\gamma_2$ must
be real in (b). Moreover, this condition always holds if $\alpha,\gamma_1,\gamma_2$ are real.
\end{remark}

We conclude this section with the notation that will be used through this note:

\begin{notation}\label{notation:h}
For any $\alpha  = (\alpha',k)\in \wDelta^{+,\re}$, we set $h_\alpha:=[e_{\alpha'},e_{-\alpha'}] \in \fh$.
For any $w\in \wI^*$ with $\deg\, w=\alpha\in \wDelta^{+,\re}$, we set $h_w=h_\alpha$. We note that
$h_\alpha,h_w$ depend on the choice of root vectors and thus are defined up to nonzero constants.
To address this ambiguity, we shall write $x\sim y$ if $x=cy$ for some $c\in \BC\backslash\{0\}$.
\end{notation}


\section{Properties of factorizations}\label{sec:Aux}

In this section, we establish some properties of the standard, costandard, and other factorizations of standard Lyndon words.
The following result is well-known:

\begin{lemma}\label{lemma:right.costfac.minimal}
For a Lyndon word $w$, the smallest proper suffix is $w^r$.
\end{lemma}

\begin{proof}
Denote the smallest proper suffix of $w$ as $u$, that is, $u = \min\{w_{|a}\}_{0<a<|w|}$. First, we note that $u$ is Lyndon
as any proper suffix of $u$ is lexicographically larger than $u$. Thus $u$ is a suffix of $w^r$. But if $u$ was a proper
suffix of $w^r$, then $w^r < u$ as $w^r$ is Lyndon, yielding a contradiction with the minimality of $u$. Hence $u = w^r$.
\end{proof}

\begin{lemma}\label{lemma:lr.geq.r}
For any Lyndon word $\ell$ with $|\ell|,|\ell^{l}| > 1$, we have $\ell^{lr} \geq \ell^r$.
\end{lemma}

\begin{proof}
If $\ell^{lr} < \ell^r$, then $\ell^{lr}\ell^r$ would be Lyndon by Lemma~\ref{lemma:lyndon}. The latter implies
$\ell^{lr}\ell^r < \ell^r$, a contradiction with Lemma~\ref{lemma:right.costfac.minimal}. This proves $\ell^{lr} \geq \ell^{r}$.
\end{proof}

We note that the above lemmas admit the following ``prefix'' counterpart:

\begin{lemma}\label{lemma:left.standfac.maximal}
(a) For any Lyndon word $\ell$ with $|\ell|,|\ell^{rs}| > 1$, we have $(\ell^{rs})^{ls} \leq \ell^{ls}$.

\medskip
\noindent
(b) For a Lyndon word $w$ with $|w| > 1$, the biggest proper Lyndon prefix is $w^{ls}$.
\end{lemma}

\begin{proof}
Part (b) is obvious from the definition of the standard factorization.
As per part (a), if $(\ell^{rs})^{ls} > \ell^{ls}$, then $\ell^{ls}(\ell^{rs})^{ls}$ would be
Lyndon by Lemma~\ref{lemma:lyndon}, thus contradicting part (b) as $\ell^{ls}<\ell^{ls}(\ell^{rs})^{ls}$.
This proves $(\ell^{rs})^{ls} \leq \ell^{ls}$.
\end{proof}

Our next several results relate any factorization to the costandard one:

\begin{lemma}\label{lemma:seq.to.costfac}
Consider any factorization of a Lyndon word $\ell = \ell_1\ell_2$ with $\ell_1,\ell_2\in \LL$. Then, $\ell^r$ belongs
to the set $P:= \{\ell_2,\ell_1^r\ell_2,\ell_1^{lr}\ell_1^r\ell_2,\ell_1^{llr}\ell_1^{lr}\ell_1^r\ell_2, \dots\}$.
\end{lemma}

\begin{proof}
Suppose that $\ell^r \not \in P$ and let $u \in P$ be the longest suffix of $\ell^r$ (such $u$ exists as $\ell_2$ is a
suffix of $\ell^r$), so that $\ell^r = vu$. But there exists a Lyndon word $w$ of the form $\ell^{l\ldots lr}$ that has
$v$ as a proper suffix and $wu \in P$. Thus, we have  $w < v$ and so $wu < vu = \ell^{r}$, which contradicts to
Lemma~\ref{lemma:right.costfac.minimal}. Therefore, $\ell^r\in P$.
\end{proof}

\begin{lemma}\label{lemma:lyndon.seq.lyndon}
For any factorization of a Lyndon word $\ell = \ell_1\ell_2$ with $\ell_1,\ell_2\in \LL$, every element of the set
$\bar{P} = \{\ell_2,\ell_1^r\ell_2,\ell_1^{lr}\ell_1^r\ell_2,\ell_1^{llr}\ell_1^{lr}\ell_1^r\ell_2, \ldots,\ell^r\}$
is Lyndon.
\end{lemma}

\begin{proof}
We prove this by induction on the length. The base case $\ell_2$ is obvious.

For the step of induction, fix $u \in \bar{P}$ and assume that the induction hypothesis holds for any $v \in \bar{P}$
with $|v| < |u|$. Suppose that $u$ is not Lyndon.  Split $u$ into $u = u_1u_2$ with $u_2 \in \bar{P}$ and
$u_1 = \ell_1^{l\cdots lr}$, and let $k$ denote the number of $l$'s in the above superscript. Similarly,
we split $\ell^r = v_1v_2$ with $v_2 \in \bar{P}$ and $v_1 = \ell_1^{l\cdots lr}$, and let $p$ be the number
of $l$'s in the above superscript, so that $p > k$ (the case $p=k$ is impossible as $u\notin \LL$).
As $u_1$ is Lyndon and $u_2$ is Lyndon by the induction hypothesis, we note that $u_1 \geq u_2$,
as otherwise $u_1u_2$ would be Lyndon by Lemma~\ref{lemma:lyndon}, a contradiction.
Also $u_1 \leq v_1$, due to a repeated application of Lemma~\ref{lemma:lr.geq.r}. Hence, we obtain:
  $$ u_2 \leq u_1 \leq v_1 < v_1v_2=\ell^r. $$
Thus, $\ell^r$ is larger than its proper suffix $u_2$, a contradiction with $\ell^r$ being Lyndon.
Therefore $u$ is a Lyndon word, which completes the step of induction.
\end{proof}

\begin{corollary}\label{cor:three.way.lyndon}
For any factorization of a Lyndon word $\ell = \ell_1\ell_2$ with $\ell_1,\ell_2\in \LL$ and $\ell_2 \neq \ell^r$,
there exists a factorization $\ell_1 = uv$ with $u,v\in \LL$ such that $v\ell_2\in \LL$.
Moreover, one can choose it to be the costandard factorization: $u=\ell_1^l$ and $v=\ell_1^r$.
\end{corollary}

\begin{proof}
Since $\ell^r \neq \ell_2$, we have $\ell_1^r\ell_2\in \bar{P}$ and hence it is a Lyndon word by
Lemma~\ref{lemma:lyndon.seq.lyndon}. Therefore, Lyndon words $u = \ell_1^l$ and $v = \ell_1^r$ satisfy both conditions.
\end{proof}

\begin{corollary}\label{cor:rotate}
For any factorization of a Lyndon word $\ell = \ell_1\ell_2$ with $\ell_1,\ell_2\in \LL$ and $\ell_2 \neq \ell^r$,
there exists another factorization $\ell=uv\ell_2$ such that $u\ell_2v,vu\ell_2>\ell$ and $u\ell_2$ is Lyndon.
Moreover, one can choose it to be the costandard factorization: $u=\ell_1^l$ and $v=\ell_1^r$.
\end{corollary}

\begin{proof}
By Corollary~\ref{cor:three.way.lyndon}, there is a splitting $\ell=uv\ell_2$ with $u,v,uv,v\ell_2\in \LL$.
Thus $u<v<\ell_2$, and so $u\ell_2$ is Lyndon by Lemma~\ref{lemma:lyndon}. As $v\ell_2\in \LL$, we have $v\ell_2 < \ell_2v$,
hence, $\ell=uv\ell_2 < u\ell_2v$. As $uv\in \LL$, we have $uv < vu$, so that $\ell=uv\ell_2 < vu\ell_2$.
\end{proof}

Let us record another interesting application of Lemma~\ref{lemma:lyndon.seq.lyndon}.

\begin{corollary}
For a factorization of a Lyndon word $\ell = \ell_1\ell_2$ with $\ell_1,\ell_2\in \LL$:
\begin{equation*}
   \ell_1^{r} \geq \ell_2 \ \ \textit{if and only if} \ \ \ell_2 = \ell^r.
\end{equation*}
\end{corollary}

\begin{proof}
The ``if'' direction follows from Lemma~\ref{lemma:lr.geq.r}. For the ``only if'' direction, we note that
$\ell_1^r\ell_2\notin \LL$ as $\ell_1^{r} \geq \ell_2$. Invoking Lemma~\ref{lemma:lyndon.seq.lyndon}, we thus
get $\bar{P}=\{\ell_2\}$ so that $\ell_2 = \ell^r$.
\end{proof}

We also note that Lemmas~\ref{lemma:seq.to.costfac}--\ref{lemma:lyndon.seq.lyndon} admit ``prefix'' counterparts:

\begin{lemma}\label{lem:prefix-aux-1}
Consider any factorization of a Lyndon word $\ell = \ell_1\ell_2$ with $\ell_1,\ell_2 \in \LL$.
Define $u_i$ and $v_i$ inductively via $u_1 = \ell_1, v_1 = \ell_2$ and $u_i = u_{i-1}v_{i-1}^{ls},v_i = v_{i-1}^{rs}$,
as long as $|v_{i-1}|>1$. Then, every element in $\{u_1,u_2,\ldots,u_n\}$ is Lyndon, where $n$ is the smallest integer
such that $v^{ls}_{n} \leq u_n$ or $|v_n|=1$.
\end{lemma}

\begin{proof}
We prove that $u_k$ is Lyndon for $1 \leq k \leq n$ by induction on $k$. The base case $k=1$ is obvious. For the inductive
step, if $u_{k-1}$ is Lyndon and $u_k = u_{k-1}v_{k-1}^{ls}$ with $v_{k-1}^{ls} > u_{k-1}$, then $u_k$ is also Lyndon by
Lemma~\ref{lemma:lyndon}.
\end{proof}

\begin{lemma}
Consider any factorization $\ell = \ell_1\ell_2$ with $\ell,\ell_1,\ell_2 \in \LL$.
Then, $\ell^{ls}$ is equal to $u_n$ as defined in Lemma~\ref{lem:prefix-aux-1}.
\end{lemma}

\begin{proof}
Assume the contrary: $\ell^{ls} \neq u_n$. Then $u_n$ is a proper prefix of $\ell^{ls}$, so that
$\ell^{ls} = u_n w$ with $w\ne \emptyset$. Consider the canonical factorization $w=w_1 \ldots w_k$, cf.~\eqref{eqn:canon.factor}.
Due to Lemma~\ref{lemma:seq.right.word.Lyndon}, the word $u_nw_1$ is Lyndon, so that $u_n<w_1$.
On the other hand, we have $w_1\leq v^{ls}_{n}$ by Lemma~\ref{lemma:left.standfac.maximal}(b),
which contradicts to our choice of $n$.
\end{proof}

The next two results will be needed later:

\begin{lemma}\label{lemma:inc.seq.right.word.lyndon}
For $u \in \LL$, consider any splitting $u=vw$ with $v \in \LL$, and let $w =w_1w_2\ldots w_N$ be
the canonical factorization. Then $w_1\geq w_2\geq \cdots \geq w_N > v$.
\end{lemma}

\begin{proof}
As $u=vw$ is Lyndon, we have $w_N > vw$, which implies $w_N > v$.
\end{proof}

\begin{lemma}\label{lemma:seq.right.word.Lyndon}
For $u \in \LL$, consider any splitting $u=vw$ with $v\in \LL$, and let $w = w_1w_2 \ldots w_N$ be the canonical factorization.
Then $vw_1$,$vw_1w_2, \ldots ,vw_1\ldots w_N\in \LL$.
\end{lemma}

\begin{proof}
We prove that $vw_1\ldots w_n$ is Lyndon for all $1\leq n\leq N$ by induction on $n$.
For the base case $n=1$, it suffices to show that $v<w_1$, due to Lemma~\ref{lemma:lyndon}.
If not, then $w_N \leq w_1\leq v < vw=u$, contradicting the condition that $u$ is Lyndon.
For the step of induction, assume that $vw_1\ldots w_{n-1} \in \LL$. Then applying the base case to
$v'=vw_1 \ldots w_{n-1}$, $w'=w_nw_{n+1}\ldots w_N$, we get $vw_1\ldots w_{n}=v'w_n \in \LL$.
\end{proof}

The above result admits a natural ``prefix'' counterpart:

\begin{lemma}\label{lemma:seq.left.word.Lyndon}
For $u\in \LL$, consider any splitting $u=vw$ with $w\in \LL$ and let $v = v_1\ldots v_N$ be the canonical factorization.
Then $v_Nw,v_{N-1}v_Nw,\dots,v_1\ldots v_Nw\in \LL$.
\end{lemma}

\begin{proof}
We prove that $v_{N-n+1}\ldots v_Nw \in \LL$ for $1\leq n\leq N$ by induction on $n$. For the base case $n=1$,
it suffices to show that $v_N < w$, due to Lemma~\ref{lemma:lyndon}. If not, then $v_1\geq \ldots \geq v_N\geq w$,
hence a contradiction with $u\in \LL$, due to the uniqueness of the canonical factorization. For the step of induction,
assume that $v_{N-n+2}\ldots v_Nw \in \LL$. Applying the base case to $v'=v_1v_2\ldots v_{N-n+1}$ and $w'=v_{N-n+2}\ldots v_Nw$,
we get $v_{N-n+1}\ldots v_Nw=v_{N-n+1}w'\in \LL$.
\end{proof}

We conclude this section with the following two results on cyclic permutations:

\begin{lemma}\label{lemma:word.to.Lyndon}
For any word $w\in I^*$, consider its canonical factorization~\eqref{eqn:canon.factor}: $w = w_1w_2\ldots w_n$.
If there is a cyclic permutation of $w$ that is Lyndon, then it must be of the form
$w_iw_{i+1} \ldots w_n w_{1} \ldots w_{i-1}$ for some $i \in \{1,2,\ldots,n\}$.
\end{lemma}

\begin{proof}
Let $\ell$ be a cyclic permutation of $w$ that is Lyndon (unique if it exists). If $\ell$ is not of the stated form,
then there is some $w_i = uv$ such that $v$ is a proper prefix of $\ell$ and $u$ is a proper suffix.
But then $\ell > v > w_i > u$, contradicting $\ell\in \LL$.
\end{proof}

\begin{corollary}\label{cor:word.to.Lyndon}
For any word $w\in I^*$, consider its canonical factorization~\eqref{eqn:canon.factor}: $w = w_1w_2\ldots w_n$.
If $\ell\in \LL$ is a cyclic permutation of $w$, then there is $i$ such that
\begin{equation*}
  \underbrace{w}_{k\text{ times}} = w_1w_2\ldots w_{i-1}\underbrace{\ell}_{k -1 \text{ times}}w_i w_{i+1}\ldots w_n
  \qquad \textit{for all} \ \ k\in \BZ_{>0}.
\end{equation*}
\end{corollary}


\section{Complete flags and auxiliary sets}\label{sec:flag}

In this section, we investigate the behavior of the bracketings with respect to different splittings of words.
For imaginary roots, it will be crucially important to consider not individual bracketings $\sb[\SL_i(k\delta)]$
but rather the induced complete flags:
\begin{equation}\label{eq:flag}
\begin{split}
  & 0=\spanset_0^k \subset \spanset_1^k \subset \cdots \subset\spanset_{|I|}^k = \fh t^k \qquad \forall\, k\in \BZ_{>0},\\
  & \mathrm{with}\quad  \spanset_i^k := \spn\big\{\sb[\SL_1(k\delta)],\ldots,\sb[\SL_i(k\delta)]\big\}.
\end{split}
\end{equation}


\subsection{Extended set of roots}
\

Following~\cite[(5.1)]{AT}, let us consider the following upgrade of~\eqref{eq:affine-roots}:
\begin{equation}\label{eq:extended-affine-roots}
  \wDelta^{+,\ext} = \wDelta^{+,\re} \cup \imx  \quad \mathrm{with} \quad
  \imx=\big\{(k\delta,r) \,|\, k \geq 1, 1 \leq r \leq |I| \big\},
\end{equation}
counting imaginary roots with multiplicities. We can thus naturally generalize~\eqref{eqn:associated word}:
\begin{equation}\label{eqn:associated word affine}
  \SL \colon \wDelta^{+,\ext} \,\iso\, \big\{\aslaws \big\},
  \quad \SL((k\delta,r)) = \SL_r(k\delta).
\end{equation}
We also consider the induced order on $\wDelta^{+,\ext}$, in analogy with~\eqref{eqn:induces}:
\begin{equation}\label{eq:order.ext}
    \alpha < \beta \ \ \mathrm{if\ and\ only\ if} \ \ \SL(\alpha) < \SL(\beta) \qquad \forall\, \alpha,\beta \in \wDelta^{+,\ext}.
\end{equation}

\begin{lemma}\label{lemma:equiv.to.standard.fac}
Let $w$ be a Lyndon word and $uv$ be the standard factorization of a Lyndon word, with $|w| < |uv|$.
Then $w > u$ if and only if $w > uv$.
\end{lemma}

\begin{proof}
We prove the ``only if'' direction by induction on the length of $|uv|$. The base case $|uv| = 2$ is clear.
As per the step of induction, let us assume the contrary, that is $u < w < uv$. Then $w = uy$ with $y=y_1\ldots y_N$
being a canonical factorization.
Due to Lemma~\ref{lemma:inc.seq.right.word.lyndon}, the Lyndon word $y_1$ is $>u$ and has length less than $|v|$.
Let $v=zt$ be the standard factorization of $v$. We note that $z\leq u$ by Lemma~\ref{lemma:left.standfac.maximal}(a).
Thus, $z\leq u<y_1\leq y<v=zt$,
which cannot occur due to the induction assumption (applied to $y_1$ in place of $w$ and $zt$ in place of $uv$).
This yields a contradiction, thus establishing the step of induction.

The ``if'' direction is a consequence of the inequalities $w > uv > u$.
\end{proof}


\subsection{$W$-set and pseudo-bracketing}
\

The key difficulty in extending the convexity~\eqref{eqn:convex} to affine root systems $\wDelta^{+}$ lies
in the treatment of imaginary roots. For example, while~\eqref{eq:generalized Leclerc} guarantees that
  $$\alpha + \beta > \min\{\alpha,\beta\}$$
if $\alpha,\beta,\alpha +\beta \in \wDelta^{+,\re}$ (cf.~Remark~\ref{rem:bracketing-irrelevance}),
the generalization of this to the case when some roots are imaginary is not obvious, cf.~Corollary~\ref{cor:im.upper.limit}
and Remark~\ref{remark:convex.im.set}. To this end, we start with the following definition:

\begin{definition}\label{def:lyndon.set}
(a) Define the set $W$ as follows:
\begin{equation*}
  W = \big\{w=(u,v) \,\big|\, u,v \in \SL, u < v \big\},
\end{equation*}
whereas we often write $w_1 = u,w_2 = v$.
We endow $W$ with the following order:
\begin{equation}\label{eq:W-order}
  (u,v) < (u',v') \ \ \textit{if} \
  \begin{cases}
      |uv| < |u'v'| , \\
       \ \text{or}\\
      |uv| = |u'v'| \text{ and } uv>u'v' , \\
       \ \text{or}\\
      uv=u'v' \text{ and } u <u' .
  \end{cases}
\end{equation}
Finally, for any $w \in W$, we define its \textbf{pseudo-bracketing} $\osb[w]\in \widehat{\fn}^+$ via:
  $$ \osb[w] = [\sb[w_1],\sb[w_2]]. $$

\noindent
(b) For any $\alpha \in \wDelta^{+}$, define the subset $W_\alpha$ of $W$ via:
\begin{equation*}
  W_\alpha = \big\{(u,v) \,\big|\, u,v \in \SL, u < v , \deg(u) + \deg(v) = \alpha \big\}.
\end{equation*}

\noindent
(c) Define the subset $\ol{W}$ of $W$ as follows:
\begin{equation*}
  \ol{W} = \big\{ (u,v) \,\big|\, u,v \in \SL, u < v, uv \in \SL \big\}.
\end{equation*}

\noindent
(d) For any $\alpha \in \wDelta^{+,\ext}$, we define the subset $\ol{W}_\alpha$ of $\ol{W}$ via:
\begin{equation*}
  \ol{W}_\alpha = \big\{ (u,v) \,\big|\, u,v \in \SL, u<v ,  uv=\SL(\alpha) \big\}.
\end{equation*}
\end{definition}

\begin{remark}\label{rem:W-minmax}
For any standard Lyndon word $w$, the set $\ol{W}$ contains the costandard factorization $(w^l,w^r)$, the standard
factorization $(w^{ls},w^{rs})$, as well as possibly some more $(\ell_1,\ell_2)$ arising from factorizations
$w=\ell_1\ell_2$ into two standard Lyndon words. Moreover, $(w^l,w^r)$ is the smallest and $(w^{ls},w^{rs})$
is the biggest among all of those.
\end{remark}

\begin{proposition}\label{prop:im.bracketing}
Let $w_1,\dots,w_N$ be the elements of $W_{k\delta}$ listed in increasing order. Then for any $1\leq n\leq N$, we have:
\begin{equation}\label{eq:Wspan-vs-flag}
  \spn\Big\{\osb[w_1],\osb[w_2],\ldots,\osb[w_n]\Big\} = \spanset_m^k,
\end{equation}
where $m = \max\{i \,|\, (w_n)_1(w_n)_2 \leq \SL_i(k\delta)\}$ and $m = 0$ if $(w_n)_1(w_n)_2 > \SL_1(k\delta)$.
\end{proposition}

\begin{proof}
We prove~\eqref{eq:Wspan-vs-flag} by induction on $n$.

For the base case $n=1$, either $w_1 = (\SL^l_1(k\delta),\SL^r_1(k\delta))$ or $(w_1)_1(w_1)_2 >\SL_1(k\delta)$.
In the first case, we have $\osb[w_1] = \sb[\SL_1(k\delta)]$ and $m=1$. In the second case, $m=0$ while $\osb[w_1] = 0$
as $w_1$ represents the costandard factorization of a non-standard Lyndon word, cf.~Remark~\ref{rem:W-minmax}.
Thus, equality~\eqref{eq:Wspan-vs-flag} holds for $n=1$.

For the inductive step, we shall assume that~\eqref{eq:Wspan-vs-flag} holds for $n'=n-1$ with the right-hand side
$\spanset_{m'}^k$. We shall now consider three cases:

$\bullet$
If $w_n=(\SL^l_i(k\delta),\SL^r_i(k\delta))$, then $m=i$ and $\osb[w_n] = \sb[\SL_i(k\delta)]$, while the inductive
hypothesis yields $\spn\{\osb[w_1],\osb[w_2],\ldots,\osb[w_{n-1}]\} = \spanset_{i-1}^k$. This implies
that $\spn\{\osb[w_1],\osb[w_2],\ldots,\osb[w_{n}]\} = \spanset_{i}^k$, as claimed.

$\bullet$
If $w_n$ represents the costandard factorization of a non-standard Lyndon word, then
$\osb[w_n]=\sb[(w_n)_1(w_n)_2]\in \spn\{\osb[w_1],\osb[w_2],\ldots,\osb[w_{n-1}]\}$. It thus remains
to show that $m=m'$, that is, $\SL_{m'+1}(k\delta) < (w_n)_1(w_n)_2 < \SL_{m'}(k\delta)$.
But if not, then we would actually have $\SL_{m'+1}(k\delta) = (w_n)_1(w_n)_2$, due to the order~\eqref{eq:W-order}
and $\SL_{m'+1}(k\delta) < (w_{n-1})_1(w_{n-1})_2 \leq \SL_{m'}(k\delta)$, thus contradicting $(w_n)_1(w_n)_2\notin \SL$.

$\bullet$
If $w_n$ does not represent the costandard factorization of any Lyndon word, then $(w_n)_1(w_n)_2 = (w_{n-1})_1(w_{n-1})_2$,
as the costandard factorization of $(w_n)_1(w_n)_2$ is among $\{w_j\}_{j=1}^{n-1}$. Therefore, we have $m=m'$. It thus
remains to show:
\begin{equation}\label{eq:aux-n}
  \osb[w_n]\in \spn\big\{\osb[w_1],\osb[w_2],\ldots,\osb[w_{n-1}]\big\}.
\end{equation}
Using Corollary~\ref{cor:three.way.lyndon}, let us split $(w_n)_1(w_n)_2 = uv(w_n)_2$ with
$u=(w_n)^l_1, v=(w_n)^r_1$, so that $u,v,v(w_n)_2 \in \LL$. We note that $\sb[(w_n)_1]=[\sb[u],\sb[v]]$.
For the notation simplicity, let $y = (w_n)_2$, so that $u<v<y$. By the Jacobi identity, we have:
\begin{gather*}
  [\sb[u],[\sb[v],\sb[y]]] + [\sb[v],[\sb[y],\sb[u]]] + [\sb[y],[\sb[u],\sb[v]]] = 0.
\end{gather*}
We shall assume that $y$ is not imaginary, as otherwise $\osb[w_n] = 0$, implying~\eqref{eq:aux-n}.
Hence, it suffices to show:
  $$[\sb[u],[\sb[v],\sb[y]]], [\sb[v],[\sb[y],\sb[u]]] \in \spn\{\osb[w_1],\osb[w_2],\ldots,\osb[w_{n-1}]\}. $$

If $u$ is imaginary or $[\sb[v],\sb[y]]=0$, then $[\sb[u],[\sb[v],\sb[y]]]=0$. Otherwise, $[\sb[u],[\sb[v],\sb[y]]]$
is a multiple of $[\sb[u],\sb[\SL(\deg(v) + \deg(y))]]$ and $u\ne \SL(\deg(v) + \deg(y))$. We note that
$\SL(\deg(v) + \deg(y))\geq vy$ due to~\eqref{eq:generalized Leclerc}, as $[\sb[v],\sb[y]]\ne 0$.
As $u<v<vy\leq \SL(\deg(v) + \deg(y))$, we see that $(u,\SL(\deg(v) + \deg(y)))$ is in $W_{k\delta}$ and is
smaller than $w_n$, that is, it belongs to $\{w_1,\ldots,w_{n-1}\}$. This implies that
$[\sb[u],[\sb[v],\sb[y]]]\in \spn\{\osb[w_1],\osb[w_2],\ldots,\osb[w_{n-1}]\}$.

If $v$ is imaginary or $[\sb[y],\sb[u]]=0$, then $[\sb[v],[\sb[y],\sb[u]]]=0$. Otherwise, $[\sb[v],[\sb[y],\sb[u]]]$
is a multiple of $[\sb[v],\sb[\SL(\deg(y) + \deg(u))]]$ and also $v \ne \SL(\deg(y) + \deg(u))$.
Let $\ell\in \mathrm{L}$ denote the appropriate concatenation of $v$ and $\SL(\deg(y) + \deg(u))$.
As $[\sb[y],\sb[u]] \neq 0$ and $u<y$, we get $\SL(\deg(u) + \deg(y)) \geq uy$ due to~\eqref{eq:generalized Leclerc}.
Combining this with Corollary~\ref{cor:rotate}, we obtain $\ell>uvy$. Hence, either $(v,\SL(\deg(y) + \deg(u))$ or
$(\SL(\deg(y) + \deg(u),v)$ is in $W_{k\delta}$ and is smaller than $w_n$, that is, it belongs to $\{w_1,\ldots,w_{n-1}\}$.
This implies that $[\sb[v],[\sb[y],\sb[u]]]\in \spn\{\osb[w_1],\osb[w_2],\ldots,\osb[w_{n-1}]\}$.

This completes the proof of the inductive step, and hence of the proposition.
\end{proof}

Let us illustrate this proposition with a couple of examples (cf.\ Notation~\ref{notation:h}):

\begin{example}\label{ex:F4}
Consider the affine type $F_4^{(1)}$ with the order $3<4<0<2<1$ (our labeling of simple roots follows \cite[\S4, Table Fin]{K}).
Using the code (see Listing~\ref{lst:w.set} of Appendix~\ref{sec:app_code}), we write the set $W_\delta$ in increasing order,
with the pairs corresponding to the costandard factorization of $\SL$-words highlighted in bold:
\begin{align*}
  & \mathbf{(3432104321, 32)}, \mathbf{(3432104, 32321)}, (343210432, 321), (34321043232, 1), \\
  & \mathbf{(343210321, 324)},(34321032132, 4), (343210, 324321), (34321032, 3214),\\
  & \mathbf{(343214, 323210)}, (34321432, 3210), (34321432321, 0), (3432143232, 01), \\
  & (34321, 3243210), (3432132, 32104), (34321343210, 2), (34, 3243210321), \\
  & (34324, 3213210), (3432, 32143210), (343234321, 012), (3432343210, 21), \\
  & (33210, 3243214), (3321, 32432104), (3, 32432104321), (332, 321432104).
\end{align*}
Then, $\osb[(343210432, 321)]\sim h_{321}t$ and $\osb[(34321043232,1)]\sim h_1t$ are in the span of
$\sb[\SL_1(\delta)]=\sb[343210432132]\sim h_{32}t$ and $\sb[\SL_2(\delta)]=\sb[343210432321]\sim h_{32321}t$.
Likewise, for non-highlighted elements in the second line:
$\osb[(34321032132, 4)]\sim h_4t$, $\osb[(343210, 324321)]\sim h_{324321}t$, $\osb[(34321032, 3214)]\sim h_{3214}t$,
all of which are in the span of the above $\sb[\SL_1(\delta)]$, $\sb[\SL_2(\delta)]$, and $\sb[\SL_3(\delta)]=\sb[343210321324]\sim h_{324}t$.
\end{example}

\begin{example}\label{ex:E6}
Consider the affine type $E_6^{(1)}$ with the order $3<0<1<5<4<6<2$
(our labeling of simple roots follows \cite[\S4, Table Fin]{K}).
Using the code (see Listing~\ref{lst:w.set} of Appendix~\ref{sec:app_code}), we write the set $W_\delta$ in increasing order,
with the pairs corresponding to the costandard factorization of $\SL$-words highlighted in bold:
\begin{align*}
  & \mathbf{(3645032641, 32)}, \mathbf{(364503264, 321)}, (36450326432, 1), \mathbf{(364503261, 324)},\\
  & (36450326132, 4), (36450326, 3241), \mathbf{(36450, 3241326)}, (364503241, 326), \\
  & (36450324132, 6), (36450324, 3261), (36450321, 3264), (3645032, 32641), \\
  & \mathbf{(3645, 32413260)}, (36453241, 3260), (36453241326, 0), (3645324132, 06),\\
  & (3645324, 32610), (3645321, 32640), (364532, 326410), \mathbf{(36403261, 3245)}, \\
  & (36403261324, 5), (3640326132, 54), (3640326, 32451), (3640, 32451326),\\
  & (3640321, 32645), (36403213645, 2), (364032, 326451), (3640323645, 12),\\
  & (364, 324513260), (364321, 326450), (36432, 3264510), (360, 324513264),\\
  & (36, 3245132640), (345, 326103264), (34, 3261032645), (3, 32641032645).
\end{align*}
Then, for example, the pseudo-bracketings of all elements in the third line are
$\osb[(36450324132,6)]\sim h_6t$,  $\osb[(36450324, 3261)]\sim h_{3261}t$,
$\osb[(36450321, 3264)]\sim h_{3264}t$, $\osb[(3645032, 32641)]\sim h_{32641}t$, which are
in the linear span of $\sb[\SL_1(\delta)]=\sb[364503264132]\sim h_{32}t$,
$\sb[\SL_2(\delta)]=\sb[364503264321]\sim h_{321}t$, $\sb[\SL_3(\delta)]=\sb[364503261324]\sim h_{324}t$,
and $\sb[\SL_4(\delta)]=\sb[364503241326]\sim h_{3241326}t$.
\end{example}

As an important corollary of Proposition~\ref{prop:im.bracketing}, we obtain:

\begin{corollary}\label{cor:im.upper.limit}
Consider two standard Lyndon words $u,v$ such that $u<v$ and $\deg(u) + \deg(v) = k\delta$.
If $\SL_i(k\delta) < uv$, then $[\sb[u],\sb[v]] \in \spanset_{i-1}^k$.
\end{corollary}

\begin{proof}
As $(u,v) \in W_{k\delta}$, we get $[\sb[u],\sb[v]] \in \spanset_{m}^k$ with
$m=\max\{j \,|\, uv\leq \SL_j(k\delta)\}$ by Proposition~\ref{prop:im.bracketing}.
As $\SL_i(k\delta) < uv$, we see that $m\leq i-1$, implying the result.
\end{proof}


\subsection{$O$-sets and their properties}
\

We can now introduce a set of roots that will be key to our notion of convexity:

\begin{definition}\label{def:convex.set.im}
For any $\alpha = (k\delta,i) \in \imx$, define the following set:
\begin{equation*}
    \convexsetim(\alpha) =
    \left\{\beta \,\Bigg\vert\,
         \substack{
            \beta \in \wDelta^{+,\re}, \beta = \beta' + p\delta\ with \ \beta'\in \wDelta^{+,\re}, |\beta'| < |\delta|, p\in \BZ_{\geq 0}\\
            [\sb[\SL(\beta')],\sb[\SL(k\delta-\beta')]] \not\in \spanset_{i-1}^k\\
            \exists\, j \leq i\ s.t.\ [\sb[\SL(\beta')],\sb[\SL_j(k\delta)]] \neq 0
         }
    \right\}.
\end{equation*}
\end{definition}

\begin{remark}\label{remark:convex.im.set}
The importance of the second line in the right-hand side is that for $\beta,\gamma \in \convexsetim(\alpha)$ with
$\beta + \gamma = k\delta$, we have $\SL(\alpha) > \min\{\SL(\beta),\SL(\gamma)\}$ by Corollary~\ref{cor:im.upper.limit}.
\end{remark}

\begin{definition}\label{def:convexsetre}
For any $\alpha \in \wDelta^{+,\re}$, define the following set:
\begin{equation*}
  \convexsetre(\alpha) =
  \Big\{\beta \,\big\vert\, \beta \in \imx, \alpha \in \convexsetim(\beta) \Big\}.
\end{equation*}
\end{definition}

The max and min elements of $O(\alpha)\cap \{(k\delta,r)\}_{r=1}^{|I|}$ are of special interest:

\begin{definition}\label{def:min.max.real.convex}
For any $\alpha \in \wDelta^{+,\re}$ and $k\in \BZ_{>0}$, we define
\begin{equation}\label{eq:Mm}
\begin{split}
  M_k(\alpha) = \max\big\{\beta \in \convexsetre(\alpha) \,\colon |\beta| = |k\delta|\big\},\\
  m_k(\alpha) = \min\big\{\beta \in \convexsetre(\alpha) \,\colon |\beta| = |k\delta|\big\}.
\end{split}
\end{equation}
\end{definition}

We note that
\begin{equation*}
  [\sb[\SL(\beta')],\sb[\SL(k\delta-\beta')]] \not\in \spanset_{i-1}^k
\end{equation*}
implies the same for any $i'<i$. Likewise,
\begin{equation*}
  \exists\, j \leq i\ \mathrm{such\ that}\ [\sb[\SL(\beta')],\sb[\SL_j(k\delta)]] \neq 0
\end{equation*}
implies the same for any $i'>i$. These two observations imply
\begin{equation*}
  \Big\{ \beta\in \convexsetre(\alpha) \,\colon |\beta| = |k\delta| \Big\} =
  \Big\{ \beta\in \imx \,\colon |\beta| = |k\delta| ,  m_k(\alpha)\leq \beta\leq M_k(\alpha) \Big\}
\end{equation*}
as well as give rise to more explicit description of~\eqref{eq:Mm}.

\begin{lemma}\label{lemma:Mk}
For any $\alpha \in \wDelta^{+,\re}$ and $k\in \BZ_{>0}$, we have
\begin{equation*}
  M_k(\alpha) = \max \big\{\beta \in \imx \,\colon |\beta| = |k\delta|,\, [\sb[\SL(\beta)],\sb[\SL(\alpha)]] \neq 0 \big\}.
\end{equation*}
\end{lemma}

\begin{proof}
Let $M'_k(\alpha)=(k\delta,i)$ denote the right-hand side above. First, we claim that $M'_k(\alpha) \in \convexsetre(\alpha)$.
Indeed, $h_\alpha t^k \notin \spanset^k_{i-1}$ follows from $h_\alpha t^k$ being orthogonal to the spanning set of
$\spanset_{i-1}^k$. Thus $M_k(\alpha)\geq M'_k(\alpha)$. But $M_k(\alpha)>M'_k(\alpha)$ would contradict the definition
of $M'_k(\alpha)$. This establishes $M_k(\alpha) = M'_k(\alpha)$, as claimed.
\end{proof}

\begin{corollary}\label{cor:im.span.vanish}
If $M_k(\alpha) = (k\delta,i)$ for $\alpha \in \wDelta^{+,\re}$, then $[ht^k,\SL(\alpha)] = 0$ for all $ht^k \in \spanset_{i-1}^k$.
\end{corollary}

\begin{lemma}\label{lemma:mk}
For any $\alpha \in \wDelta^{+,\re}$ and $k\in \BZ_{>0}$, we have
\begin{equation*}
  m_k(\alpha) = (k\delta,i)\quad \mathrm{with} \quad  h_\alpha t^k \in \spanset_{i}^k \backslash\spanset_{i-1}^k.
\end{equation*}
\end{lemma}

\begin{proof}
Pick $i$ such that $h_\alpha t^{k} \in \spanset_{i}^k \backslash\spanset_{i-1}^k$. First, we claim that
$(k\delta,i) \in \convexsetre(\alpha)$. To do so, it suffices to prove that there is index $j\leq i$ with
$[\sb[\SL_j(k\delta)],\sb[\SL(\alpha)]] \neq 0$. The latter follows immediately by noting that $h_\alpha t^k$
can not be orthogonal to the spanning set of $\spanset_{i}^k$ which contains $h_\alpha t^k$. Thus
$m_k(\alpha) \leq (k\delta,i)$. But $m_k(\alpha) < (k\delta,i)$ would contradict $\alpha\in C(m_k(\alpha))$.
This completes the proof.
\end{proof}

\begin{proposition}\label{prop:general.bracketing}
For any factorization $u = u_1u_2$ with $u,u_1,u_2 \in \SL$, we have:

\medskip
\noindent
(a) $[\sb[u_1],\sb[u_2]] \neq 0$;

\medskip
\noindent
(b) if $u$ is imaginary and $u = \SL_i(k\delta)$, then $[\sb[u_1],\sb[u_2]] \in \spanset^{k}_i \backslash\spanset_{i-1}^k$.
\end{proposition}

\begin{remark}\label{rem:stan=costar-words}
This result immediately implies that if we were to use the standard factorization
instead of the costandard one in Definition~\ref{bracketing}, we would still get exactly the same affine
standard Lyndon words as well as the same flags $\spanset_\bullet^k$.
\end{remark}

\begin{proof}[Proof of Proposition~\ref{prop:general.bracketing}]
Let $w_1,w_2, \ldots, w_n,\ldots$ be the elements of $\ol{W}$ ordered in increasing order. Then,
$[\sb[u_1],\sb[u_2]]=\osb[w_n]$ if $w_n=(u_1,u_2)$. We shall prove both parts by induction on $n$.
The base case $n=1$ is clear, since $w_1$ represents the costandard factorization of an affine
standard Lyndon word. For the inductive step, we shall assume that the result holds for all $\{w_m\}_{m<n}$.

First, we consider the case when
$(w_n)_1(w_n)_2$ is real. Part (a) is clear if $w_n$ represents the costandard factorization. Let us now assume that
$w_n$ does not represent the costandard factorization. Invoking Corollary~\ref{cor:three.way.lyndon}, consider the
costandard factorization $(w_n)_1 = uv$, so that $u,v,vy \in \LL$, where we use $y=(w_n)_2$.
Being factors of a standard Lyndon word, we actually have $u,v,vy \in \SL$. By the Jacobi identity:
\begin{equation}\label{eq:jacobi-proof}
    [\sb[u],[\sb[v],\sb[y]]] + [\sb[v],[\sb[y],\sb[u]]] + [\sb[y],[\sb[u],\sb[v]]] = 0.
\end{equation}
Therefore, it suffices to show that $[\sb[u],[\sb[v],\sb[y]]]\ne 0$ and $[\sb[v],[\sb[y],\sb[u]]]=0$.

First, let us show that $[\sb[u],[\sb[v],\sb[y]]]\ne 0$. This is clear if $\deg(vy)$ is real, since we have $[\sb[v],\sb[y]]\ne 0$
(by the inductive hypothesis applied to $(v,y)$) and so $[\sb[u],[\sb[v],\sb[y]]]$ is a nonzero multiple of $[\sb[u],\sb[vy]]$
which is nonzero (by the inductive hypothesis applied to $(u,vy)$). If $\deg(vy)$ is imaginary, then we claim that actually
$vy = \SL(M_k(\deg(u)))$ for $k=|vy|/|\delta|$. Indeed, if $vy > \SL(M_k(\deg(u)))$, then $[\sb[u],\sb[vy]] = 0$ by
Corollary~\ref{cor:im.span.vanish}, which contradicts the inductive hypothesis applied to $(u,vy)$. If $vy < \SL(M_k(\deg(u)))$,
then we would get a contradiction with the generalized Leclerc algorithm~\eqref{eq:generalized Leclerc}, since $uvy<u\SL(M_k(\deg(u)))$
and $[\sb[u],\sb[\SL(M_k(\deg(u)))]] \neq 0 $ by Lemma~\ref{lemma:Mk}. This proves the claimed equality $vy = \SL(M_k(\deg(u)))$.
Applying the inductive hypothesis (part (b)) to the pair $(v,y)$, we then have
$[\sb[v],\sb[y]] \in \spanset_i^k \backslash\spanset_{i-1}^{k}$ whereas $M_k(\deg(u)) = (k\delta,i)$.
Hence $[\sb[u],[\sb[v],\sb[y]]] \neq 0$ by Lemma~\ref{lemma:Mk}.

Next, let us prove that $[\sb[v],[\sb[y],\sb[u]]]=0$. Assuming the contradiction, we see that $[\sb[y],\sb[u]] \neq 0$,
and hence $\deg(y) + \deg(u) \in \wDelta^{+}$. If $\deg(y) + \deg(u) \in \wDelta^{+,\re}$, then
$uy \leq \SL(\deg(u) + \deg(y))=:z$ by the generalized Leclerc algorithm~\eqref{eq:generalized Leclerc}.
If $\deg(y) + \deg(u) = k\delta$, then we have $uy \leq \SL(M_k(\deg(v)))=:z$ by Corollary~\ref{cor:im.span.vanish}.
Invoking Corollary~\ref{cor:rotate}, in both cases we see that the appropriate concatenation of $v$ and $z$ is bigger than
$uvy$ and $[\sb[v],\sb[z]] \neq 0$. This implies that $uvy$ is not standard Lyndon, due to the generalized Leclerc algorithm,
a contradiction. Therefore, indeed we have $[\sb[v],[\sb[y],\sb[u]]] = 0$.

Let us now consider the case when $(w_n)_1(w_n)_2$ is imaginary, that is, equal to $\SL_i(k\delta)$ for some $i,k$.
It suffices to prove part (b) only. The result is clear if $w_n$ represents the costandard factorization of $\SL_i(k\delta)$.
Let us now assume that $w_n$ does not represent the costandard factorization of $\SL_i(k\delta)$. Using the same notations
$u,v,y$ as above, and invoking the equality~\eqref{eq:jacobi-proof}, it thus suffices to prove:
  $$[\sb[u],[\sb[v],\sb[y]]] \in \spanset_i^k \backslash \spanset_{i-1}^k, \qquad [\sb[v],[\sb[y],\sb[u]]] \in \spanset_{i-1}^k.$$

By the inductive hypothesis applied to $(u,vy)$, we see that $\deg(u),\deg(vy) \in \wDelta^{+,\re}$, and furthermore
$[\sb[u],\sb[vy]] \in \spanset_i^k \backslash \spanset_{i-1}^k$. On the other hand, by the inductive hypothesis applied
to $(v,y)$, we know that $[\sb[v],\sb[y]]$ is a nonzero multiple of $\sb[vy]$. This implies the first inclusion above:
$[\sb[u],[\sb[v],\sb[y]]] \in \spanset_i^k \backslash \spanset_{i-1}^k$.

If $[\sb[v],[\sb[y],\sb[u]]]\ne 0$, then $\deg(v), \deg(u) + \deg(y) \in \wDelta^{+,\re}$. As $u < v < y$, we then have
$uy \leq \SL(\deg(u) + \deg(y))$ by~\eqref{eq:generalized Leclerc}. Invoking Corollary~\ref{cor:rotate} once again, we see
that the appropriate concatenation of $v$ and $\SL(\deg(u) + \deg(y))$ is bigger than $uvy=\SL_i(k\delta)$. Therefore,
$[\sb[v],\sb[\SL(\deg(u) + \deg(y))]] \in \spanset_{i-1}^k$ by Corollary~\ref{cor:im.upper.limit}. Since $[\sb[u],\sb[y]]$
is a multiple of $\sb[\SL(\deg(u) + \deg(y))]$, we thus get the second inclusion above:
$[\sb[v],[\sb[y],\sb[u]]] \in \spanset_{i-1}^k$.
\end{proof}

\begin{corollary}\label{cor:imaginary.no.im.splitting}
For any $1\leq i\leq |I|$ and $k\in \BZ_{>0}$, we cannot have $\SL_i(k\delta) = uv$ with $u,v \in \SL$ being imaginary words.
\end{corollary}

\begin{proof}
If $u,v$ were imaginary, we would have $[\sb[u],\sb[v]] = 0$ as $[\fh t^a, \fh t^b]=0$ for all $a,b>0$,
a contradiction with Proposition~\ref{prop:general.bracketing}(a).
\end{proof}

\begin{corollary}\label{cor:invariance.to.im.bracketing}
For any $\alpha \in \wDelta^{+,\re}$ and $k\in \BZ_{>0}$, consider any factorization $\SL(M_k(\alpha)) = u_1u_2$ with $u_1,u_2 \in \SL$.
Then $[[\sb[u_1],\sb[u_2]],\sb[\SL(\alpha)]] \neq 0$.
\end{corollary}

\begin{proof}
Let $M_k(\alpha) = (k\delta,i)$. Due to Proposition~\ref{prop:general.bracketing}(b), we have
$[\sb[u_1],\sb[u_2]]=\sum_{j=1}^i a_j\sb[\SL_j(k\delta)]$ with $a_j\in \BC$ and $a_i\ne 0$. But
$[\sb[\SL_j(k\delta)],\sb[\SL(\alpha)]]$ vanishes for $j<i$ and is nonzero for $j=i$ by Lemma~\ref{lemma:Mk}.
The result follows.
\end{proof}

\begin{corollary}\label{cor:mk.im.factor}
For any $1\leq i\leq |I|$, $k\in \BZ_{>0}$, and any splitting $\SL_i(k\delta) = uv$ with $u,v \in \SL$, we have
$m_k(\deg(u)) = (k\delta,i)=m_k(\deg(v))$.
\end{corollary}

\begin{proof}
By Corollary~\ref{cor:imaginary.no.im.splitting} we know that $\deg(u)\in \wDelta^{+,\re}$ and so
$h_{\deg(u)}t^k \in \spanset_i^{k} \backslash \spanset_{i-1}^k$ by Proposition~\ref{prop:general.bracketing}(b).
Therefore, $m_k(\deg(u)) = (k\delta,i)$ by Lemma~\ref{lemma:mk}. The proof of $m_k(\deg(v)) = (k\delta,i)$ is analogous.
\end{proof}

\begin{corollary}\label{cor:imaginary.suffix.prefix}
For any $\alpha \in \wDelta^{+,\re}$ consider any splitting $\SL(\alpha) = uv$ with $u,v \in \SL$.
If $\deg(u) = k\delta$, then $u = \SL(M_k(\alpha))$. If $\deg(v) = k\delta$, then $v = \SL(M_k(\alpha))$.
\end{corollary}

\begin{proof}
First, let us assume that $\deg(u) = k\delta$. If $u < \SL(M_k(\alpha))$, then we would get a contradiction to the generalized
Leclerc algorithm~\eqref{eq:generalized Leclerc}. Indeed, for $\SL(M_k(\alpha))<v$ we have $\SL(\alpha)=uv<\SL(M_k(\alpha))v$,
while for $\SL(M_k(\alpha))>v$ we have $\SL(\alpha)=uv<vu<v\SL(M_k(\alpha))$, as well as $[\sb[\SL(M_k(\alpha))],\sb[v]] \neq 0$.
On the other hand, if $u > \SL(M_k(\alpha))=\SL(M_k(\alpha-k\delta))$, then $[\sb[u],\sb[v]] = 0$ by Corollary~\ref{cor:im.span.vanish},
a contradiction to Proposition~\ref{prop:general.bracketing}(a). Therefore, $u=\SL(M_k(\alpha))$.

The proof of the claim for the case $\deg(v) = k\delta$ is analogous.
\end{proof}


\section{Convexity and Monotonicity}\label{sec:keyresults}

In this section, we prove two key results of the present note, which generalize the results of~\cite[\S5]{AT} to
all types and do not rely on the explicit formulas.


\subsection{Statement of convexity}
\

For any $\alpha\in \wDelta^{+,\ext}$, let us introduce sets $\leftfactorsset_\alpha, \rightfactorsset_\alpha$ that
are key for the convexity:

\begin{definition}\label{def:left.and.right.factor.sets}
(a) For $\alpha \in \wDelta^{+,\re}$, we define:
\begin{gather*}
        \leftfactorsset_\alpha = \left\{
        \beta \,\bigg\vert\,
        \substack{
            \beta \in \wDelta^{+,\re} \\
            \gamma:=\alpha-\beta \in \wDelta^{+,\re} \\
            \beta < \gamma}
        \right\} \cup
        \Big\{ \beta \,\Big\vert\,
        \substack{
            \beta \in \convexsetre(\alpha), |\beta| = |k\delta| < |\alpha| \\
            M_k(\alpha) < \alpha - k\delta}
        \Big\} \cup
        \Big\{ \alpha - k\delta \,\Big\vert \,
          \substack{0<|k\delta| < |\alpha| \\
          M_k(\alpha) > \alpha - k\delta}
        \Big\}, \\
        \rightfactorsset_\alpha =
        \left\{ \gamma \,\bigg\vert \,
        \substack{
            \gamma \in \wDelta^{+,\re} \\
            \beta:=\alpha-\gamma \in \wDelta^{+,\re} \\
            \beta < \gamma}
        \right\} \cup
        \Big\{ \beta \,\Big\vert\,
        \substack{
            \beta \in \convexsetre(\alpha), |\beta| = |k\delta| < |\alpha|\\
            M_k(\alpha)  > \alpha- k\delta}
        \Big\} \cup
        \Big\{ \alpha - k\delta \,\Big\vert\,
        \substack{0 < |k\delta| < |\alpha|\\
        M_k(\alpha) < \alpha - k\delta}\Big\}.
    \end{gather*}

\noindent
(b) For $\alpha \in \imx$, we define the corresponding sets
$\leftfactorsset_\alpha, \rightfactorsset_\alpha\subset \convexsetim(\alpha)$ via:
\begin{align*}
  \leftfactorsset_\alpha &=
    \big\{\beta \,\big|\, \beta \in \convexsetim(\alpha),\,  |\beta| < |\alpha|,\, \mathrm{and}\ \beta < \alpha - \beta \big\},\\
  \rightfactorsset_\alpha &=
    \big\{\alpha - \beta \,\big|\, \beta \in \leftfactorsset_\alpha\big\}.
\end{align*}
\end{definition}

\begin{remark}
While the above usage of $M_k(\alpha)$ in (a), when comparing with $\alpha-k\delta$, may look strange, we will see in
Lemma~\ref{lemma:weak.monotonicity} that $m_k(\alpha) < \alpha - k\delta \Leftrightarrow M_k(\alpha) < \alpha- k\delta$ and
$m_k(\alpha) > \alpha - k\delta \Leftrightarrow M_k(\alpha) > \alpha - k\delta$. Thus, in fact, one can use any $\beta\in O(\alpha)$
with $|\beta|=|k\delta|$ instead of $M_k(\alpha)$.
\end{remark}

The first main result of this section is the following convexity:

\begin{theorem}\label{thm:convexity}
For any $\alpha \in \wDelta^{+,\ext}$, we have:
  $$ \beta < \alpha < \gamma \quad \textit{for all} \quad \beta \in \leftfactorsset_{\alpha}, \gamma\in \rightfactorsset_{\alpha}. $$
Equivalently:
\begin{gather}\label{eqn:convex.2}
    \max(\leftfactorsset_\alpha) < \alpha < \min(\rightfactorsset_\alpha) \quad \textit{for all} \quad \alpha \in \wDelta^{+,\ext}.
\end{gather}
\end{theorem}

Similarly to the convexity for finite type of Proposition~\ref{prop:fin.convex}, the inequality
  $$ \max(\leftfactorsset_\alpha) < \alpha $$
is easy, as it follows from Lemma~\ref{lemma:max.left.factor.set}.
Thus, the key is to prove $\alpha < \min(\rightfactorsset_\alpha)$.

\begin{remark}
We note that~\eqref{eqn:convex.2} implies the pre-convexity~\eqref{eq:pre-convex}, see Remark~\ref{rem:convexity-rephrased}.
\end{remark}


\subsection{Auxiliary results}
\

In preparation for the general monotonicity of Proposition~\ref{prop:monotonicity}, let us establish:

\begin{lemma}\label{lemma:weak.monotonicity}
Fix $\alpha \in \wDelta^{+,\re}$ and $k\in \BZ_{>0}$. If Theorem~\ref{thm:convexity} holds for all roots of height
$\leq |\alpha + k\delta|$, then the following properties are equivalent:
\begin{enumerate}

\item
$\alpha < \alpha+\delta < \cdots < \alpha + k\delta < M_1(\alpha)$
(resp.\ $\alpha > \alpha+\delta > \cdots > \alpha + k\delta > M_1(\alpha)$),

\item
$\alpha < M_1(\alpha)$
(resp.\ $\alpha > M_1(\alpha))$,

\item
$\alpha$ is less than (resp.\ greater than) all elements in $\convexsetre(\alpha)$ of height $< |\alpha + k\delta|$.

\end{enumerate}
\end{lemma}

\begin{proof}
First, let us prove the ``$(2) \Rightarrow (1)$'' implication. If $\alpha < M_1(\alpha)$, then
$\alpha \in \leftfactorsset_{\alpha + \delta}$ and $M_1(\alpha) \in \rightfactorsset_{\alpha + \delta}$,
so that $\alpha < \alpha + \delta < M_1(\alpha)$ by the convexity~\eqref{eqn:convex.2}. But
$M_1(\alpha+p\delta)=M_1(\alpha)$ for any $p\in \BZ_{>0}$, as $\sb[\SL(\alpha+p\delta)]\sim \sb[\SL(\alpha)]t^p$.
Thus, repeating the above argument with $\alpha+\delta$ instead of $\alpha$, then with $\alpha+2\delta$ and
so on, we eventually get the desired chain
  $\alpha < \alpha + \delta < \alpha + 2\delta < \cdots < \alpha + k\delta < M_1(\alpha)$.
Similarly, if $M_1(\alpha) < \alpha$, then $M_1(\alpha) \in \leftfactorsset_{\alpha + \delta}$ and
$\alpha \in \rightfactorsset_{\alpha + \delta}$, so that $M_1(\alpha)< \alpha + \delta < \alpha$
by the convexity~\eqref{eqn:convex.2}. Then, following the same logic as above, we obtain
  $\alpha > \alpha + \delta > \alpha + 2\delta > \cdots > \alpha + k\delta  > M_1(\alpha)$.

Let us prove the ``$(1) \Rightarrow (3)$'' implication for
$\alpha < \alpha+\delta < \cdots < \alpha + k\delta < M_1(\alpha)$. Pick any $\beta \in \convexsetre(\alpha)$
with $|\beta| = |p\delta| < |\alpha+k\delta|$. We consider two cases:
\begin{itemize}[leftmargin=0.5cm]

\item
If $|\beta| < |\alpha|$, then set $\alpha':=\alpha - p\delta$. Considering $\alpha'$ and $M_p(\alpha)$, we see
that either $\alpha'\in \leftfactorsset_{\alpha}, M_p(\alpha) \in \rightfactorsset_{\alpha}$ or
$\alpha'\in \rightfactorsset_{\alpha}, M_p(\alpha) \in \leftfactorsset_{\alpha}$. However, it must be the first
of these options, due to Theorem~\ref{thm:convexity} and the assumption $\alpha' < \alpha$. Hence,
$\beta \in \rightfactorsset_{\alpha}$, and we obtain $\alpha < \beta$ by applying Theorem~\ref{thm:convexity} once again.

\item
If $|\beta| > |\alpha|$, let us choose $s\in \BZ_{\geq 0}$ so that
$\ol{\alpha} = \alpha - s\delta\in \wDelta^{+,\re}$ satisfies $0 < |\overline{\alpha}| < |\delta|$.
As $\alpha<M_1(\alpha)$, we have $\alpha - \delta < \alpha < M_1(\alpha)$ by Theorem~\ref{thm:convexity}.
Repeating this argument, we get $\ol{\alpha} = \alpha-s\delta < \alpha < M_1(\alpha)$. Define
$\ol{\alpha}':=\delta-\ol{\alpha}$. Applying Theorem~\ref{thm:convexity} to $\delta = \ol{\alpha} + \ol{\alpha}'$,
we obtain $\ol{\alpha} < M_1(\alpha) < \ol{\alpha}'$. Pick $t\in \BZ_{\geq 0}$ so that
$|\alpha + t\delta| < |p\delta| < |\alpha + (t+1)\delta|$. Combining the assumption with $M_1(\alpha) < \ol{\alpha}'$
proved above, we get $\alpha + t\delta < M_1(\alpha) < \ol{\alpha}'$. But then we have
$\alpha + t\delta \in \leftfactorsset_{\beta}$ and $\overline{\alpha}' \in \rightfactorsset_{\beta}$, so that
$\alpha + t\delta < \beta$ by Theorem~\ref{thm:convexity}. Combining this with $\alpha \leq \alpha + t\delta$,
we ultimately get $\alpha < \beta$.

\end{itemize}

The proof of the ``$(1) \Rightarrow (3)$'' implication for $\alpha > \alpha+\delta > \cdots > \alpha + k\delta > M_1(\alpha)$
is similar. Pick $\beta \in\convexsetre(\alpha)$ with $\beta = |p\delta| < |\alpha+k\delta|$. We consider two cases:
\begin{itemize}[leftmargin=0.5cm]

\item
If $|\beta| < |\alpha|$, then consider $\alpha' = \alpha - p\delta$, so that $\alpha' > \alpha$ by the assumption.
Considering $\alpha'$ and $M_p(\alpha)$, we likewise deduce from Theorem~\ref{thm:convexity} that
$\alpha' \in \rightfactorsset_{\alpha}$, $M_p(\alpha) \in \leftfactorsset_{\alpha}$.
Therefore, we have $\beta \in \leftfactorsset_{\alpha}$ and so $\beta < \alpha$ by Theorem~\ref{thm:convexity}.

\item
If $|\beta| > |\alpha|$, then consider $\ol{\alpha}, \ol{\alpha}'$ defined as in the previous case.
As $M_1(\alpha)<\alpha$, we have $M_1(\alpha) < \alpha < \alpha - \delta$ by Theorem~\ref{thm:convexity}.
Repeating this argument, we get $M_1(\alpha) < \alpha < \alpha - s\delta = \ol{\alpha}$. Applying Theorem~\ref{thm:convexity}
to $\delta=\ol{\alpha}+\ol{\alpha}'$, we then obtain $\ol{\alpha}' < M_1(\alpha) < \ol{\alpha}$. Pick $t$ as in the previous case, and note
that $\ol{\alpha}' < M_1(\alpha) < \alpha + t\delta$. Thus  $\alpha + t\delta \in \rightfactorsset_{\beta}$ and
$\ol{\alpha}' \in \leftfactorsset_{\beta}$, and so $\beta < \alpha + t\delta$ by Theorem~\ref{thm:convexity}.
Combining this with $\alpha + t\delta\leq \alpha$, we get $\beta < \alpha$.

\end{itemize}

Finally, the implication ``$(3) \Rightarrow (2)$'' is obvious. This completes the proof.
\end{proof}

\begin{remark}\label{rem:convexity-rephrased}
The above result provides a more natural version of Theorem~\ref{thm:convexity}:
\begin{itemize}[leftmargin=0.5cm]

\item
if $\alpha < \beta$ with $\alpha,\beta,\alpha + \beta \in \wDelta^{+,\re}$, then $\alpha < \alpha + \beta < \beta$;

\item
if $\alpha < \beta$ with $\alpha \in \wDelta^{+,\re}, \beta \in \convexsetre(\alpha)$, then
$\alpha < \alpha + \beta < \beta$;

\item
if $\alpha < \beta$ with $\beta \in \wDelta^{+,\re}, \alpha \in \convexsetre(\beta)$, then
$\alpha < \alpha + \beta < \beta$;

\item
if $\alpha < \beta$ with $\alpha,\beta \in \wDelta^{+,\re}$ and $\alpha + \beta = k\delta$, then
$\alpha < \gamma < \beta$ for all $\gamma \in \convexsetre(\alpha)$ satisfying $|\gamma| = |k\delta|$.

\end{itemize}
\end{remark}

The next result indicates the importance of the set $\leftfactorsset_\alpha$:

\begin{lemma}\label{lemma:max.left.factor.set}
For $\alpha \in \wDelta^{+,\ext}$, we have:
\begin{equation*}
  \max(\leftfactorsset_\alpha) = \deg(\SL^{ls}(\alpha)).
\end{equation*}
\end{lemma}

\begin{proof}
First, note $\SL^{ls}(\alpha)<\SL(\alpha)<\SL^{rs}(\alpha)$. We claim that $\deg(\SL^{ls}(\alpha)) \in \leftfactorsset_\alpha$.
This is clear if $\deg(\SL^{ls}(\alpha)),\deg(\SL^{rs}(\alpha)) \in \wDelta^{+,\re}$. If $\deg(\SL^{rs}(\alpha)) \in \imx$, then
$\SL^{rs}(\alpha)=\SL(M_k(\alpha))$ for some $k$ by Corollary~\ref{cor:imaginary.suffix.prefix}, and the claim follows. Likewise,
if $\deg(\SL^{ls}(\alpha)) \in \imx$, then $\SL^{ls}(\alpha)=\SL(M_k(\alpha))$ for some $k$, and the claim follows.

Let us now assume the contrary, that is, $\max(\leftfactorsset_\alpha) > \deg(\SL^{ls}(\alpha))$. This implies
$\max (\leftfactorsset_\alpha) > \alpha$, due to Lemma~\ref{lemma:equiv.to.standard.fac}. We shall now consider two cases:
\begin{itemize}

\item[(1)] For $\alpha \in \wDelta^{+,\re}$, we claim that we get a contradiction with the generalized Leclerc
algorithm~\eqref{eq:generalized Leclerc}. This is clear if $\beta=\max(\leftfactorsset_\alpha)$ is real. On the other hand,
if $\beta\in \imx$, then we must have $\beta=M_k(\alpha)$ for some $k$ by the maximality of $\beta$,
and so $M_k(\alpha) < \alpha - k\delta$.
Therefore, $\SL(M_k(\alpha))\SL(\alpha - k\delta)>\SL(\alpha)$ contradicts to the generalized Leclerc
algorithm~\eqref{eq:generalized Leclerc}, as $[\sb[\SL(M_k(\alpha))],\sb[\SL(\alpha-k\delta)]] \neq 0$.

\item[(2)] For $\alpha = (k\delta,i)\in \imx$, let $\beta = \max( \leftfactorsset_\alpha)$. Then $\gamma:=\alpha-\beta$
satisfies $\gamma \in \rightfactorsset_\alpha\subset \convexsetim(\alpha)$ and $\alpha < \beta < \gamma$.
This contradicts to Remark~\ref{remark:convex.im.set}.

\end{itemize}
This completes the proof, as we get contradiction in both cases.
\end{proof}

We conclude this subsection with several simple results.

\begin{lemma}\label{lemma:convex.sum.rule}
For $\alpha, \beta, \alpha+\beta \in \wDelta^{+,\re}$, $\gamma \in \imx$ if $[\sb[\SL(\gamma)],\sb[\SL(\alpha + \beta)]] \neq 0$,
then $\sb[\SL(\gamma)]$ commutes with at most one of $\sb[\SL(\alpha)]$, $\sb[\SL(\beta)]$.
\end{lemma}

\begin{proof}
According to the Jacobi identity, we have:
\begin{multline*}
  [\sb[\SL(\gamma)],[\sb[\SL(\alpha)],\sb[\SL(\beta)]]] + [\sb[\SL(\alpha)],[\sb[\SL(\beta)],\sb[\SL(\gamma)]]] + \\
  [\sb[\SL(\beta)],[\sb[\SL(\gamma)],\sb[\SL(\alpha)]]] = 0.
\end{multline*}
Thus, if $\sb[\SL(\gamma)]$ commuted with both $\sb[\SL(\alpha)]$ and $\sb[\SL(\beta)]$, then we would have
\begin{equation*}
  [\sb[\SL(\gamma)],[\sb[\SL(\alpha)],\sb[\SL(\beta)]]] = 0.
\end{equation*}
The latter contradicts to the assumption as $[\sb[\SL(\alpha)],\sb[\SL(\beta)]]\sim \sb[\SL(\alpha + \beta)]$.
\end{proof}

\begin{corollary}\label{cor:max.im.rule}
Let $\alpha,\beta \in \wDelta^{+,\re}$ satisfy $\alpha + \beta \in \wDelta^{+,\re}$. If $M_k(\alpha) > M_k(\beta)$
then $M_k(\alpha) = M_k(\alpha + \beta)$, and if $M_k(\alpha) = M_k(\beta)$ then $M_k(\alpha + \beta) \leq M_k(\alpha)$.
\end{corollary}

\begin{proof}
If $M_k(\alpha) > M_k(\beta)$, then $\sb[\SL(M_k(\alpha))]$ commutes with $\sb[\SL(\beta)]$. Combining this with
$[\sb[\SL(\alpha+k\delta)],\sb[\SL(\beta)]]\ne 0$ (due to $\alpha + \beta \in \wDelta^{+,\re}$) and the fact that
$[\sb[\SL(M_k(\alpha))],\sb[\SL(\alpha)]] \sim \sb[\SL(\alpha)]t^k \sim \sb[\SL(\alpha+k\delta)]$, we obtain
$[\sb[\SL(M_k(\alpha))],\sb[\SL(\alpha + \beta)]] \neq 0$ by the Jacobi identity. Hence, $M_k(\alpha + \beta) \geq M_k(\alpha)$
by Lemma~\ref{lemma:Mk}. But the strict inequality $M_k(\alpha + \beta) > M_k(\alpha)$ would contradict
Lemma~\ref{lemma:convex.sum.rule}. Therefore, we have $M_k(\alpha + \beta ) = M_k(\alpha)$.

Likewise, if $M_k(\alpha) = M_k(\beta)$, then $M_k(\alpha + \beta) \leq M_k(\alpha)$ by Lemma~\ref{lemma:convex.sum.rule}.
\end{proof}

\begin{lemma}\label{lemma:min.im.rule}
Let $\alpha,\beta \in \wDelta^{+,\re}$ satisfy $\alpha + \beta \in \wDelta^{+,\re}$. If $m_k(\alpha) < m_k(\beta)$ then
$m_k(\alpha + \beta) = m_k(\alpha)$, and if $m_k(\alpha) = m_k(\beta)$ then $m_k(\alpha + \beta) \geq m_k(\alpha)$.
\end{lemma}

\begin{proof}
Both results follow immediately from Lemma~\ref{lemma:mk}, since $h_{\alpha+\beta}=ah_\alpha+bh_\beta$ with $a,b\ne 0$.
\end{proof}

The following result is standard, cf.~\cite[Claim 2.35]{NT}:

\begin{lemma}\label{lemma:gamma.difference}
Assume that roots $\alpha, \beta, \alpha', \beta' \in \wDelta^{+,\re}$ satisfy $\alpha + \beta = \alpha' + \beta'$ and

$\mathrm{(a)}$ $\alpha + \beta \not\in \wDelta^{+,\im}$

\noindent
or

$\mathrm{(b)}$ $\alpha + \beta \in \wDelta^{+,\im}$, but $(\alpha,\beta') \neq 0$.

\noindent
Then one of the following four cases must hold:
$$
  \alpha + \gamma = \alpha' \quad \text{and} \quad  \beta - \gamma = \beta'
$$
or
$$
  \alpha - \gamma = \alpha' \quad \text{and} \quad \beta + \gamma = \beta'
$$
or
$$
  \alpha + \gamma = \beta' \quad \text{and} \quad \beta - \gamma = \alpha'
$$
or
$$
  \alpha - \gamma = \beta' \quad \text{and} \quad \beta + \gamma = \alpha'
$$
with $\gamma \in \wDelta^{+}\cup \{0\}$.
\end{lemma}

\begin{proof}
Assume first that $\alpha + \beta \not\in \wDelta^{+,\im}$. If $(\alpha,\alpha') > 0$, then the reflection
$s_{\alpha}(\alpha') = \alpha' - k\alpha$ is a root, for some $k\in \BZ_{>0}$. This implies that
$\alpha' - \alpha\in \wDelta \sqcup \{0\}$, proving the claim. The same argument applies if $(\alpha,\beta')>0$,
$(\beta,\alpha')>0$, or $(\beta,\beta')>0$. However, one of these inequalities must hold as
$(\alpha+\beta, \alpha'+\beta') = (\alpha+\beta,\alpha+\beta)>0$.

In the case (b), we have $(\alpha+ \beta,\alpha' + \beta') = (\alpha + \beta, \alpha + \beta) = 0$,
but $(\alpha,\beta') \neq 0$. Hence, the above argument can be applied without changes.
\end{proof}


\subsection{Proof of convexity}
\

In this subsection, we establish the convexity~\eqref{eqn:convex.2}.

\begin{proof}[Proof of Theorem~\ref{thm:convexity}]
As noted earlier, Lemma~\ref{lemma:max.left.factor.set} implies that $\max(\leftfactorsset_{\mu}) < \mu$ for any
$\mu \in \wDelta^{+,\ext}$. Thus it remains to show that $\mu < \min(\rightfactorsset_\mu)$. Invoking
Lemma~\ref{lemma:equiv.to.standard.fac} and Lemma~\ref{lemma:max.left.factor.set}, the latter is equivalent to
\begin{equation}\label{eq:convexity-rephrased}
  \max(\leftfactorsset_{\mu}) < \beta' \quad \forall\, \beta'\in \rightfactorsset_{\mu}.
\end{equation}
We show this by induction on the height of $\mu$ (with the conventions $|(k\delta,i)|=|k\delta|$).
The base case $|\mu|=2$ is obvious.
For the inductive step, let $\SL(\mu) = \SL(\alpha)\SL(\beta)$ be the standard factorization, so that
$\alpha = \max(\leftfactorsset_\mu)$ by Lemma~\ref{lemma:max.left.factor.set}. Pick any
$\alpha' \in \leftfactorsset_\mu, \beta' \in \rightfactorsset_\mu$ with $\alpha' + \beta' = \alpha + \beta$.
We can assume that $\alpha' < \alpha$ and also that $\beta' < \beta$ as otherwise $\alpha < \mu < \beta \leq \beta'$.
Let us assume the contrary to~\eqref{eq:convexity-rephrased}:
  $$\beta' < \alpha.$$

We start with several general results.

\begin{claim}\label{claim:imaginary.claim.delta.movement}
For any $\alpha,\beta,\alpha + \beta \in \wDelta^{+,\re}$ and $k\in \BZ_{>0}$ such that
$\alpha'=\alpha + k\delta$ and $\beta'=\beta - k\delta$ are also affine positive roots, if
Theorem~\ref{thm:convexity} holds for all roots of height $<|\alpha + \beta|$, then
we cannot have $\alpha',\beta'<\alpha,\beta$.
\end{claim}

\begin{proof}
Assume to the contrary that $\alpha',\beta' < \alpha,\beta$. According to Lemma~\ref{lemma:weak.monotonicity}:
\begin{equation*}
  m_p(\alpha) \leq M_p(\alpha) < \alpha', \beta'< \alpha, \beta < m_p(\beta) \leq M_p(\beta)
  \qquad \forall\, 0<p<|\alpha+\beta|/|\delta|.
\end{equation*}
Hence, we have $m_p(\alpha + \beta) = m_p(\alpha)$ and $M_p(\alpha + \beta) = M_p(\beta)$, due to
Corollary~\ref{cor:max.im.rule} and Lemma~\ref{lemma:min.im.rule}.
Applying Theorem~\ref{thm:convexity} to $\alpha,\beta',\alpha+\beta'$, we obtain
  $m_p(\alpha + \beta) < \beta' < \alpha + \beta' < \alpha < M_p(\alpha + \beta)$.
As $m_p(\alpha + \beta')=m_p(\alpha + \beta), M_p(\alpha + \beta')=M_p(\alpha + \beta)$, the resulting inequalities
  $m_p(\alpha + \beta') < \alpha + \beta' < M_p(\alpha + \beta')$ contradict to Lemma~\ref{lemma:weak.monotonicity}.
\end{proof}

\begin{claim}\label{claim:imaginary.claim.delta.movement.imaginary}
If $\alpha + \beta \in \wDelta^{+,\im}$, $\alpha' = \alpha + k\delta$, $\beta' = \beta - k\delta$ with
$\alpha,\beta,\alpha+k\delta,\beta - k\delta\in \wDelta^{+,\re}$, $k\ne 0$, and Theorem~\ref{thm:convexity} holds
for all roots of height $<|\alpha + \beta|$, then we cannot have $\alpha',\beta' < \alpha,\beta$.
\end{claim}

\begin{proof}
Assume to the contrary that $\alpha',\beta' < \alpha,\beta$. If $k>0$, then according to Lemma~\ref{lemma:weak.monotonicity},
we have $M_1(\alpha) < \alpha',\beta'<\alpha,\beta < M_1(\beta)$, a contradiction with
$M_1(\alpha) = M_1(\beta)$ due to $\alpha + \beta \in \wDelta^{+,\im}$. If $k<0$, then we likewise have
$M_1(\beta) < \alpha',\beta'<\alpha,\beta < M_1(\alpha)$, which again contradicts to the equality $M_1(\alpha) = M_1(\beta)$.
\end{proof}

The next two claims cover the cases when exactly one of $\alpha,\alpha',\beta,\beta'$ is in $\imx$.

\begin{claim}\label{claim:imaginary.claim.one}
If $\alpha,\beta,\alpha + \beta \in \wDelta^{+,\re}$ and Theorem~\ref{thm:convexity} holds for all roots of height
$<|\alpha + \beta|$, then we cannot have $\alpha + \beta - k\delta,m_k(\alpha+\beta) < \alpha,\beta$ whenever
$|k\delta| < |\alpha + \beta|$.
\end{claim}

\begin{proof}
Assume the contrary, that is, $\alpha + \beta - k\delta ,m_k(\alpha + \beta) < \alpha,\beta$ for some $k$ with
$|k\delta| < |\alpha + \beta|$. According to Lemma~\ref{lemma:min.im.rule}, we have
$m_k(\alpha + \beta)\geq \min\{m_k(\alpha),m_k(\beta)\}$. We can assume without loss of
generality that $m_{k}(\alpha) \leq m_k(\beta)$, so that $m_k(\alpha)<\alpha$.

Let us first consider the case $|k\delta| > |\alpha|$. Applying Theorem~\ref{thm:convexity} to
$\bar{\alpha} = k\delta - \alpha$ and $\alpha$, we get $\bar{\alpha} < m_k(\alpha)< \alpha$, and so
$\bar{\alpha} < m_k(\alpha) \leq m_k(\alpha + \beta) < \beta$. But then $(\alpha + \beta - k\delta) + \bar{\alpha} = \beta$
while $\alpha + \beta - k\delta, \bar{\alpha}<\beta$, a contradiction with Theorem~\ref{thm:convexity}.

Let us now consider the case $|k\delta| < |\alpha|$. Applying Theorem~\ref{thm:convexity} to $\alpha - k\delta, \beta$
and using $\alpha+\beta-k\delta<\beta$, we get $\alpha - k\delta < \alpha + \beta -k\delta$. Thus,
$\alpha-k\delta<\alpha+\beta-k\delta<\alpha$, and we get $\alpha - k\delta < \alpha < m_k(\alpha)$ by
Lemma~\ref{lemma:weak.monotonicity}, a contradiction.
\end{proof}

\begin{claim}\label{claim:imaginary.claim.two}
If $\alpha,\beta, \alpha + \beta \in \wDelta^{+,\re}$ and Theorem~\ref{thm:convexity} holds for all roots of height
$<|\alpha + \beta|$, then we cannot have $\alpha,\beta < M_k(\alpha + \beta) ,\alpha + \beta - k\delta$ whenever
$|k\delta| < |\alpha + \beta|$.
\end{claim}

\begin{proof}
Assume the contrary: $\alpha,\beta < M_k(\alpha + \beta) ,\alpha + \beta - k\delta$ for some
$k$ with $|k\delta| < |\alpha + \beta|$. According to Corollary~\ref{cor:max.im.rule}, we have
$M_k(\alpha + \beta)\leq \max\{M_k(\alpha),M_k(\beta)\}$. We can assume without loss of
generality that $M_{k}(\alpha) \geq M_k(\beta)$, so that $\alpha<M_k(\alpha)$.

Let us first consider the case $|k\delta| > |\alpha|$. Applying Theorem~\ref{thm:convexity} to
$\bar{\alpha} = k\delta - \alpha$ and $\alpha$, we get $\alpha < M_k(\alpha)< \bar{\alpha}$, and so
$\alpha < M_k(\alpha + \beta) \leq M_k(\alpha) < \bar{\alpha}$. But then $(\alpha + \beta - k\delta) + \bar{\alpha} = \beta$,
while $\alpha + \beta - k\delta, \bar{\alpha}>\beta$, a contradiction with Theorem~\ref{thm:convexity}.

Let us now consider the case $|k\delta| < |\alpha|$. Applying Theorem~\ref{thm:convexity} to
$\alpha-k\delta,\beta$ and using $\alpha+\beta-k\delta>\beta$, we get $\alpha - k\delta > \alpha + \beta - k\delta$.
Thus, $\alpha-k\delta>\alpha+\beta-k\delta>\alpha$, and we get $\alpha - k\delta > \alpha > M_k(\alpha)$
by Lemma~\ref{lemma:weak.monotonicity}, a contradiction.
\end{proof}

Returning to the proof of Theorem~\ref{thm:convexity}, we note that we cannot have more than two of $\alpha,\alpha',\beta,\beta'$ in $\imx$.
Moreover, if two of these roots are imaginary, then we claim that $\alpha < \beta'$. To this end, we consider four cases:
\begin{itemize}[leftmargin=0.5cm]

\item $\alpha,\alpha'\in\wDelta^{+,\re}$, $\beta,\beta' \in \imx$.

Since $\beta,\beta' \in \convexsetre(\alpha)=\convexsetre(\alpha')$ and $\alpha < \beta$, we then have
$\alpha < \beta'$ by Lemma~\ref{lemma:weak.monotonicity}.

\item $\alpha,\alpha'\in\imx$, $\beta,\beta' \in \wDelta^{+,\re}$.

As $\alpha,\alpha' \in \convexsetre(\beta')=\convexsetre(\beta)$ and $\alpha' < \beta'$, we must have
$\alpha < \beta'$ by Lemma~\ref{lemma:weak.monotonicity}.

\item $\alpha,\beta' \in \wDelta^{+,\re}$, $\beta,\alpha' \in \imx$.

If $\beta' < \alpha$, then we would have $\alpha' < \beta' < \alpha < \beta$, which contradicts to
Lemma~\ref{lemma:weak.monotonicity} as $\alpha',\beta \in \convexsetre(\beta') = \convexsetre(\alpha)$.
Therefore, $\beta' > \alpha$.

\item $\alpha',\beta \in \wDelta^{+,\re}$, $\beta',\alpha \in \imx$.

If $\beta' < \alpha$, then we would have $\alpha' < \beta' < \alpha < \beta$, which contradicts
to Lemma~\ref{lemma:weak.monotonicity} as $\alpha,\beta' \in \convexsetre(\beta) = \convexsetre(\alpha')$
and $\chain(\beta)=\chain(\alpha')$. Therefore, $\beta' > \alpha$.

\end{itemize}
Henceforth, we assume that there is at most one imaginary root among $\alpha,\alpha',\beta,\beta'$.
If $\alpha+\beta=\alpha'+\beta'$ is real, then we can apply Lemma~\ref{lemma:gamma.difference}(a).
If $\alpha'+\beta'=k\delta$ with $k>0$ and $\alpha'<\beta'$ are real roots, then it suffices to show that $\mu=M_k(\beta')$ is $<\beta'$.
In this case $[[\sb[\SL(\alpha)],\sb[\SL(\beta)]],\sb[\SL(\beta')]] \neq 0$ by Corollary~\ref{cor:invariance.to.im.bracketing}.
Therefore, $[h_{\alpha}t^k,e_{\beta'}] \neq 0$ so that $(\alpha,\beta') \neq 0$ and we meet the requirements of Lemma~\ref{lemma:gamma.difference}(b).

Thus, we can apply Lemma~\ref{lemma:gamma.difference}, producing four different cases to consider (we note that $\gamma\ne 0$).
If $\gamma=k\delta$ with $k>0$,
then we cannot have $\alpha',\beta' < \alpha,\beta$ in all cases, due to Claim~\ref{claim:imaginary.claim.delta.movement}
if $\alpha+\beta\in \wDelta^{+,\re}$ or Claim~\ref{claim:imaginary.claim.delta.movement.imaginary} if $\alpha+\beta\in \wDelta^{+,\im}$.
If $\gamma \in \wDelta^{+,\re}$ and one of $\alpha,\alpha',\beta,\beta'$ is imaginary, then we cannot have $\alpha',\beta' < \alpha,\beta$
in all cases, due to Claims~\ref{claim:imaginary.claim.one}--\ref{claim:imaginary.claim.two}.
Finally, if $\gamma \in \wDelta^{+,\re}$ and $\alpha,\beta,\alpha',\beta'\in \wDelta^{+,\re}$, then we perform analysis
analogous to that of~\cite[Proof of Proposition 2.34]{NT}:
\begin{itemize}[leftmargin=0.5cm]

\item $\alpha + \gamma = \alpha'$, $\beta - \gamma = \beta'$.

Applying the inductive hypothesis to $\alpha'=\gamma+\alpha$ and $\beta=\gamma+\beta'$ and using
$\alpha'<\alpha$, $\beta'<\beta$, we get $\gamma < \alpha' < \alpha$ and $\beta' < \beta < \gamma$,
a contradiction with $\alpha < \beta$.

\item $\alpha - \gamma = \alpha'$, $\beta + \gamma = \beta'$.

Applying the inductive hypothesis to $\alpha=\gamma+\alpha'$ and $\beta'=\gamma+\beta$ and using
$\alpha'<\alpha$, $\beta'<\beta$, we get $\alpha' < \alpha < \gamma$ and $\gamma < \beta' < \beta$,
a contradiction with $\beta' < \alpha$.

\item $\alpha + \gamma = \beta'$, $\beta - \gamma = \alpha'$.

Applying the inductive hypothesis to $\beta'=\gamma+\alpha$ and $\beta=\gamma+\alpha'$ and using
$\beta'<\alpha$, $\alpha'<\beta$, we get $\gamma < \beta' < \alpha$ and $\alpha' < \beta < \gamma$,
a contradiction with $\alpha < \beta$.

\item $\alpha - \gamma = \beta'$, $\beta + \gamma = \alpha'$.

Applying the inductive hypothesis to $\alpha=\gamma+\beta'$ and $\alpha'=\gamma+\beta$ and using
$\beta'<\alpha$, $\alpha'<\beta$, we get $\beta' < \alpha < \gamma$ and $\gamma < \alpha' < \beta$,
a contradiction with $\alpha' < \beta'$.

\end{itemize}
Thus, in all cases, we get a contradiction with the assumed inequality $\beta'<\alpha$.
\end{proof}


\subsection{Monotonicity}\label{ssec:monotonicity}
\

This subsection generalizes~\cite[\S5.3]{AT}.

\begin{definition}\label{def:chains}
For $\alpha\in \wDelta^{+,\re}$, consider the decomposition $\alpha=\alpha'+k\delta$ with $\alpha'\in \wDelta^{+,\re}, |\alpha'|<|\delta|, k\in \BZ_{\geq 0}$
(i.e.\ $\alpha'\in \Delta^+\cup (\delta-\Delta^+)$). Define the \textbf{chain} of $\alpha$ as:
\begin{equation*}
  \chain(\alpha) = (\alpha',\alpha'+\delta,\alpha'+2\delta,\ldots).
\end{equation*}
\end{definition}

Combining Lemma~\ref{lemma:weak.monotonicity} with Theorem~\ref{thm:convexity}, we get our second key result:

\begin{proposition}\label{prop:monotonicity}
For any $\alpha \in \wDelta^{+,\re}$, the following properties are equivalent:
\begin{enumerate}

\item $\chain(\alpha)$ is monotone increasing (resp.\ decreasing),

\item $\alpha < M_1(\alpha)$ (resp.\ $\alpha > M_1(\alpha))$,

\item $\alpha$ is less than (resp.\ greater than) all elements in $\convexsetre(\alpha)$.

\end{enumerate}
In particular, each chain $\chain(\alpha)$ is monotonous.
\end{proposition}

\begin{remark}
This monotonicity heavily relies on Lemma~\ref{lemma:weak.monotonicity}, which was actually established
at the same time as the convexity of Theorem~\ref{thm:convexity}.
\end{remark}

\begin{corollary}\label{cor:monotonicity}
For $\alpha \in \wDelta^{+,\re}$ with $|\alpha| < |\delta|$, the chain $\chain(\alpha)$ is monotone increasing
(resp.\ decreasing) iff $\chain(\delta - \alpha)$ is monotone decreasing (resp.\ increasing).
\end{corollary}

\begin{proof}
If $\chain(\alpha)$ is monotone increasing, then $\alpha < M_1(\alpha)$ by Proposition~\ref{prop:monotonicity}. Since
$\alpha,\delta-\alpha\in C(M_1(\alpha))$, we have $\alpha < M_1(\alpha) < \delta - \alpha$ by Theorem~\ref{thm:convexity},
and so $\chain(\delta - \alpha)$ is monotone decreasing due to Proposition~\ref{prop:monotonicity}. Similarly,
if $\chain(\alpha)$ is monotone decreasing, then $M_1(\alpha) < \alpha$, so that  $\delta - \alpha<M_1(\alpha) < \alpha$
by Theorem~\ref{thm:convexity}, which implies that the chain $\chain(\delta - \alpha)$ is monotone increasing,
due to Proposition~\ref{prop:monotonicity}.
\end{proof}

\begin{remark}
We note that Proposition~\ref{prop:monotonicity} and Corollary~\ref{cor:monotonicity} are natural generalizations
of~\cite[Proposition 5.4, Remark 5.5]{AT} for the affine type $A^{(1)}_n$, where they were rather derived from the explicit formulas
for $\aslaws$.
\end{remark}

\begin{lemma}\label{lemma:chain.sum.monotone}
If $\alpha,\beta, \alpha + \beta \in \wDelta^{+,\re}$, and both chains $\chain(\alpha),\chain(\beta)$ are increasing
(resp.\ decreasing), then the chain $\chain(\alpha + \beta)$ is also increasing (resp.\ decreasing).
\end{lemma}

\begin{proof}
Let us first assume that $\chain(\alpha),\chain(\beta)$ are both increasing. If $M_1(\alpha) \neq M_1(\beta)$, then we have
$\alpha,\beta < \max\{M_1(\alpha),M_1(\beta)\}$ by Proposition~\ref{prop:monotonicity}. Assume without loss of generality
that $\alpha < \beta$. Then, combining Theorem~\ref{thm:convexity} and Corollary~\ref{cor:max.im.rule}, we get:
  $$\alpha < \alpha + \beta < \beta < \max\{M_1(\alpha),M_1(\beta)\}=M_1(\alpha + \beta).$$
Hence, $\chain(\alpha + \beta)$ is increasing by Proposition~\ref{prop:monotonicity}.

Let us now consider the case $M_1(\alpha) = M_1(\beta)$. Let $\alpha = (\bar{\alpha},p)$, $\beta = (\bar{\beta},s)$
with $|\bar{\alpha}|,|\bar{\beta}|<|\delta|$ and $p,s\in \BZ_{\geq 0}$. Assume without loss of generality that
$\delta - \bar{\alpha} < \delta - \bar{\beta}$. We note that $M_1(\alpha)=M_1(\bar{\alpha})=M_1(\delta-\bar{\alpha})$
and $M_1(\beta)=M_1(\bar{\beta})=M_1(\delta-\bar{\beta})$. Then both chains
$\chain(\delta - \bar{\alpha}),\chain(\delta - \bar{\beta})$ are decreasing by Corollary~\ref{cor:monotonicity}.
Combining Proposition~\ref{prop:monotonicity}, Theorem~\ref{thm:convexity}, and Corollary~\ref{cor:max.im.rule},
we thus obtain:
  $$ M_1(2\delta - (\bar{\alpha} + \bar{\beta})) \leq M_1(\alpha)=M_1(\beta) <
     \delta - \bar{\alpha} < 2\delta - (\bar{\alpha} + \bar{\beta}) < \delta - \bar{\beta}. $$
Then, the chain $\chain(2\delta - (\bar{\alpha} + \bar{\beta}))$ is decreasing by Proposition~\ref{prop:monotonicity},
and therefore the chain $\chain(\bar{\alpha} + \bar{\beta})=\chain(\alpha + \beta)$ is increasing by Corollary~\ref{cor:monotonicity}.

The case when $\chain(\alpha),\chain(\beta)$ are both decreasing follows from above. Indeed, if
$\alpha = (\bar{\alpha},p), \beta = (\bar{\beta},s)$ with $|\bar{\alpha}|,|\bar{\beta}|<|\delta|$ and $p,s\in \BZ_{\geq 0}$,
then the chains $\chain(\delta - \bar{\alpha}), \chain(\delta - \bar{\beta})$ are both increasing by
Corollary~\ref{cor:monotonicity}. Hence, $\chain(2\delta - (\bar{\alpha}  + \bar{\beta}))$ is increasing by above,
and so $\chain(\bar{\alpha} + \bar{\beta})=\chain(\alpha + \beta)$ is decreasing by Corollary~\ref{cor:monotonicity}.
\end{proof}

For $\alpha \in \wDelta^{+,\re}$, the set $O(\alpha)$ provides an upper or lower bound on $\chain(\alpha)$, due to
Remark~\ref{rem:convexity-rephrased}. However, the following result often yields better bounds on $\chain(\alpha)$:

\begin{lemma}\label{lemma:parity.same.sum}
(a) If $\alpha,\beta,\alpha + \beta \in \wDelta^{+,\re}$, with both $\chain(\alpha), \chain(\beta)$ increasing, then:
\begin{equation*}
  \chain(\alpha + \beta) < \min\{m_k(\alpha),m_k(\beta)\} \quad \textit{for all} \quad k \in \BZ_{>0}.
\end{equation*}

\noindent
(b) If $\alpha,\beta,\alpha + \beta \in \wDelta^{+,\re}$, with both $\chain(\alpha), \chain(\beta)$ decreasing, then:
\begin{equation*}
  \chain(\alpha + \beta) > \max\{M_k(\alpha),M_k(\beta)\} \quad \textit{for all} \quad k \in \BZ_{>0}.
\end{equation*}
\end{lemma}

\begin{proof}
(a) According to Lemma~\ref{lemma:chain.sum.monotone}, the chain $\chain(\alpha + \beta)$ is increasing, hence,
$\chain(\alpha + \beta) < m_k(\alpha + \beta)$ by Proposition~\ref{prop:monotonicity}. If $m_k(\alpha) \ne m_k(\beta)$,
then $m_k(\alpha + \beta) = \min\{m_k(\alpha),m_k(\beta)\}$ by Lemma~\ref{lemma:min.im.rule}, and the result follows.
If $m_k(\alpha) = m_k(\beta)$, then $\chain(\alpha),\chain(\beta) < m_k(\alpha)$ by Proposition~\ref{prop:monotonicity}.
Pick any $\gamma \in \chain(\alpha + \beta)$ with $|\gamma| \geq |\alpha + \beta|$, so that $\gamma = \alpha + \beta + p\delta$
for some $p \in \BZ_{\geq 0}$. Assuming without loss of generality that $\alpha < \beta$, we then have
$\alpha < \beta \leq \beta + p\delta < m_k(\alpha)$ by Proposition~\ref{prop:monotonicity}, so that
$\gamma < \beta +p\delta < m_k(\alpha)$ by Theorem~\ref{thm:convexity}. The claim follows as
$\chain(\alpha + \beta)$ is increasing

(b) The chain $\chain(\alpha + \beta)$ is decreasing by Lemma~\ref{lemma:chain.sum.monotone}, hence,
$M_k(\alpha + \beta)<\chain(\alpha + \beta)$ by Proposition~\ref{prop:monotonicity}. If $M_k(\alpha) \ne M_k(\beta)$,
then $M_k(\alpha + \beta) = \max\{M_k(\alpha),M_k(\beta)\}$ by Corollary~\ref{cor:max.im.rule}, and the result follows.
If $M_k(\alpha) = M_k(\beta)$, then $M_k(\alpha) < \chain(\alpha), \chain(\beta)$ by Proposition~\ref{prop:monotonicity}.
Pick any $\gamma \in \chain(\alpha + \beta)$ with $|\gamma| \geq |\alpha + \beta|$, so that $\gamma = \alpha + \beta + p\delta$
for some $p \in \BZ_{\geq 0}$. Assuming without loss of generality that $\alpha < \beta$, we then have
$M_k(\alpha) < \alpha + p\delta \leq \alpha < \beta$ by Proposition~\ref{prop:monotonicity}, so that
$M_k(\alpha) < \alpha + p\delta < \gamma$ by Theorem~\ref{thm:convexity}. The claim follows as
$\chain(\alpha + \beta)$ is decreasing.
\end{proof}

The following result is crucial for the next section:

\begin{lemma}\label{lemma:left.standard.comp.imaginary}
For any $i \in \{1,2,\ldots,|I|-1\}$ and $k \in \BZ_{>0}$, we have:
\begin{equation}\label{eq:im-ls-bound}
  \SL_{i+1}(k\delta) < \SL^{ls}_i(k\delta).
\end{equation}
\end{lemma}

\begin{proof}
As $|\SL^{ls}_i(k\delta)| < |k\delta|$, the inequality~\eqref{eq:im-ls-bound} is equivalent to
$\SL^{ls}_{i+1}(k\delta) < \SL^{ls}_i(k\delta)$ by Lemma~\ref{lemma:equiv.to.standard.fac}.
Assuming the contrary to the latter, we get $\SL^{ls}_{i+1}(k\delta) > \SL^{ls}_i(k\delta)$, since
$\SL^{ls}_{i+1}(k\delta) \ne \SL^{ls}_i(k\delta)$. Invoking Lemma~\ref{lemma:equiv.to.standard.fac} once again,
we then get $\SL_i(k\delta) < \SL^{ls}_{i+1}(k\delta) < \SL_{i+1}(k\delta)$, a contradiction.
This establishes~\eqref{eq:im-ls-bound}.
\end{proof}

We conclude this section with a couple of interesting results.

\begin{lemma}\label{lemma:suffix.increasing}
If $\alpha,\beta,\gamma \in \wDelta^{+,\re}$, $\SL(\alpha) = \SL(\gamma)\SL(\beta)$, and $\chain(\beta)$ is increasing,
then $\chain(\alpha)$ is increasing.
\end{lemma}

\begin{proof}
We have the following chain of inequalities
$\SL(\alpha + \delta) \geq \SL(\gamma)\SL(\beta + \delta) > \SL(\gamma)\SL(\beta) = \SL(\alpha)$ and
$\SL(\gamma) < \SL(\beta) < \SL(\beta + \delta)$. Thus $\chain(\alpha)$ is increasing.
\end{proof}

Henceforth, we shall often use that if $m_k(\gamma)<(k\delta,i)$, then clearly $1\leq i<|I|$.

\begin{lemma}
For any $1\leq i \leq |I|$ and $k \in \BZ_{>0}$, let $\alpha = \deg(\SL_i^{ls}(k\delta))$ and pick any decomposition
$\alpha=\beta+\gamma$ with $\beta, \gamma \in \wDelta^{+,\re}$. We cannot have both $\chain(\beta), \chain(\gamma)$ decreasing.
Moreover, if $\chain(\beta), \chain(\gamma)$ are both increasing, then exactly one of $m_k(\beta),m_k(\gamma)$ is equal
to $m_k(\alpha)$, and the other must be larger than $m_k(\alpha)$.
\end{lemma}

\begin{proof}
If both chains $\chain(\beta),\chain(\gamma)$ were decreasing, then so would be $\chain(\alpha)$ by Lemma~\ref{lemma:chain.sum.monotone}.
On the other hand, $\alpha<(k\delta,i)=m_k(\alpha)$ by Corollary~\ref{cor:mk.im.factor}, which implies
that $\chain(\alpha)$ is increasing by Proposition~\ref{prop:monotonicity}, a contradiction.

Assume now that both chains $\chain(\beta),\chain(\gamma)$ are increasing. Then we have
$\chain(\alpha)\leq \min\{m_k(\beta),m_k(\gamma)\}$ by Lemma~\ref{lemma:parity.same.sum}.
If $\min\{m_k(\beta),m_k(\gamma)\} < (k\delta,i)$, then the above would yield $\alpha < (k\delta,i+1)$,
a contradiction with Lemma~\ref{lemma:left.standard.comp.imaginary}. According to Lemma~\ref{lemma:min.im.rule}
and Corollary~\ref{cor:mk.im.factor}, we have $(k\delta,i)=m_k(\alpha)\geq \min\{m_k(\beta),m_k(\gamma)\}$.
Therefore, it suffices to show that $m_k(\beta) = m_k(\gamma) = (k\delta,i)$ is not possible. Assume the
contrary, and without loss of generality we shall assume that $\beta<\gamma$. Due to Theorem~\ref{thm:convexity},
we then get $\gamma>\alpha=\deg(\SL_i^{ls}(k\delta))$, which implies $\gamma>(k\delta,i)$ by
Lemma~\ref{lemma:equiv.to.standard.fac}. But then $\chain(\gamma)$ is decreasing by Proposition~\ref{prop:monotonicity},
a contradiction.
\end{proof}


\section{Imaginary affine standard Lyndon words}\label{sec:imaginary.words}

In this section, we investigate imaginary $\aslaws$ and relations among them.
Our analysis is crucially based on the results of Section~\ref{sec:keyresults}.


\subsection{Compatibility of flags and basic inequalities}
\

We start with the important compatibility of flags $\{\spanset_{\bullet}^{k}\}_{k\geq 1}$ from~\eqref{eq:flag}:

\begin{proposition}\label{prop:spanset.equiv}
For any $i \in \{0,1,\ldots,|I|\}$ and $k \in \BZ_{>0}$, we have:
\begin{equation*}
  \spanset_{i}^{k+1} = \spanset_i^kt.
\end{equation*}
Additionally, if $i > 0$ and $w = \SL_i^{ls}((k+1)\delta)$, then $m_k(\deg(w)) = (k\delta,i)$.
\end{proposition}

\begin{proof}
The proof proceeds by induction on $i$. The base case $i=0$ is obvious. For the step of induction,
let us assume that both claims hold for all $j < i$ and any $k\in \BZ_{>0}$. Let us first verify that $m_k(\deg(w)) = (k\delta,i)$.
If $m_k(\deg(w)) > (k\delta,i)$, then we actually have
$[\sb[w],\sb[\SL((k+1)\delta - \deg(w))]] \sim h_{w}t^{k+1} \in \spanset_{i-1}^{k}t$ by Lemma~\ref{lemma:mk}.
As $\spanset_{i-1}^{k}t =  \spanset_{i-1}^{k+1}$ by the inductive hypothesis, we get a contradiction with
Proposition~\ref{prop:general.bracketing}(b). If $m_k(\deg(w)) < (k\delta,i)$, then
$w < \SL(m_k(\deg(w))) \leq  \SL_{i+1}(k\delta)<\SL_{i}^{ls}(k\delta)$ by Proposition~\ref{prop:monotonicity}
and Lemma~\ref{lemma:left.standard.comp.imaginary}. Let $\alpha=\deg(\SL^{ls}_i(k\delta))$, so that
$\SL(\alpha)=\SL^{ls}_i(k\delta)>\SL_i((k+1)\delta)$ by Lemma~\ref{lemma:equiv.to.standard.fac}.
As $|\alpha|<|(k+1)\delta|$, we thus have $\SL(\alpha)\SL((k+1)\delta - \alpha) > \SL_i((k+1)\delta)$.
We also have $\alpha<(k\delta,i)<(k+1)\delta-\alpha$ by Proposition~\ref{prop:monotonicity}.
Then $[\sb[\SL(\alpha)],\sb[\SL((k+1)\delta - \alpha)]]\sim h_{\alpha}t^{k+1} \in \spanset_{i-1}^{k+1}$ by
Corollary~\ref{cor:im.upper.limit}. But this contradicts the inductive hypothesis as
$h_{\alpha}t^k \in \spanset_{i}^k\backslash\spanset_{i-1}^{k}$ by Corollary~\ref{cor:mk.im.factor} and Lemma~\ref{lemma:mk}.
This establishes $m_k(\deg(w)) = (k\delta,i)$.

We thus get $[\sb[w],\sb[\SL((k+1)\delta - \deg(w))]] \in \spanset_{i}^{k}t\backslash\spanset_{i-1}^kt$, due to Lemma~\ref{lemma:mk}.
On the other hand, $[\sb[w],\sb[\SL((k+1)\delta - \deg(w))]] \in \spanset_{i}^{k+1}\backslash\spanset_{i-1}^{k+1}$ by
Proposition~\ref{prop:general.bracketing}. As $\spanset_{i-1}^{k+1} = \spanset_{i-1}^{k}t$ by the inductive hypothesis,
and quotients $\spanset_{i}^{k}t/\spanset_{i-1}^kt, \spanset_{i}^{k+1}/\spanset_{i-1}^{k+1}$ are one-dimensional, we obtain
$\spanset_i^{k+1} = \spanset_i^kt$. This completes the inductive step.
\end{proof}

\begin{remark}
We note that there is no periodicity of bracketings $\{\sb[\SL_i(k\delta)]\}_{k\geq 1}$ for $i\ne 1$.
For the affine type $A^{(1)}_n$, the explicit formulas of~\cite[(4.67, 4.73)]{AT} possessed
a periodicity for $k\geq 2$, but not for $k=1$. In particular, Proposition~\ref{prop:spanset.equiv} can be used to
simplify some arguments from the proof of~\cite[Theorem 4.7]{AT}. We also note that an alternative proof
of~\cite[Theorem 4.7]{AT} is presented in~\cite[\S7]{ET}.
\end{remark}

\begin{corollary}\label{cor:equiv.mk.Mk}
For any $k,p \in \BZ_{>0}$ and $\alpha \in \wDelta^{+,\re}$, we have:
\begin{gather*}
  M_k(\alpha) = (k\delta,i) \ \ \textit{if and only if} \ \ M_p(\alpha) = (p\delta,i), \\
  m_k(\alpha) = (k\delta,i) \ \ \textit{if and only if} \ \ m_p(\alpha) = (p\delta,i).
\end{gather*}
\end{corollary}

\begin{proof}
This follows from the above result combined with Lemmas~\ref{lemma:Mk} and~\ref{lemma:mk}.
\end{proof}

Combining this with Corollary~\ref{cor:mk.im.factor}, we obtain:

\begin{corollary}\label{cor:im.splitting.convex.set}
For any $k,p \in \BZ_{>0}$, $i \in \{1,2,\ldots,|I|\}$, and any splitting $\SL_i(k\delta) = uv$ with $u,v \in \SL$,
we have: $m_p(\deg(u)) = (p\delta,i)=m_p(\deg(v))$.
\end{corollary}

The following result will be useful later:

\begin{lemma}\label{lemma:lifting.imaginary.left.standard}
For any $i \in \{1,2,\ldots,|I|\}$ and $k\in \BZ_{>0}$, let $\alpha = \deg(\SL^{ls}_i(k\delta))$. Then:
\begin{gather*}
  \SL(\alpha + \delta) \leq \SL^{ls}_i((k+1)\delta).
\end{gather*}
\end{lemma}

\begin{proof}
Assume by contradiction that $\SL(\alpha + \delta) > \SL^{ls}_i((k+1)\delta)$. Then $\SL(\alpha + \delta) > \SL_i((k+1)\delta)$
by Lemma~\ref{lemma:equiv.to.standard.fac}. Combining Proposition~\ref{prop:monotonicity} and Corollary~\ref{cor:monotonicity},
we get $\alpha + \delta < M_1(\alpha) < k\delta-\alpha$, so that $\SL(\alpha + \delta)\SL(k\delta - \alpha)$ is Lyndon by
Lemma~\ref{lemma:lyndon}. We also note that
  $[\sb[\SL(\alpha + \delta)],\sb[\SL(k\delta - \alpha)]] \in \spanset_{i}^{k+1} \backslash\spanset_{i-1}^{k+1}$
by Lemma~\ref{lemma:mk} and Proposition~\ref{prop:spanset.equiv}. But this contradicts to Corollary~\ref{cor:im.upper.limit},
which applies as $\SL_i((k+1)\delta)<\SL(\alpha + \delta)<\SL(\alpha + \delta)\SL(k\delta - \alpha)$. This completes the proof.
\end{proof}

The following result generalizes~\cite[Remark 5.6]{AT} to all types:

\begin{lemma}\label{cor:imaginary.words.decreasing}
For any $i \in \{1,2,\ldots,|I|\}$, we have:
\begin{equation*}
  \SL_i(\delta) > \SL_i(2\delta) > \SL_i(3\delta) > \dots
\end{equation*}
\end{lemma}

\begin{proof}
Pick any $k \in \BZ_{>0}$. We have
$m_k(\deg(\SL_i^{ls}((k+1)\delta))) = (k\delta,i)$ by Corollary~\ref{cor:im.splitting.convex.set}.
As $\SL_i^{ls}((k+1)\delta) < \SL_i((k+1)\delta)=\SL(m_{k+1}(\deg(\SL_i^{ls}((k+1)\delta))))$, we have
$\SL_i^{ls}((k+1)\delta) < \SL_i(k\delta)=\SL(m_k(\deg(\SL_i^{ls}((k+1)\delta))))$ due to Proposition~\ref{prop:monotonicity}.
We thus conclude that $\SL_{i}((k+1)\delta) < \SL_i(k\delta)$ due to Lemma~\ref{lemma:equiv.to.standard.fac}.
\end{proof}

The following result is important for the rest of this section:

\begin{lemma}\label{lemma:length.standfac.imaginary}
For any $i \in \{1,2,\ldots,|I|\}$ and $k \in \BZ_{>0}$, let $\beta = \deg(\SL^{ls}_i(k\delta))$. Then $|\beta| > |(k-1)\delta|$.
\end{lemma}

\begin{proof}
It suffices to consider the case $k > 1$. Let $\alpha = \deg(\SL^{ls}_i((k-1)\delta))$, so that
$\SL(\alpha + \delta) \leq \SL(\beta)$ by Lemma~\ref{lemma:lifting.imaginary.left.standard}.
We also note that $\SL(\beta) < \SL_i((k-1)\delta)$ due to Proposition~\ref{prop:monotonicity},
since $\SL(\beta) < \SL_i(k\delta)=\SL(m_k(\beta))$ and $\SL_i((k-1)\delta)=\SL(m_{k-1}(\beta))$
by Corollary~\ref{cor:equiv.mk.Mk}. Invoking Proposition~\ref{prop:monotonicity} again, we also note that
$\alpha < \alpha + \delta$. Combining all the above, we obtain:
  $$\SL^{ls}_i((k-1)\delta)=\SL(\alpha) < \SL(\alpha+\delta) \leq \SL(\beta) < \SL_i((k-1)\delta).$$
Therefore, $|\beta| \geq |(k-1)\delta|$ by Lemma~\ref{lemma:equiv.to.standard.fac}, which implies the result as
$|\beta| \ne |(k-1)\delta|$ by Corollary~\ref{cor:imaginary.no.im.splitting}.
\end{proof}

We conclude this subsection with the following important observation:

\begin{corollary}\label{cor:imaginary.suffix.prefix.improved}
For any standard Lyndon word $w$ and any splitting $w=uv$, we cannot have
$u = \SL_i(k\delta)$ or $v = \SL_i(k\delta)$ for some $i$ and $k > 1$.
\end{corollary}

\begin{proof}
First, we claim that it suffices to assume that both $u,v \in \SL$. Indeed, if $u = \SL_i(k\delta)$ but $v\notin \LL$,
then there is a prefix $w'=uv'$ of $uv$ with $w',v'\in \LL$ by Lemma~\ref{lemma:seq.right.word.Lyndon}. Similarly,
if $v=\SL_i(k\delta)$ but $u\notin \LL$, then there is a suffix $w'=u'v$ of $uv$ with $w',u'\in \LL$ by
Lemma~\ref{lemma:seq.left.word.Lyndon}. In fact, $w',u',v'\in \SL$ as subwords of $w\in \SL$.

The result follows from Corollary~\ref{cor:imaginary.no.im.splitting} if $\deg(w) \in \imx$. Let us now assume that
$\deg(w) \in \wDelta^{+,\re}$. We have $M_k(\alpha) < M_1(\alpha)$ for any $\alpha\in \wDelta^{+,\re}$ and $k>1$, due
to Corollary~\ref{cor:equiv.mk.Mk} and Lemma~\ref{cor:imaginary.words.decreasing}. Assume first that $v=\SL_i(k\delta)$
for some $i$ and $k>1$. Then $v = \SL(M_k(\deg(u)))$ by Corollary~\ref{cor:imaginary.suffix.prefix}. Thus
$\chain(\deg(u)) < M_1(\deg(u))$ and so $\SL(\deg(w) - (k-1)\delta) \geq u\SL(M_1(\deg(u)))$ by the generalized Leclerc
algorithm. Then $\SL(\deg(w) - (k-1)\delta) \geq u\SL(M_1(\deg(u))) > u\SL(M_k(\deg(u))) = w$, which contradicts to
Proposition~\ref{prop:monotonicity} as $u<w<v=\SL(M_k(\deg(u)))=\SL(M_k(\deg(w)))$ by Theorem~\ref{thm:convexity}.
The case $u=\SL_i(k\delta)$ for some $i$ and $k>1$ is treated similarly. First, we note that $u = \SL(M_k(\deg(v)))$
by Corollary~\ref{cor:imaginary.suffix.prefix}. Then, we have $\SL(\deg(w) - \delta) > w > \SL(M_1(\deg(v)))=\SL(M_1(\deg(w)))$,
in accordance with Proposition~\ref{prop:monotonicity}. As $[\sb[\SL(M_1(\deg(v)))],\sb[\SL(\deg(w)-\delta)]] \neq 0$,
we have $\SL(M_1(\deg(v)))\SL(\deg(w)-\delta)\leq w$ by the generalized Leclerc algorithm. This contradicts to
$u=\SL(M_k(\deg(v)))<\SL(M_1(\deg(v)))$ above.
\end{proof}


\subsection{Special orders}
\

In this subsection, we obtain explicit formulas for all imaginary words $\SL_i(k\delta)$ when
$\delta$ contains only one instance of the smallest simple root, denoted by $\alpha_\varepsilon$.

\begin{remark}
This applies to any order in the affine type $A^{(1)}_n$, and to an arbitrary affine type and the orders on the alphabet $\wI$ with
$0\in \wI$ being the smallest letter.
\end{remark}

We start with the following simple observation:

\begin{lemma}\label{lemma:delta.stand.one.letter}
If $\alpha_\varepsilon$ occurs only once in $\delta$, then for any $i \in \{1,2,\ldots,|I|\}$, we have
$|\SL^{rs}_i(\delta)| = 1$ and all $\SL_i(\delta)$ end with different letters.
\end{lemma}

We also have a simple criteria for when chains are increasing:

\begin{lemma}\label{lemma:monotonicity.smallest.once}
If $\alpha_\varepsilon$ occurs only once in $\delta$, then for any $\alpha \in \wDelta^{+,\re}$ with $|\alpha| < |\delta|$:
  $$ \chain(\alpha) \ \textit{is\ increasing\ if\ and\ only\ if} \ \  \alpha \ \textit{contains\ } \alpha_\varepsilon. $$
\end{lemma}

\begin{proof}
To prove the ``only if'' direction, it suffices to show that if $\alpha$ does not contain $\alpha_\varepsilon$
then $\chain(\alpha)$ is decreasing. Comparing the first letters, we get $M_1(\alpha)\leq \SL_1(\delta) <\SL(\alpha)$, and
so  $\chain(\alpha)$ is decreasing by Proposition~\ref{prop:monotonicity}.
The ``if'' direction follows from above and Corollary~\ref{cor:monotonicity},
since if $\alpha$ contains $\alpha_\varepsilon$ then $\delta - \alpha\in \wDelta^{+,\re}$ does not contain $\alpha_\varepsilon$.
\end{proof}

We are now ready to establish the structure of all $\SL_i(k\delta)$ in the present setup:

\begin{theorem}\label{thm:imaginary-SL-orderI}
If $\alpha_\varepsilon$ occurs only once in $\delta$, then we have:
\begin{equation*}
  \SL_i(k\delta) = \SL^{ls}_{i}(\delta) \underbrace{\SL(M_1(\gamma_i))}_{k-1 \text{ times}}\SL_i^{rs}(\delta),\
  \quad \SL_i^{ls}(k\delta) = \SL^{ls}_{i}(\delta) \underbrace{\SL(M_1(\gamma_i))}_{k-1 \text{ times}},
\end{equation*}
for any $k \in \BZ_{>0}$, $i \in \{1,2,\ldots,|I|\}$, where $\gamma_i := \deg(\SL^{ls}_i(\delta))$.
\end{theorem}

\begin{proof}
The proof proceeds by induction on $k$, the base case $k=1$ is obvious. As per the inductive step, let us assume
the validity for $k-1$. Let $\alpha = \deg(\SL_i^{ls}((k-1)\delta))$, so that
$b = \SL^{rs}_i((k-1)\delta) = \SL^{rs}_i(\delta)$, which is a single letter by Lemma~\ref{lemma:delta.stand.one.letter}.

Let $\beta = \deg(\SL_i^{ls}(k\delta))$, so that $\alpha + \delta \leq \beta$ by
Lemma~\ref{lemma:lifting.imaginary.left.standard}. By Corollary~\ref{cor:im.splitting.convex.set}, we have
$m_{k-1}(\alpha) = m_{k-1}(\beta) = ((k-1)\delta,i)$. Since $\chain(\alpha),\chain(\beta)$ are increasing by
Theorem~\ref{thm:convexity} and Proposition~\ref{prop:monotonicity}, we have $\alpha < \beta < ((k-1)\delta,i)$.
Therefore, we have $\SL(\beta) = \SL(\alpha)w$ for some nonempty word $w$ satisfying $w<\SL^{rs}_i((k-1)\delta)$.

Let $w = w_1 \ldots w_n$ be the canonical factorization. We claim that $w_1$ contains the smallest letter $\varepsilon$.
If not, then the first letter $\iota$ of $w_1$ is bigger than $\varepsilon$, so that $\SL(\alpha)\iota$ is (standard) Lyndon
by Lemma~\ref{lemma:lyndon}, but then $(k-1)\delta - \alpha_b + \alpha_\iota=\deg(\SL(\alpha)\iota)\in \wDelta^+$ implies
$\iota=b$ and so $\SL_{i}((k-1)\delta)=\SL(\alpha)b$ is a prefix of $\SL_i(k\delta)$, a contradiction with
Corollary~\ref{cor:imaginary.words.decreasing}.
Thus $w_1$ contains $\varepsilon$, and so $\iota=\varepsilon$ as $w_1\in \LL$. Since $w$ contains at most
one $\varepsilon$, we get $n=1$, as otherwise we have a contradiction with $w_1\geq w_2$.

Thus, $n=1$ so that $w$ is Lyndon. Note that $|\deg(w)| \leq |\delta|$ as $|\alpha| = |(k-1)\delta| - 1$.
If $|\deg(w)| < |\delta|$, then since $w$ contains the smallest letter $\varepsilon$, the word $\SL(\delta - \deg(w))$
does not and so $\SL(\alpha)w\SL(\delta - \deg(w))$ is Lyndon by Lemma~\ref{lemma:lyndon}. This implies
  $\SL(\alpha + \delta)\geq \SL(\alpha)w\SL(\delta - \deg(w)) > \SL(\alpha)w=\SL(\beta)$,
a contradiction with Lemma~\ref{lemma:lifting.imaginary.left.standard}. Therefore, $\deg(w) = \delta$.
Then $\SL_i^{ls}(k\delta) = \SL(\alpha + \delta) = \SL(\alpha)\SL(M_1(\alpha))$ by Corollary~\ref{cor:imaginary.suffix.prefix}.
Combining this with the inductive hypothesis and the equality $M_1(\alpha)=M_1(\gamma_i)$ completes the inductive step.
\end{proof}

The following result pertains to a slight generalization of the present setup:

\begin{lemma}\label{lemma:right.stand.fac.imaginary.chain}
Fix $i \in \{1,2,\ldots,|I|\}$ and let $\alpha = \deg(\SL_i^{ls}(\delta))$. If
\begin{equation*}
  \SL_i(k\delta) = \SL(\alpha)\underbrace{\SL(M_1(\alpha))}_{k-1 \text{ times}}\SL(\delta - \alpha)
  \quad \textit{for all} \quad k \in \BZ_{>0},
\end{equation*}
then
\begin{equation*}
  \SL((\delta - \alpha) + p\delta) = \underbrace{\SL(M_1(\alpha))}_{p \text{ times}}\SL(\delta - \alpha)
  \quad \textit{for all} \quad p \in \BZ_{\geq 0}.
\end{equation*}
\end{lemma}

\begin{proof}
We prove the result by induction on $p>0$ (as the $p=0$ case is obvious). For the base case $p=1$, we have
$\SL(M_1(\alpha)) < \SL(\delta - \alpha)$ by Theorem~\ref{thm:convexity}, cf.~Remark~\ref{rem:convexity-rephrased}.
Thus $\SL(M_1(\alpha))\SL(\delta - \alpha)$ is a Lyndon factor (see Lemma~\ref{lemma:lyndon}) of a standard Lyndon
word $\SL_i(2\delta)$, so that $\SL(M_1(\alpha))\SL(\delta - \alpha)=\SL((\delta - \alpha) + \delta)$.

For the inductive step, let us assume that the result holds for $p-1$, so that
$\SL(M_1(\alpha)) < \SL((\delta - \alpha) + (p-1)\delta)$. Then $\SL(M_1(\alpha))\SL((\delta - \alpha) + (p-1)\delta)$
is a Lyndon factor of a standard Lyndon word $\SL_i((p+1)\delta)$. Hence, similarly to the base case, we conclude that
$\SL((\delta - \alpha) + p\delta) = \SL(M_1(\alpha))\SL((\delta - \alpha) + (p-1)\delta)$. Combining this with the
inductive hypothesis completes the step of induction.
\end{proof}

The above result implies the following corollary in the present setup:

\begin{corollary}
If $\alpha_\varepsilon$ occurs only once in $\delta$, then for any simple root $\alpha_i$ with $i\ne \varepsilon$:
\begin{equation*}
  \SL(\alpha_i + k\delta) = \underbrace{\SL(M_1(\alpha_i))}_{k\text{ times}}i
  \quad \textit{for all} \quad k \in \BZ_{\geq 0}.
\end{equation*}
\end{corollary}

\begin{proof}
There are $|\wI|=|I| + 1$ simple roots and according to Lemma~\ref{lemma:delta.stand.one.letter}, all of them besides
$\alpha_{\varepsilon}$ appear as $\deg(\SL_j^{rs}(\delta))$ for some $j$. Thus, the result follows from
Lemma~\ref{lemma:right.stand.fac.imaginary.chain}, which can be applied due to Theorem~\ref{thm:imaginary-SL-orderI}.
\end{proof}


\subsection{General orders.}
\

We shall now discuss the case of general orders on $\wI$.

\begin{lemma}\label{lemma:weak.form.left.standard.imaginary}
For any $i \in \{1,2,\ldots,|I|\}$ and $k > 1$, we have:
\begin{equation*}
  \SL^{ls}_i(k\delta) = \SL_i^{ls}((k-1)\delta)w \quad \textit{for\ some\ word} \quad w\ne \emptyset.
\end{equation*}
\end{lemma}

\begin{proof}
According to Lemma~\ref{lemma:lifting.imaginary.left.standard}, we have
$\SL(\deg(\SL^{ls}_i((k-1)\delta)) + \delta) \leq \SL^{ls}_i(k\delta)$.
Combining Corollary~\ref{cor:im.splitting.convex.set} and Proposition~\ref{prop:monotonicity},
we see that $\chain(\deg(\SL^{ls}_i((k-1)\delta)))$ is increasing, so that
$\SL_i^{ls}((k-1)\delta) < \SL(\deg(\SL^{ls}_i((k-1)\delta)) + \delta) \leq \SL^{ls}_i(k\delta)$.
Additionally, we have $\SL^{ls}_i(k\delta) < \SL_i(k\delta) < \SL_i((k-1)\delta)$ by
Corollary~\ref{cor:imaginary.words.decreasing}. Hence, we have $\SL^{ls}_i(k\delta) = \SL^{ls}_i((k-1)\delta)w$
for some nonempty word $w<\SL^{rs}_i((k-1)\delta)$.
\end{proof}

We now prove several technical results that will ultimately yield Proposition~\ref{prop:SL.1.form}.

\begin{lemma}\label{lemma:lifting}
Fix any $i \in \{1,2,\ldots,|I|\}$, $k \geq 1$, $\alpha,\beta \in \wDelta^{+,\re}$ such that $\alpha + \beta \in \wDelta^{+,\re}$,
$\alpha \geq \deg(\SL_i^{ls}(k\delta))$, $m_k(\alpha) = (k\delta,i)$, $|\beta| < |2\delta|$, $|\alpha + \beta|< |(k+1)\delta|$,
$\chain(\alpha)$ is increasing, $\chain(\alpha + \beta)$ is increasing, and $\alpha + \beta <  (k\delta,i)$.
Then $\alpha + \delta > \alpha + \beta$.
\end{lemma}

\begin{proof}
Suppose to the contrary that $\alpha + \beta > \alpha + \delta$. As $\alpha + \delta > \alpha$ we get
$\alpha+\beta>\alpha$, and so $\beta > \alpha$ by Theorem~\ref{thm:convexity}.
As $(k\delta,i) > \alpha+\beta > \deg(\SL_i^{ls}(k\delta))$, we also get $|\alpha + \beta| \geq |k\delta|$
by Lemma~\ref{lemma:equiv.to.standard.fac}.

First, assume that $\chain(\beta)$ is increasing. If $|\beta| > |\delta|$, then $\beta - \delta < \beta$.
Considering the splitting $\alpha + \beta = (\beta - \delta) + (\alpha + \delta)$, we get
$\alpha + \delta < \alpha + \beta < \beta - \delta$ by Theorem~\ref{thm:convexity}. Combining
$\alpha<\alpha+\beta<\beta-\delta$ with $|\beta - \delta| < |\delta|$, we obtain $\beta - \delta > (k\delta,i)$
by Lemma~\ref{lemma:equiv.to.standard.fac}. As $\chain(\beta)$ is increasing, we thus have
$m_k(\beta-\delta) > (k\delta,i) = m_k(\alpha)$ by Proposition~\ref{prop:monotonicity}.
Hence $m_k(\alpha + (\beta - \delta)) = (k\delta,i)$ by Lemma~\ref{lemma:min.im.rule}. As $\chain(\alpha+\beta)$ is increasing,
we obtain $\alpha + (\beta - \delta)<(k\delta,i)$. On the other hand, $\alpha<\beta-\delta$ implies
$\alpha < \alpha + (\beta - \delta)$ by Theorem~\ref{thm:convexity}, which together with $|\alpha + (\beta - \delta)| < |k\delta|$
yields $\alpha + (\beta-\delta) > (k\delta,i)$ by Lemma~\ref{lemma:equiv.to.standard.fac}, a contradiction.
If $|\beta| < |\delta|$, we then have $\beta > (k\delta,i)$ by Lemma~\ref{lemma:equiv.to.standard.fac}, as $\beta > \alpha$.
Since $\chain(\beta)$ is increasing, we thus get $m_k(\beta) > (k\delta,i)$ by Proposition~\ref{prop:monotonicity}.
According to Corollary~\ref{cor:monotonicity}, the chain $\chain(\delta - \beta)$ is decreasing and is greater than
$(k\delta,i)$, so that $\delta - \beta > (k\delta,i) > \alpha + \delta$. Applying Theorem~\ref{thm:convexity} to the
decomposition $(\alpha + \beta) + (\delta - \beta) = \alpha + \delta$, we finally obtain $\alpha + \delta > \alpha +\beta$,
a contradiction.

Let us now assume that $\chain(\beta)$ is decreasing. If $|\beta| > |\delta|$, then $\beta-\delta > \beta$. We must have
$\beta-\delta > \alpha + \beta -\delta > \alpha$ by Theorem~\ref{thm:convexity}.
Since $\alpha + \beta < (k\delta,i)$ and $\chain(\alpha+\beta)$ is increasing, we have
$\alpha + \beta - \delta < (k\delta,i)$. As $|\alpha + \beta -\delta| < |k\delta|$, we then obtain
$\alpha + \beta - \delta < \alpha$ by Lemma~\ref{lemma:equiv.to.standard.fac}, which contradicts to
$\beta - \delta > \alpha + \beta - \delta > \alpha$ above.
If $|\beta| < |\delta|$, then as $|\alpha + \beta - \delta| < |k\delta|$, we again conclude that
$\alpha + \beta - \delta < \alpha$ by Lemma~\ref{lemma:equiv.to.standard.fac}.
As $\alpha = (\alpha + \beta - \delta) + (\delta - \beta)$, we obtain $\alpha  < \delta - \beta$ by
Theorem~\ref{thm:convexity}. Applying Lemma~\ref{lemma:equiv.to.standard.fac} once again, we get
$(k\delta,i) < \delta - \beta$. On the other hand, $(k\delta,i)>\alpha$ and so $(k\delta,i)>\alpha + \delta$
by Proposition~\ref{prop:monotonicity}.
Thus $\delta - \beta > \alpha + \delta$ and applying Theorem~\ref{thm:convexity} to the decomposition
$\alpha + \delta = (\alpha + \beta) + (\delta - \beta)$, we finally obtain $\alpha + \delta > \alpha + \beta$,
a contradiction.
\end{proof}

\begin{corollary}\label{cor:lifting}
Assume that all the conditions of Lemma~\ref{lemma:lifting} hold, as well as $\beta > \alpha$ and
$\SL(\alpha+\beta)$ is a prefix of $\SL_i^{ls}((k+1)\delta)$.
Then $|\beta| < |\delta|$ and $m_k(\beta) > (k\delta,i)$.
\end{corollary}

\begin{proof}
Let us first prove that $|\beta| < |\delta|$. If not, then $|\beta| > |\delta|$ as $\beta$ is real.
Since $\SL(\alpha+\delta)>\SL(\alpha+\beta)$ by the previous lemma, $|\SL(\alpha+\delta)|<|\SL(\alpha+\beta)|$,
and $\SL(\alpha+\beta)$ is a prefix of $\SL_i^{ls}((k+1)\delta)$, we obtain $\SL(\alpha+\delta)>\SL_i^{ls}((k+1)\delta)$.
Therefore, $\alpha + \delta > ((k+1)\delta,i)$ by Lemma~\ref{lemma:equiv.to.standard.fac}.
As $m_k(\alpha) = (k\delta,i)$ and $\chain(\alpha)$ is increasing, we have
$\alpha+\delta<m_{k+1}(\alpha+\delta)= ((k+1)\delta,i)$ by Proposition~\ref{prop:monotonicity} and
Corollary~\ref{cor:equiv.mk.Mk}, a contradiction with above. Thus, indeed we have $|\beta| < |\delta|$.

Assume that $m_k(\beta) < (k\delta,i)$. Then $m_k(\alpha + \beta) < (k\delta,i)$ by Lemma~\ref{lemma:min.im.rule}.
We then get $\SL_i^{ls}(\delta) > \SL_{i+1}(\delta) \geq \SL_{i+1}(k\delta) \geq \SL(m_k(\alpha + \beta))$, due to
Lemma~\ref{lemma:left.standard.comp.imaginary} and Corollary~\ref{cor:imaginary.words.decreasing}. Since $\SL(\alpha + \beta)$
is a prefix of $\SL_i^{ls}((k+1)\delta)$ and $\alpha+\beta > \alpha \geq \SL_i^{ls}(k\delta) \geq \SL_i^{ls}(\delta)$
by repeated application of Lemma~\ref{lemma:weak.form.left.standard.imaginary}, we see that
$\SL_i^{ls}(\delta)$ is a prefix of $\SL(\alpha+\beta)$. This implies $\SL(m_k(\alpha + \beta)) < \SL_i^{ls}(\delta) < \SL(\alpha + \beta)$.
On the other hand, $\alpha + \beta < m_k(\alpha + \beta) < (k\delta,i)$ as $\chain(\alpha + \beta)$ is increasing,
a contradiction with the above. Thus $m_k(\beta) \geq (k\delta,i)$.

If $m_k(\beta) = (k\delta,i)$, then $\chain(\beta)$ is decreasing, as otherwise $\beta < (k\delta,i)$ by Proposition~\ref{prop:monotonicity},
which together with $|\beta| < |\delta|$ and Lemma~\ref{lemma:equiv.to.standard.fac} implies $\beta < \deg(\SL^{ls}_i(k\delta)) < \alpha$,
a contradiction. Then, we have $\delta - \beta < (k\delta,i) < \beta$ by
Proposition~\ref{prop:monotonicity} and Corollary~\ref{cor:monotonicity}.
Thus $\delta - \beta < \deg(\SL_i^{ls}(k\delta)) < \alpha < \alpha + \delta$, where the first inequality follows from
Lemma~\ref{lemma:equiv.to.standard.fac} while the last one holds as $\chain(\alpha)$ is increasing.
Applying Theorem~\ref{thm:convexity} to the decomposition $(\alpha + \beta) + (\delta-\beta) = \alpha+\delta$,
we get $\alpha + \delta < \alpha + \beta$, a contradiction with Lemma~\ref{lemma:lifting}.
Therefore, $m_k(\beta) > (k\delta,i)$.
\end{proof}

In the next few results, we investigate prefixes of $\SL_i(k\delta)$.

\begin{lemma}\label{lemma:no.imaginary.prefix}
For any $i \in \{1,2,\ldots,|I|\}$ and $k \in \BZ_{>0}$, no proper prefix of $\SL_i(k\delta)$ can be
an imaginary affine standard Lyndon word.
\end{lemma}

\begin{proof}
Assume the contradiction, that is, $\SL_j(p\delta)$ is a prefix of $\SL_i(k\delta)$ for $0<p<k$ and some $i,j$.
As $\SL_j(p\delta)>\SL_j(k\delta)$ by Corollary~\ref{cor:imaginary.words.decreasing}, we obtain $j > i$. On the
other hand, $\SL^{ls}_i(\delta)$ is a prefix of $\SL_i(k\delta)$ by repeated application of
Lemma~\ref{lemma:weak.form.left.standard.imaginary}.
According to Lemma~\ref{lemma:left.standard.comp.imaginary} and Corollary~\ref{cor:imaginary.words.decreasing},
we have $\SL^{ls}_i(\delta) > \SL_j(\delta) > \SL_j(p\delta)$.
This contradicts to $\SL^{ls}_i(\delta)$ being a shorter prefix of $\SL_i(k\delta)$ than $\SL_j(p\delta)$.
\end{proof}

\begin{lemma}\label{lemma:prefix.lyndon.imaginary}
For any $i \in \{1,2,\ldots,|I|\}$, $k \in \BZ_{>0}$, and any splitting $\SL_i(k\delta) = \ell w$ with
$\ell \in \SL$ and $|\ell| \geq |\SL^{ls}_i(\delta)|$, the chain $\chain(\deg(\ell))$ is increasing.
Moreover, we have $m_k(\deg(\ell)) = (k\delta,i)$.
\end{lemma}

\begin{proof}
Fix $i,k$. By Lemma~\ref{lemma:no.imaginary.prefix} there are no imaginary affine standard Lyndon proper prefixes
of $\SL_i(k\delta)$. We will now perform two rounds of induction: the first will show that $\chain(\deg(\ell))$
is increasing and $m_k(\deg(\ell)) \geq (k\delta,i)$, while the second will then show that actually $m_k(\deg(\ell)) = (k\delta,i)$.

We first do a decreasing induction on the length of $\ell$. The base case $\ell=\SL^{ls}_i(k\delta)$ is clear,
due to Corollary~\ref{cor:mk.im.factor} and Proposition~\ref{prop:monotonicity}.
By the repeated application of Lemma~\ref{lemma:weak.form.left.standard.imaginary}, we note that $\SL_i^{ls}(\delta)$
is a prefix of all $\ell$ since $|\ell| \geq |\SL_i^{ls}(\delta)|$. As per the induction step,
we assume the induction hypothesis holds for all Lyndon prefixes $u$ with $|u| > |\ell|$. Pick such shortest $u$,
so that $u = \ell v$ with $\ell = u^{ls}, v = u^{rs}$. We have $v > \SL_i(k\delta)$ since $v > u$ (as $u$ is Lyndon),
$|v| < |u|$, and $u$ is a prefix of $\SL_i(k\delta)$. This implies that either $\chain(\deg(v))$ is decreasing, or
$m_k(\deg(v)) > (k\delta,i)$, or $v$ is imaginary, due to Proposition~\ref{prop:monotonicity}.
Let us consider each of these cases separately:
\begin{itemize}[leftmargin=0.5cm]

\item
If $\chain(\deg(v))$ is decreasing, then the chain $\chain(\deg(\ell))$ is increasing by Lemma~\ref{lemma:chain.sum.monotone}
and the inductive hypothesis. On the other hand, as $\SL_i^{ls}(\delta) \leq\ell$, we get $\ell > \SL_{i+1}(\delta)$ by
Lemma~\ref{lemma:left.standard.comp.imaginary}. If we had $m_k(\deg(\ell)) < (k\delta,i)$, then $m_1(\deg(\ell)) < (\delta,i)$
by Corollary~\ref{cor:equiv.mk.Mk}, which then contradicts to Proposition~\ref{prop:monotonicity}, as we get
$\SL_{i+1}(\delta)< \ell < \SL(m_1(\deg(\ell))) \leq \SL_{i+1}(\delta)$. Thus, $m_k(\deg(\ell)) \geq (k\delta,i)$.

\item
If $m_k(\deg(v)) > (k\delta,i)$ then $m_k(\deg(\ell)) \geq (k\delta,i)$ by Lemma~\ref{lemma:min.im.rule} and the induction
assumption. Thus $\chain(\deg(\ell))$ is increasing by Proposition~\ref{prop:monotonicity} as $\ell < \SL_i(k\delta)$.

\item
If $v$ is imaginary, then $m_k(\deg(\ell)) = m_k(\deg(u))$ and $\chain(\deg(\ell))= \chain(\deg(u))$.
Hence, the results for $\ell$ follow immediately from the inductive hypothesis for $u$.

\end{itemize}

We now do the second round of induction on the length of $\ell$. The base case $\ell=\SL_i^{ls}(\delta)$
(by repeated applications of Lemma~\ref{lemma:weak.form.left.standard.imaginary}) follows from
Corollary~\ref{cor:im.splitting.convex.set}. For the step of induction, let us assume that $m_k(\deg(u)) = (k\delta,i)$
for all Lyndon prefixes $u$ satisfying $|\SL^{ls}_i(\delta)|\leq |u| < |\ell|$. Given $\ell$, set $u = \ell^{ls}$ and
$v = \ell^{rs}$, so that the result holds for $u$ by the induction assumption. Let $\alpha=\deg(u)$ and $\beta=\deg(v)$.

If $\beta$ is imaginary, then the result follows from the inductive hypothesis for $u$. Henceforth, we assume that
$\beta \in \wDelta^{+,\re}$. By the inductive hypothesis, we have $m_k(\alpha) = (k\delta,i)$ and $\chain(\alpha)$ is
increasing. Let $0<p<k$ be the largest such that $\SL^{ls}_i(p\delta)$ is a prefix of $u$. We note that
$|\alpha + \beta| \leq |\SL^{ls}_i((p+1)\delta)|<|(p+1)\delta|$, since $\SL^{ls}_i((p+1)\delta)$ is a Lyndon prefix of $\SL_i(k\delta)$
by Lemma~\ref{lemma:weak.form.left.standard.imaginary}. On the other hand, we have
$|\alpha| = |u| \geq |\SL^{ls}_{i}(p\delta)|  > |(p-1)\delta|$ by Lemma~\ref{lemma:length.standfac.imaginary}. Therefore,
$|\beta| < |2\delta|$. As shown in the first round of induction, the chain $\chain(\alpha + \beta)$ is increasing.
We also have $\beta > \alpha$ and $\alpha + \beta < (k\delta,i) < (p\delta,i)$ by Corollary~\ref{cor:imaginary.words.decreasing}.
We can thus apply Corollary~\ref{cor:lifting} to deduce $m_p(\beta) > (p\delta,i)$, so that $m_k(\beta) > (k\delta,i)$
by Corollary~\ref{cor:equiv.mk.Mk}. Then, $m_k(\alpha + \beta) = (k\delta,i)$ by Lemma~\ref{lemma:min.im.rule} and
the inductive hypothesis.
\end{proof}

\begin{corollary}\label{cor:w1.for.imaginary}
For any $i \in \{1,2,\ldots,|I|\}$ and $k \geq 1$, we have:
\begin{equation*}
  \SL_i^{ls}(k\delta) = \SL_i^{ls}((k-1)\delta)w \quad \text{for some word} \quad w \neq \emptyset.
\end{equation*}
Consider the canonical factorization $w=w_1 \ldots w_n$. Then $|w_j| < |\delta|$ for $j>1$.
We also have $n=1$ iff $\deg(w_1) = \delta$. Finally, for $n>1$, we have $|w_1| < |\delta|$,
$\chain(\deg(w_1))$ is decreasing, and $m_k(\deg(w_1)) > (k\delta,i)$.
\end{corollary}

\begin{proof}
The first claim follows from Lemma~\ref{lemma:weak.form.left.standard.imaginary}.

By Lemma~\ref{lemma:seq.right.word.Lyndon}, $\SL^{ls}_i((k-1)\delta)w_1$ is a Lyndon prefix of
$\SL_i^{ls}(k\delta)$, and is thus equal to $\SL(\alpha+\beta)$ where $\alpha = \deg(\SL_i^{ls}((k-1)\delta))$
and $\beta = \deg(w_1)$. We have $\SL^{ls}_i((k-1)\delta)w_1<\SL_i(k\delta)=\SL(m_k(\deg(\SL^{ls}_i((k-1)\delta)w_1)))$ by
Corollary~\ref{cor:mk.im.factor}. Then, $\SL^{ls}_i((k-1)\delta)w_1<\SL_i((k-1)\delta)$ as
$m_{k-1}(\deg(\SL^{ls}_i((k-1)\delta)w_1)) = ((k-1)\delta,i)$ by Corollary~\ref{cor:equiv.mk.Mk}.
We note that $|\SL^{ls}_i((k-1)\delta)w_1| \geq |(k-1)\delta|$, as otherwise
$\SL^{ls}_i((k-1)\delta)<\SL^{ls}_i((k-1)\delta)w_1<\SL_i((k-1)\delta)$
contradicts to Lemma~\ref{lemma:equiv.to.standard.fac}. Hence
  $$|w_j| < |\delta| \qquad \forall\, j>1. $$

According to Lemma~\ref{lemma:prefix.lyndon.imaginary}, $m_k(\alpha + \beta) = (k\delta,i)$ and $\chain(\alpha + \beta)$ is
increasing, so that $\alpha+\beta < ((k-1)\delta,i)$. Additionally $|\beta| < |2\delta|$ since $|\alpha| > |(k-2)\delta|$ by Lemma~\ref{lemma:length.standfac.imaginary}.
If $n=1$ and $\beta \in \wDelta^{+,\re}$, then by Lemma~\ref{lemma:lifting}, $\alpha + \beta < \alpha + \delta$, a contradiction
with Lemma~\ref{lemma:lifting.imaginary.left.standard}. This proves that $\deg(w_1)=\delta$ if $n=1$.
Vice versa, if $\beta = \delta$ and $n>1$, then applying Lemma~\ref{lemma:lifting} to $\alpha$ and $\delta + \deg(w_2)$ we get
$\alpha + \delta + \deg(w_2) < \alpha + \delta$, a contradiction with
$\SL(\alpha + \delta) = \SL(\alpha)w_1 < \SL(\alpha)w_1w_2=\SL(\alpha + \delta + \deg(w_2))$, due to Lemma~\ref{lemma:seq.right.word.Lyndon}.
This proves
  $$ n=1 \ \ \mathrm{if\ and\ only\ if} \ \ \beta = \delta. $$

If $n>1$, then $\beta\in \wDelta^{+,\re}$, and so $|\beta| < |\delta|$ and $m_{k-1}(\beta) > ((k-1)\delta,i)$
by Corollary~\ref{cor:lifting}. The latter implies $m_k(\beta) > (k\delta,i)$ by Corollary~\ref{cor:equiv.mk.Mk}.
Suppose that $\chain(\beta)$ is increasing. We then have $\alpha +\beta< (k\delta,i) < (k\delta,i-1) <  \delta - \beta$
by Proposition~\ref{prop:monotonicity}, so that $\alpha + \beta < \delta - \beta$. Thus we get
$\SL(\alpha + \delta) \geq \SL(\alpha+\beta)\SL(\delta-\beta) = \SL(\alpha)\SL(\beta)\SL(\delta-\beta)$.
But since $\SL(\delta - \beta) >\SL(m_1(\beta)) > \SL(\beta) = w_1 \geq w_2$, we get
$\SL(\delta - \beta) > w_2w_3\ldots w_n$ by Lemma~\ref{lem:Lyndon-vs-canonical}. But then
$\SL(\alpha + \delta) > \SL(\alpha)w_1w_2\ldots w_n$, which contradicts to Lemma~\ref{lemma:lifting.imaginary.left.standard}.
Hence $\chain(\beta)$ is decreasing.
\end{proof}

\begin{corollary}\label{cor:imaginary.standard.fac.sequence}
Using notations of Corollary~\ref{cor:w1.for.imaginary}, set
$\gamma_j = \deg(\SL^{ls}_i((k-1)\delta)) + \sum_{p=1}^j \deg(w_p)$ for any $1\leq j\leq n$.
Then $\chain(\gamma_j)$ is increasing and $m_{k}(\gamma_j) = (k\delta,i)$ for any $j$.
If $\deg(w_1) \neq \delta$, then $m_{k}(\deg(w_j)) > (k\delta,i)$ for all $j$.
\end{corollary}

\begin{proof}
According to Lemma~\ref{lemma:seq.right.word.Lyndon}, each $\SL^{ls}_i((k-1)\delta)w_1\dots w_j$ is a Lyndon prefix
of $\SL_i^{ls}(k\delta)$, and is thus equal to $\SL(\gamma_j)$. Note that each $\gamma_j$ is real by
Lemma~\ref{lemma:no.imaginary.prefix}. Moreover, $\chain(\gamma_j)$ is increasing and
$m_{k}(\gamma_j) = (k\delta,i)$ by Lemma~\ref{lemma:prefix.lyndon.imaginary}.

If $\deg(w_1) \neq \delta$, let us now show that $m_{k}(\deg(w_j)) > (k\delta,i)$ by induction on $j$.
The base case $j=1$ follows from Corollary~\ref{cor:w1.for.imaginary}. As per the step of induction, assume that
$m_{k}(\deg(w_p)) > (k\delta,p)$ for all $p<j$. Then $\gamma_{j-1}$ and $\deg(w_j)$ satisfy the requirements for
Corollary~\ref{cor:lifting}, hence we have $m_{k-1}(\deg(w_j)) > ((k-1)\delta,i)$, and so
$m_k(\deg(w_j)) > (k\delta,i)$ by Corollary~\ref{cor:equiv.mk.Mk}.
\end{proof}

We can now describe the biggest imaginary $\aslaws$:

\begin{proposition}\label{prop:SL.1.form}
For any $k \in \BZ_{>0}$, we have:
\begin{equation*}
  \SL_1(k\delta) = \SL^{ls}_1(\delta)\underbrace{\SL_1(\delta)}_{k-1 \text{ times }}\SL^{rs}_1(\delta), \qquad
  \SL^{ls}_1(k\delta) = \SL_1^{ls}(\delta)\underbrace{\SL_1(\delta)}_{k-1 \text{ times }}.
\end{equation*}
\end{proposition}

\begin{proof}
The proof is by induction on $k$, the base case $k = 1$ being trivial. For the inductive step, assume the above
equalities hold for $k-1$.
Invoking the notations of Corollary~\ref{cor:w1.for.imaginary}, assume first that $n > 1$. Then
$m_k(\deg(w_1)) > (k\delta,1)$ which is impossible, a contradiction. Thus $n=1$ and $\deg(w_1) = \delta$ by Corollary~\ref{cor:w1.for.imaginary}.
By Corollary~\ref{cor:imaginary.suffix.prefix} and
$M_1(\deg(\SL_1^{ls}((k-1)\delta)))=M_1(\deg(\SL_1^{ls}(\delta))) = (\delta,1)$, we get:
  $$ \SL_1^{ls}(k\delta) = \SL_i^{ls}((k-1)\delta)\SL_1(\delta). $$
Combining this with the inductive hypothesis completes the inductive step.
\end{proof}

We now propose the structure of all imaginary $\aslaws$:

\begin{conjecture}\label{conj:imaginary-SL}
For all $i \in \{1,2,\ldots,|I|\}$ and $k \in \BZ_{>0}$, we have:
\begin{equation}\label{eq:conf_formula}
  \SL_i(k\delta) = \SL^{ls}_i(\delta)\underbrace{w'}_{k-1 \text{ times}} \SL_i^{rs}(\delta), \qquad
  \SL^{ls}_i(k\delta) = \SL^{ls}_i(\delta)\underbrace{w'}_{k-1 \text{ times}},
\end{equation}
where $w'$ is a cyclic permutation of some $\SL_j(\delta)$ for $j \leq i$.
\end{conjecture}

\begin{remark}\label{rem:conj_special_case}
Taking $w' = \SL_1(\delta)$ if $i=1$ or $w' = \SL(M_1(\deg(\SL_i^{ls}(\delta))))$ if the smallest simple root
occurs only once in $\delta$, we recover Proposition~\ref{prop:SL.1.form} and Theorem~\ref{thm:imaginary-SL-orderI}, respectively.
\end{remark}

\begin{remark}
The first part of the conjecture implies the form of $w'$ in the second part.
Indeed, we must have $\deg(w') = \delta$, and since $\alpha_0$ occurs only once in $\delta$, there must be a cyclic permutation
$\ell$ of $w'$ which is Lyndon (specifically, $\ell$ is the smallest among all cyclic permutations of $w'$).
Then $\ell$ is a subword of $\SL_i(3\delta)$ by Corollary~\ref{cor:word.to.Lyndon}, and so
$\ell\in \SL$. This implies that $\ell=\SL_j(\delta)$ for some~$j$. If $j > i$, then $\SL_i(3\delta)$ has a suffix of
the form $\SL_j(\delta)v$ for some $v$. Combining Lemmas~\ref{lemma:left.standard.comp.imaginary}
and~\ref{lemma:weak.form.left.standard.imaginary}, we get
  $\SL_{j}(\delta)\leq \SL_{i+1}(\delta)<\SL_i^{ls}(\delta) < \SL^{ls}_i(3\delta)<\SL_i(3\delta)$.
As $\SL_{j}(\delta)$ cannot be a prefix of $\SL_i(3\delta)$ by Lemma~\ref{lemma:no.imaginary.prefix}, we then get
$\SL_{j}(\delta)v < \SL_i(3\delta)$, a contradiction with $\SL_i(3\delta)$ being Lyndon. This shows that $j \leq i$.
\end{remark}

We note that $w'$ may be a nontrivial cyclic permutation of some $\SL_j(\delta)$ which is a new phenomenon in comparison
with the type $A^{(1)}_n$ of~\cite{AT} that we illustrate with a couple of examples.

\begin{example}
Consider the affine type $C_3^{(1)}$ with the order $1<3<0<2$ (our labeling of simple roots follows \cite[\S4, Table Fin]{K}).
Using the code, we find the following structure of imaginary affine standard Lyndon words:
\begin{gather*}
        \SL_1(\delta) = 123120,\quad \SL_2(\delta) = 121203, \quad \SL_3(\delta) = 101232,\\
        \SL_1(k\delta) = 123\ub{1}{k-1}120 \qquad \text{for $k > 1$}, \\
        \SL_2(k\delta) = 12120\ub{1}{k-1}3 \qquad \text{for $k > 1$}, \\
        \SL_3(k\delta) = 10123\underbrace{120123}_{k-1 \text{ times}}2 \qquad \text{for $k > 1$}.
\end{gather*}
Here, we note that $120123$ is a nontrivial cyclic permutation of $\SL_1(\delta) = 123120$.
\end{example}

\begin{example}
Consider the affine type $F_4^{(1)}$ with the order $4<2<0<3<1$ (our labeling of simple roots follows \cite[\S4, Table Fin]{K}).
Using the code, we find the following structure of imaginary affine standard Lyndon words:
\begin{gather*}
        \SL_1(\delta) = 432132432130, \qquad \SL_2(\delta) = 432134321302, \\
        \SL_3(\delta) = 432104321323, \qquad \SL_4(\delta) = 432343213021, \\
        \SL_1(k\delta) = 432132\ub{1}{k-1}432130\qquad \text{for $k > 1$}, \\
        \SL_2(k\delta) = 43213432130\ub{1}{k-1}2\qquad \text{for $k > 1$}, \\
        \SL_3(k\delta) = 43210432132\underbrace{432130432132}_{k-1 \text{ times}}3\qquad \text{for $k > 1$}, \\
        \SL_4(k\delta) = 43234321302\ub{2}{k-1}1\qquad \text{for $k > 1$}.
\end{gather*}
Here, we note that $432130432132$ is a cyclic permutation of $\SL_1(\delta) = 432132432130$.
\end{example}

\begin{proposition}\label{lem:conj_deduction}
For any $i$, if~\eqref{eq:conf_formula} holds for $k= 2$, then~\eqref{eq:conf_formula} holds for all~$k$.
\end{proposition}

\begin{proof}
We prove this by induction on $k$. The base cases $k=1$ and $k=2$ are clear.
For the induction step, assume that~\eqref{eq:conf_formula} holds for $k-1$.
By Lemma~\ref{lemma:weak.form.left.standard.imaginary}, we have $\SL^{ls}_i(k\delta) = \SL_i^{ls}((k-1)\delta)w$
for some $w$. It thus suffices to show that $w'=w$, where we use the assumption to write
$\SL^{ls}_i(2\delta) = \SL_i^{ls}(\delta)w'$. Let $\beta = \deg(\SL_i^{ls}(\delta))$.

First, let us assume that $w < w'$.
Consider the canonical factorizations $w = w_1 w_2  \ldots w_m$ and $w' = w'_1w'_2\ldots w'_{n}$.
Choose $p$ so that $w_{p}\ne w'_{p}$ and $w_k=w'_k$ for $k<p$.
By the result of~\cite{M} (a generalization of Lemma~\ref{lem:Lyndon-vs-canonical}), $w<w'$ implies $w_p<w'_p$.
We note that $\SL_i^{ls}((k-1)\delta)w'_1\ldots w'_{p-1}=\SL_i^{ls}((k-1)\delta)w_1\ldots w_{p-1}$ is Lyndon by
Lemma~\ref{lemma:seq.right.word.Lyndon}. Set $\alpha:=\deg(\SL_i^{ls}((k-1)\delta)w'_1w'_2 \ldots w'_{p-1}) + \deg(w'_p)$,
so that $\alpha = \beta + \sum_{j=1}^p\deg(w'_j) + (k-2)\delta$ by the inductive hypothesis. As
$\SL_i^{ls}(\delta)w'_1\ldots w'_p$ is a Lyndon prefix of $\SL_i(2\delta)$, we see that $\beta + \sum_{j=1}^p\deg(w'_j)\in \wDelta^{+,\re}$
and invoking Lemma~\ref{lemma:prefix.lyndon.imaginary} we get $m_2(\alpha) = m_2(\beta + \sum_{j=1}^p\deg(w'_j)) = (2\delta,i)$ and
$\chain(\alpha)$ is increasing.
Recall that $\SL^{ls}_i(\delta) < w'_p$ by Lemma~\ref{lemma:inc.seq.right.word.lyndon}.
According to Corollary~\ref{cor:w1.for.imaginary}, we have $|w'_p| \leq |\delta|$ with the equality iff $p=n=1$.
Then $\SL_i(\delta) < w'_p$ by Lemma~\ref{lemma:equiv.to.standard.fac} unless $p=n=1$, while in the latter case we have
$w'_p = \SL_i(\delta)$ by Corollary~\ref{cor:imaginary.suffix.prefix}.
As $\SL_i^{ls}((k-1)\delta)w'_1\ldots w'_{p-1}<\SL_i(k\delta)<\SL_i(\delta)\leq w'_p$,
we see that $\SL(\alpha) \geq \SL_i^{ls}((k-1)\delta)w'_1\ldots w'_p>\SL_i^{ls}(k\delta)$ by the generalized Leclerc algorithm
(which applies by Lemma~\ref{lemma:no.imaginary.prefix}) and Lemma~\ref{lem:Lyndon-vs-canonical}.
As $|\alpha|  = |\beta + (k-2)\delta + \sum_{j=1}^p \deg(w'_j)| < |k\delta|$, we get
$\SL_i(k\delta)<\SL(\alpha)$ by Lemma~\ref{lemma:equiv.to.standard.fac}. On the other hand, as $\chain(\alpha)$ is increasing
and $m_k(\alpha) = (k\delta,i)$ by Corollary~\ref{cor:equiv.mk.Mk}, we obtain $\SL_i(k\delta) > \SL(\alpha)$, a contradiction.

Assume now that $w > w'$. Consider the canonical factorizations of $w',w$, and choose $p$ as in the previous case,
so that $w_p>w'_p$ now. Invoking Lemma~\ref{lemma:prefix.lyndon.imaginary}, we again see that
$\tilde{\alpha}:=\deg(\SL_i^{ls}((k-1)\delta)w_1w_2\ldots w_{p}) \in \wDelta^{+,\re}$ satisfies $m_k(\tilde{\alpha})=(k\delta,i)$ and
$\chain(\tilde{\alpha})$ is increasing.
Combining the inductive hypothesis with Lemma~\ref{lem:Lyndon-vs-canonical}, we get
$\SL_i^{ls}((k-1)\delta)=\SL_i^{ls}((k-2)\delta)w' < \SL_i^{ls}((k-2)\delta)w_1w_2 \ldots w_p < w_p$.
Thus $w_p > \SL_i((k-1)\delta)$ by Lemma~\ref{lemma:equiv.to.standard.fac}.
Hence $\SL_i^{ls}((k-2)\delta)w_1w_2 \ldots w_{p-1} < \SL_i^{ls}((k-1)\delta) < w_p$.
We thus have $\SL(\tilde{\alpha} - \delta) \geq \SL_i^{ls}((k-2)\delta)w_1w_2 \ldots w_{p} > \SL_i^{ls}((k-1)\delta)$ by Lemma~\ref{lemma:no.imaginary.prefix},
and so  $\SL_i((k-1)\delta) < \SL(\tilde{\alpha} - \delta)$ by Lemma~\ref{lemma:equiv.to.standard.fac}.
But since $m_{k-1}(\tilde{\alpha}) = ((k-1)\delta,i)$ and $\chain(\tilde{\alpha})$ is increasing,
we get $\SL_i((k-1)\delta) > \SL(\tilde{\alpha}-\delta)$, a contradiction.

Therefore, we must have $w'=w$, which completes the step of induction.
\end{proof}

\begin{remark}\label{rem:conj-exceptional}
Using the computer code (Appendix~\ref{sec:app_code}) we verified the validity of \eqref{eq:conf_formula} for all
$1\leq i\leq |I|$ and $k= 2$ for all exceptional types and any orders on $\wI$. This confirms Conjecture~\ref{conj:imaginary-SL}
for all orders in exceptional types, due to Proposition~\ref{lem:conj_deduction}.
\end{remark}


\appendix

\section{Code}\label{sec:app_code}

In this Appendix, we present the source code (written using Python):

$\qquad$ \href{https://github.com/corbyte/AffineStandardLyndonWords}{https://github.com/corbyte/AffineStandardLyndonWords}.

\noindent
Almost everything in the code is done through a rootSystem object and example of the initialization can be seen below:

\begin{lstlisting}[language=python,caption=RootSystem Initialization,captionpos=t]
"""
rootSystem(ordering,type:str):
Initialization of root system

ordering -- list of ordering for the rootsystem with ordering[0] < ordering[1] < and so on
type -- type of the rootsystem
"""

G2 = rootSystem([2,1,0],'G')

G2.delta
#[1,2,3]
G2.baseRoots
# Will return all roots in the root system with height \leq \delta
\end{lstlisting}

\medskip
\noindent
In addition to the rootSystem class another important class is the word class: with this class you can do comparison
and concatenation between words. The word class acts as a wrapper around a list of elements from the letter class:
\begin{lstlisting}[language=python,caption=Word Class,captionpos=t]
#'b'<'a'<'c'
u #abc
v #bc
u < v
# False
u + v
# abcbc
print(u)
#a,b,c
u.no_commas()
#abc
\end{lstlisting}

\medskip
\noindent
Getting a standard Lyndon word for a given rootSystem and ordering is very quick, additionally one can quickly get
chains of standard Lyndon words:

\begin{lstlisting}[language=python,caption= Standard Lyndon Words,captionpos=t]
root_system # any rootSystem object
l = root_system.SL(degree) #Where degree is an element of the root system

#l will be an array of word objects, if degree is real there will only be one, but if the degree is imaginary there will be several

chain = root_system.chain(degree)

#chain is a list of all currently generated standard Lyndon words with degree, degree + k\delta

root_system.periodicity(degree)

#Returns the periodicity of ch(degree)
\end{lstlisting}

\medskip
\noindent
The notation to use for degree is that the $(i+1)$-th element of the degree list you want corresponds to the multiplicity
of $\alpha_i$ in that degree. Additionally, words can be quickly parsed into the ``chunk format'':

\begin{lstlisting}[language=python,caption=Chunk Format,captionpos=t]
G2 = rootSystem([1,2,0],'G')

G2.SL(G2.delta*8 + [1,0,0])[0].no_commas()
#'1222101222102122210101222101222102122210101222102'

print(G2.to_chunk_format(G2.SL(G2.delta*8 + [1,0,0])[0]))
#[[1, 2], '2', [1, 1], '10', [1, 2], '2', [1, 1], '10', [1, 1], '2']
#[i,j] means that there is an \SL_i(\delta) j times in that spot
\end{lstlisting}

\medskip
\noindent
There are some additional functions which will give useful information about standard Lyndon words and degrees:
\begin{lstlisting}[language=python,caption=Additional Functions,captionpos=t]
G2 = rootSystem([1,2,0],'G')
#get_monotonicity returns 1 if the chain is increasing and -1 if it is decreasing
G2.get_monotonicity([0,1,0])
#1
#rootSystem.M_k(degree) returns i where M_k(degree) = (k\delta,i)
G2.M_k([0,1,0])
#1
#rootSystem.m_k(degree) returns i where m_k(degree) = (k\delta,i)
G2.m_k([0,1,0])
#2

#mod_delta(\alpha+k\delta) will return (\alpha,k\delta)

G2.mod_delta(G2.delta*5 + [0,1,0])
#(array([0, 1, 0]), 5)

#generate_up_to_delta(k) will generate all standard Lyndon words up to height n\delta, results will be cached

G2.generate_up_to_delta(5)

#get_decompositions(\alpha) will return all possible \beta,\gamma \in \wDelta^{+} such that \beta + \gamma = \alpha
G2.get_decompositions(G2.delta)

#You can also get the standard and costandard factorization of words

l = G2.SL(G2.delta)[1]

print(*[i.no_commas()for i in G2.costfac(l)],sep=',')
#2,21210
print(*[i.no_commas()for i in G2.standfac(l)],sep=',')
#22121,0
\end{lstlisting}

The following is used to generate Examples~\ref{ex:F4} and~\ref{ex:E6}:

\begin{lstlisting}[language=python,caption=W-set,label={lst:w.set},captionpos=t]
    F4 = rootSystem([3,4,0,2,1],"F")
    F4.text_W_set(1)
    #Prints the W_{\delta} set for F4 with ordering 3<4<0<2<1
    print(*[i.no_commas() for i in F4.SL(F4.delta)],sep="\n")
    #Prints imaginary SL words of height delta
    E6 = rootSystem([3,0,1,5,4,6,2],"E")
    E6.text_W_set(1)
    #Prints the W_{\delta} set for E6 with ordering 3<0<1<5<4<6<2
    print(*[i.no_commas() for i in E6.SL(E6.delta)],sep="\n")
    #Prints imaginary SL words of height delta
\end{lstlisting}


\section{Explicit formulas for $G^{(1)}_2$}\label{sec:app-G}

In this Appendix, we present the list of all $\aslaws$ in~affine type $G^{(1)}_2$. These were derived
using the code of Appendix~\ref{sec:app_code}. We use the conventions that $\alpha_1$ is a long root and
$\alpha_2$ is a short root of $G_2$. We note that having these formulas at hand, one can directly verify
them by induction on the height of a root using the generalized Leclerc algorithm. While similar to
type $A^{(1)}_n$ of~\cite{AT}, the structure is more compelling as we get up to $5$ ``chunks'' with
$\SL_1(\delta)$ for real roots.


\subsection{Order $0<1<2$}\hfill\\

\centering
\begin{tabular}{|c|c|}
    \hline
     & $\SL(\cdot)$ \\
     \hline
     $(\delta,1)$& 012221\\
     \hline
     $(\delta,2)$&012212\\
     \hline
\end{tabular}

\bigskip
\begin{tabular}{|>{\centering\arraybackslash}m{2.25cm}|c|}
    \hline
    \multicolumn{2} {|c|} {$\alpha_0 + k\delta$}\\
    \hline
     $k=0$ & $0$\\
     \hline
     $k=1$ & $0120122$\\
     \hline
     $k=2$ & $0122012201221$\\
     \hline
     $k=3$ & $0122012210122101222$\\
     \hline
     $k \equiv 0 \mod 3$\newline $k\geq 4$ & $01221 \ub{1}{k/3 - 2} 0122201221 \ub{1}{k/3 - 1} 01221 \ub{1}{k/3 - 1} 01222$\\
     \hline
     $k \equiv 1 \mod 3$\newline$k\geq 4$& $01221 \ub{1}{\lfloor k/3\rfloor - 1} 01221 \ub{1}{\lfloor k/3\rfloor - 1} 0122201221 \ub{1}{\lfloor k/3\rfloor - 1} 01222$\\
     \hline
     $k \equiv 2 \mod 3$\newline $k \geq 4$ & $01221 \ub{1}{\lfloor k/3\rfloor-1} 0122201221 \ub{1}{\lfloor k/3 \rfloor-1} 0122201221 \ub{1}{\lceil k/3 \rceil -1}$\\
     \hline
\end{tabular}

\bigskip
\begin{tabular}{|c|c|}
    \hline
    \multicolumn{2} {|c|} {$\alpha_1 + k\delta$}\\
    \hline
     $k$ & $\ub{1}{k}1$\\
    \hline
\end{tabular}
\begin{tabular}{|c|c|}
    \hline
    \multicolumn{2} {|c|} {$\alpha_2 + k\delta$}\\
    \hline
     $k$ & $\ub{1}{k}2$\\
    \hline
\end{tabular}
%
\begin{tabular}{|>{\centering\arraybackslash}m{2.25cm}|c|}
    \hline
    \multicolumn{2} {|c|} {$\alpha_1 + \alpha_2 + k\delta$}\\
    \hline
     $k\equiv 0 \mod 2$& $\ub{1}{k/2}1\ub{1}{k/2}2$ \\
     \hline
     $k \equiv 1 \mod 2$ & $\ub{1}{\lceil k/2 \rceil}2\ub{1}{\lfloor k/2 \rfloor}1$\\
    \hline
\end{tabular}

\bigskip
\begin{tabular}{|>{\centering\arraybackslash}m{2.25cm}|c|}
    \hline
    \multicolumn{2} {|c|} {$\alpha_0 + \alpha_1 + k\delta$}\\
    \hline
     $k=0$& $01$ \\
     \hline
    $k=1$ & 01201221\\
    \hline
    $k=2$ & $01220122101221$\\
    \hline
    $k\equiv 0 \mod 3$\newline $k\geq3$ & $01221\ub{1}{k/3-1} 01221 \ub{1}{k/3-1} 01221 \ub{1}{k/3-1} 01222$\\
    \hline
    $k \equiv 1 \mod 3$\newline$k\geq3$ & $01221 \ub{1}{\lfloor k/3 \rfloor -1} 01221 \ub{1}{\lfloor k/3 \rfloor -1} 0122201221 \ub{1}{\lceil k/3 \rceil -1}$\\
    \hline
    $k \equiv 2 \mod 3$\newline$k\geq3$ & $01221 \ub{1}{\lfloor k/3 \rfloor-1} 0122201221 \ub{1}{\lceil k/3 \rceil -1} 01221 \ub{1}{\lceil k/3 \rceil -1}$\\
    \hline
\end{tabular}

\bigskip
\begin{tabular}{|c|c|}
    \hline
    \multicolumn{2} {|c|} {$\alpha_1 +2\alpha_2 + k\delta$}\\
    \hline
    $k\equiv0 \mod 3$ & $\ub{1}{k/3} 1 \ub{1}{k/3} 2 \ub{1}{k/3} 2$\\
    \hline
    $k\equiv 1 \mod 3$ & $\ub{1}{\lceil k/3 \rceil} 2 \ub{1}{\lfloor k/3 \rfloor} 1 \ub{1}{\lfloor k/3 \rfloor} 2$\\
    \hline
    $k \equiv 2 \mod 3$ & $\ub{1}{\lceil k/3 \rceil} 2 \ub{1}{\lceil k/3 \rceil} 2 \ub{1}{\lfloor k/3 \rfloor} 1$\\
    \hline
\end{tabular}

\bigskip
\begin{tabular}{|>{\centering\arraybackslash}m{2.25cm}|c|}
    \hline
    \multicolumn{2} {|c|} {$\alpha_0 + \alpha_1 +\alpha_2 + k\delta$}\\
    \hline
    $k = 0$ & $012$\\
    \hline
    $k =1$ & $012201221$\\
    \hline
    $k \equiv 0 \mod 2$\newline $k\geq 2$ & $01221 \ub{1}{k/2-1} 01221 \ub{1}{k/2-1} 01222$\\
    \hline
    $k \equiv 1 \mod 2$\newline $k\geq2$ & $01221 \ub{1}{\lfloor k/2 \rfloor -1} 0122201221 \ub{1}{\lceil k/2 \rceil -1}$\\
    \hline
\end{tabular}

\bigskip
\begin{tabular}{|c|c|}
    \hline
    \multicolumn{2} {|c|} {$\alpha_1 +3\alpha_2 + k\delta$}\\
    \hline
    $k \equiv 0 \mod 4$ & $\ub{1}{k/4} 1 \ub{1}{k/4} 2 \ub{1}{k/4} 2 \ub{1}{k/4} 2$\\
    \hline
    $k \equiv 1 \mod 4$ & $\ub{1}{\lceil k/4 \rceil} 2 \ub{1}{\lfloor k/4 \rfloor} 1 \ub{1}{\lfloor k/4 \rfloor}2 \ub{1}{\lfloor k/4 \rfloor} 2$\\
    \hline
    $k \equiv 2 \mod 4$ & $\ub{1}{\lceil k/4 \rceil} 2 \ub{1}{\lceil k/4 \rceil}2 \ub{1}{\lfloor k/4 \rfloor} 1 \ub{1}{\lfloor k/4 \rfloor} 2$\\
    \hline
    $k \equiv 3 \mod 4$ &  $\ub{1}{\lceil k/4 \rceil} 2 \ub{1}{\lceil k/4 \rceil} 2 \ub{1}{\lceil k/4 \rceil} 2 \ub{1}{\lfloor k/4 \rfloor} 1$\\
    \hline
\end{tabular}

\bigskip
\begin{tabular}{|c|c|}
    \hline
    \multicolumn{2} {|c|} {$\alpha_0 + \alpha_1 +2\alpha_2 + k\delta$}\\
    \hline
    $k = 0$ & 0122\\
    \hline
    $k \geq 1$ & $01221 \ub{1}{k-1} 01222$\\
    \hline
\end{tabular}

\bigskip
\begin{tabular}{|c|c|}
    \hline
    \multicolumn{2} {|c|} {$2\alpha_1 +3\alpha_2 + k\delta$}\\
    \hline
    $k \equiv 0 \mod 5$ & $\ub{1}{k/5}1 \ub{1}{k/5} 2 \ub{1}{k/5} 1 \ub{1}{k/5} 2 \ub{1}{k/5}2$\\
    \hline
    $k \equiv 1 \mod 5$ & $\ub{1}{\lceil k/5 \rceil} 2 \ub{1}{\lfloor k/5 \rfloor} 1 \ub{1}{\lfloor k/5 \rfloor} 2 \ub{1}{\lfloor k/5 \rfloor} 2 \ub{1}{\lfloor k/5 \rfloor} 1$\\
    \hline
    $k \equiv 2 \mod 5$ & $\ub{1}{\lceil k/5 \rceil} 2 \ub{1}{\lfloor k/5 \rfloor} 1 \ub{1}{\lceil k/5 \rceil} 2 \ub{1}{\lfloor k/5 \rfloor} 1 \ub{1}{\lfloor k/5 \rfloor}2$\\
    \hline
    $k \equiv 3 \mod 5$ & $\ub{1}{\lceil k/5 \rceil} 2 \ub{1}{\lceil k/5 \rceil} 2 \ub{1}{\lfloor k/5 \rfloor} 1 \ub{1}{\lceil k/5 \rceil} 2 \ub{1}{\lfloor k/5 \rfloor} 1$\\
    \hline
    $k \equiv 4 \mod 5$ & $\ub{1}{\lceil k/5 \rceil} 1 \ub{1}{\lceil k/5 \rceil} 2 \ub{1}{\lceil k/5 \rceil} 2 \ub{1}{\lceil k/5 \rceil} 2 \ub{1}{\lfloor k/5 \rfloor} 1$\\
    \hline
\end{tabular}

\bigskip
\begin{tabular}{|c|c|}
    \hline
    \multicolumn{2} {|c|} {$\alpha_0 + \alpha_1 +3\alpha_2 + k\delta$}\\
    \hline
    $k$ & $01222 \ub{1}{k}$\\
    \hline
\end{tabular}
\begin{tabular}{|c|c|}
    \hline
    \multicolumn{2} {|c|} {$\alpha_0 + 2\alpha_1 +2\alpha_2 + k\delta$}\\
    \hline
    $k$ & $01221\ub{1}{k}$\\
    \hline
\end{tabular}

\bigskip
\begin{tabular}{|c|c|}
    \hline
    \multicolumn{2} {|c|} {$(k\delta,1)$}\\
     \hline
     $k>1$ & $01222 \ub{1}{k-1} 1$\\
     \hline
\end{tabular}
\begin{tabular}{|c|c|}
    \hline
    \multicolumn{2} {|c|} {$(k\delta,2)$}\\
     \hline
     $k>1$ & $01221 \ub{1}{k-1} 2$\\
    \hline
\end{tabular}


\raggedright
\medskip
\subsection{Order $0 < 2 < 1$}\hfill\\

\centering
\bigskip
\begin{tabular}{|c|c|}
    \hline
     & $\SL(\cdot)$ \\
     \hline
     $(\delta,1)$& 012212\\
     \hline
     $(\delta,2)$ & 012221\\
     \hline
\end{tabular}

\bigskip
\begin{tabular}{|>{\centering\arraybackslash}m{2.25cm}|c|}
    \hline
    \multicolumn{2} {|c|} {$\alpha_0 + k\delta$}\\
    \hline
     $k=0$ & $0$\\
     \hline
     $k=1$& 0120122\\
    \hline
     $k=2$ & 0122012201221\\
     \hline
     $k=3$ & 0122012210122101222\\
     \hline
     $k=4$ & 0122201221012220122101221\\
     \hline
     $k \equiv 0 \mod 2$\newline$k\geq 5$ & $01222 \ub{1}{k/2 -2} 0122101222 \ub{1}{k/2-2} 0122101221$\\
     \hline
     $k \equiv 1 \mod 2$\newline $k \geq 5$ & $01222 \ub{1}{\lfloor k/2\rfloor -2} 01221012210122101222 \ub{1}{\lceil k/2 \rceil -2}$\\
     \hline
\end{tabular}

\bigskip
\begin{tabular}{|c|c|}
    \hline
     \multicolumn{2}{|c|}{$\alpha_1 + k\delta$}\\
     \hline
     $k$ & $\ub{1}{k}$1 \\
     \hline
\end{tabular}
\begin{tabular}{|c|c|}
    \hline
     \multicolumn{2}{|c|}{$\alpha_2 + k\delta$}\\
     \hline
     $k$ & $\ub{1}{k}$2 \\
     \hline
\end{tabular}
\begin{tabular}{|c|c|}
    \hline
    \multicolumn{2}{|c|}{$\alpha_1 + \alpha_2 + k\delta$}\\
    \hline
    $k \equiv 0 \mod 2$ & $\ub{1}{k/2}2\ub{1}{k/2}1$\\
    \hline
    $k \equiv 1 \mod 2$ & $\ub{1}{\lceil k/2 \rceil}1\ub{1}{\lfloor k/2 \rfloor}2$\\
    \hline
\end{tabular}

\bigskip
\begin{tabular}{|>{\centering\arraybackslash}m{2.25cm}|c|}
    \hline
    \multicolumn{2} {|c|} {$\alpha_0 + \alpha_1 + k\delta$}\\
    \hline
     $k = 0$& 01\\
     \hline
     $k=1$ & 01201221\\
     \hline
     $k = 2$ & $01220122101221$\\
     \hline
     $k\geq 3$ & $01222 \ub{1}{k-3} 012210122101221$\\
     \hline
\end{tabular}

\bigskip
\begin{tabular}{|>{\centering\arraybackslash}m{2.25cm}|c|}
    \hline
    \multicolumn{2} {|c|} {$\alpha_1 + 2\alpha_2 + k\delta$}\\
    \hline
    $k \equiv 0 \mod 3$ & $\ub{1}{k/3}2\ub{1}{k/3}2\ub{1}{k/3}1$\\
    \hline
    $k \equiv 1 \mod 3$ & $\ub{1}{\lceil k/3 \rceil} 1 \ub{1}{\lfloor k/3 \rfloor} 2 \ub{1}{\lfloor k/3 \rfloor} 2$\\
    \hline
    $k \equiv 2 \mod 3$ & $\ub{1}{\lceil k/3\rceil} 2 \ub{1}{\lceil k/3 \rceil} 1 \ub{1}{\lfloor k/3 \rfloor} 2$\\
    \hline
\end{tabular}

\bigskip
\begin{tabular}{|>{\centering\arraybackslash}m{2.25cm}|c|}
    \hline
    \multicolumn{2} {|c|} {$\alpha_0 + \alpha_1 + \alpha_2 + k\delta$}\\
    \hline
    $k=0$ & 012\\
    \hline
    $k=1$ & 012201221\\
    \hline
    $k \geq 2$ & $01222 \ub{1}{k-2} 0122101221$\\
    \hline
\end{tabular}

\bigskip
\begin{tabular}{|>{\centering\arraybackslash}m{2.25cm}|c|}
    \hline
    \multicolumn{2} {|c|} {$\alpha_1 + 3\alpha_2 + k\delta$}\\
    \hline
    $k \equiv 0 \mod 4$ & $\ub{1}{k/4} 2 \ub{1}{k/4} 2 \ub{1}{k/4} 2 \ub{1}{k/4} 1$\\
    \hline
    $k \equiv 1 \mod 4$ & $\ub{1}{\lceil k/4 \rceil} 1 \ub{1}{\lfloor k/4 \rfloor} 2 \ub{1}{\lfloor k/4 \rfloor} 2 \ub{1}{\lfloor k/4 \rfloor} 2$\\
    \hline
    $k \equiv 2 \mod 4$ & $\ub{1}{\lceil k/4 \rceil} 2 \ub{1}{\lceil k/4 \rceil} 1 \ub{1}{\lfloor k/4 \rfloor} 2 \ub{1}{\lfloor k/4 \rfloor} 2$\\
    \hline
    $k \equiv 3 \mod 4$ & $\ub{1}{\lceil k/4 \rceil} 2 \ub{1}{\lceil k/4 \rceil} 2 \ub{1}{\lceil k/4 \rceil} 1 \ub{1}{\lfloor k/4 \rfloor} 2$\\
    \hline
\end{tabular}

\bigskip
\begin{tabular}{|c|c|}
    \hline
    \multicolumn{2} {|c|} {$\alpha_0 + \alpha_1 + 2\alpha_2 + k\delta$}\\
    \hline
     $k=0$& 0122\\
     \hline
     $k \geq 1$ & $01222 \ub{1}{k-1} 01221$\\
     \hline
\end{tabular}

\bigskip
\begin{tabular}{|>{\centering\arraybackslash}m{2.25cm}|c|}
    \hline
    \multicolumn{2} {|c|} {$2\alpha_1 + 3\alpha_2 + k\delta$}\\
    \hline
    $k \equiv 0 \mod 5$ & $\ub{1}{k/5} 2 \ub{1}{k/5} 2 \ub{1}{k/5} 1 \ub{1}{k/5} 2 \ub{1}{k/5} 1$\\
    \hline
    $k \equiv 1 \mod 5$& $\ub{1}{\lceil k/5 \rceil} 1 \ub{1}{\lfloor k/5 \rfloor} 2 \ub{1}{\lfloor k/5 \rfloor} 2 \ub{1}{\lfloor k/5 \rfloor} 2 \ub{1}{\lfloor k/5 \rfloor} 1$\\
    \hline
    $k \equiv 2 \mod 5$ & $\ub{1}{\lceil k/5 \rceil} 1 \ub{1}{\lfloor k/5 \rfloor} 2 \ub{1}{\lceil k/5 \rceil} 1 \ub{   1}{\lfloor k/5\rfloor} 2 \ub{1}{\lfloor k/5 \rfloor} 2$\\
    \hline
    $k \equiv 3 \mod 5$ & $\ub{1}{\lceil k/5 \rceil} 2 \ub{1}{\lceil k/5 \rceil} 1 \ub{1}{\lfloor k/5 \rfloor} 2 \ub{1}{\lfloor k/5 \rfloor} 2 \ub{1}{\lceil k/5 \rceil} 1$\\
    \hline
    $k \equiv 4 \mod 5$ & $\ub{1}{\lceil k/5 \rceil} 2 \ub{1}{\lceil k/5 \rceil} 1 \ub{1}{\lceil k/5 \rceil} 2 \ub{1}{\lceil k/5 \rceil} 1 \ub{1}{\lfloor k/5 \rfloor} 2$\\
    \hline
\end{tabular}

\bigskip
\begin{tabular}{|c|c|}
    \hline
    \multicolumn{2} {|c|} {$\alpha_0 + \alpha_1 + 3\alpha_2 + k\delta$}\\
    \hline
    $k$ & $01222\ub{1}{k}$\\
    \hline
\end{tabular}
\begin{tabular}{|c|c|}
    \hline
    \multicolumn{2} {|c|} {$\alpha_0 + 2\alpha_1 + 2\alpha_2 + k\delta$}\\
    \hline
    $k$ & $01221\ub{1}{k}$\\
    \hline
\end{tabular}

\bigskip
\begin{tabular}{|c|c|}
    \hline
    \multicolumn{2} {|c|} {$(k\delta,1)$}\\
    \hline
    $k>1$ & $01221 \ub{1}{k-1} 2$\\
    \hline
\end{tabular}
\begin{tabular}{|c|c|}
    \hline
    \multicolumn{2} {|c|} {$(k\delta,2)$}\\
    \hline
    $k>1$ & $01222 \ub{1}{k-1} 1$\\
    \hline
\end{tabular}


\raggedright
\medskip
\subsection{Order $1 < 0 < 2$}\hfill\\

\bigskip
\begin{center}

\begin{tabular}{|c|c|}
    \hline
     &  $\SL(\cdot)$\\
     \hline
     $(\delta,1)$& 120122\\
     \hline
     $(\delta,2)$ & 121220\\
     \hline
\end{tabular}

\bigskip
\begin{tabular}{|>{\centering\arraybackslash}m{2.25cm}|c|}
    \hline
    \multicolumn{2} {|c|} { $\alpha_1 + k\delta$}\\
    \hline
    $k=0$ & 1\\
    \hline
    $k=1$ & 1212120\\
    \hline
    $k=2$ & 1212012012122\\
    \hline
    $k \equiv 0 \mod 2$ \newline $k \geq 3$ & $12122 \ub{1}{k/2 - 2} 12012012012122 \ub{1}{k/2-1}$\\
    \hline
    $k \equiv 1 \mod 2$ \newline $k \geq 3$ & $12122 \ub{1}{\lfloor k/2 \rfloor -1} 12012122 \ub{1}{\lfloor k/2 \rfloor -1} 120120$\\
    \hline
\end{tabular}

\bigskip
\begin{tabular}{|>{\centering\arraybackslash}m{2.25cm}|c|}
    \hline
    \multicolumn{2} {|c|} {$\alpha_2 + k\delta$}\\
    \hline
    $k=0$ & 2\\
    \hline
    $k \equiv 0 \mod 3$ \newline $k \geq 1$ & $\ub{1}{k/3} 122 \ub{1}{k/3} 0 \ub{1}{k/3-1} 122$\\
    \hline
    $k \equiv 1 \mod 3$ \newline $k \geq 1$ & $\ub{1}{\lfloor k/3 \rfloor} 122\ub{1}{\lfloor k/3 \rfloor} 122 \ub{1}{\lfloor k/3 \rfloor} 0$\\
    \hline
   $k \equiv 2 \mod 3$ \newline $k \geq 1$ & $\ub{1}{\lceil k/3 \rfloor} 0 \ub{1}{\lfloor k/3 \rfloor} 122 \ub{1}{\lfloor k/3 \rfloor} 122$\\
    \hline
\end{tabular}

\bigskip
\begin{tabular}{|c|c|}
    \hline
    \multicolumn{2} {|c|} {$\alpha_0 + k\delta$}\\
    \hline
    $k$ & $\ub{1}{k}0$\\
    \hline
\end{tabular}
\begin{tabular}{|c|c|}
    \hline
    \multicolumn{2} {|c|} {$\alpha_1 + \alpha_2+ k\delta$}\\
    \hline
    $k=0$ & $12$\\
    \hline
    $k \geq 1$ & $12122 \ub{1}{k-1} 120$\\
    \hline
\end{tabular}

\bigskip
\begin{tabular}{|c|c|}
    \hline
    \multicolumn{2} {|c|} {$\alpha_0 + \alpha_1+ k\delta$}\\
    \hline
    $k=0$ & $10$\\
    \hline
    $k = 1$ & $12120120$\\
    \hline
    $k \geq 2$ & $12122 \ub{1}{k-2} 120120120$\\
    \hline
\end{tabular}
\begin{tabular}{|c|c|}
    \hline
    \multicolumn{2} {|c|} {$\alpha_1 + 2\alpha_2+ k\delta$}\\
    \hline
     $k$& $\ub{1}{k}122$\\
     \hline
\end{tabular}
\begin{tabular}{|c|c|}
    \hline
    \multicolumn{2} {|c|} {$\alpha_0 + \alpha_1 + \alpha_2+ k\delta$}\\
    \hline
     $k$& $120\ub{1}{k}$\\
     \hline
\end{tabular}

\bigskip
\begin{tabular}{|>{\centering\arraybackslash}m{2.25cm}|c|}
    \hline
    \multicolumn{2} {|c|} {$\alpha_1 + 3\alpha_2 + k\delta$}\\
    \hline
    $k = 1$ & $1222$\\
    \hline
    $k \equiv 0 \mod 4$ \newline $k \geq 2$ & $\ub{1}{k/4} 122 \ub{1}{k/4} 122 \ub{1}{k/4} 0 \ub{1}{k/4 - 1} 122$\\
    \hline
    $k \equiv 1 \mod 4$ \newline $k \geq 2$ & $\ub{1}{\lfloor k/4 \rfloor} 122 \ub{1}{\lfloor k/4 \rfloor} 122 \ub{1}{\lfloor k/4 \rfloor} 122 \ub{1}{\lfloor k/4 \rfloor} 0$\\
    \hline
    $k \equiv 2 \mod 4$ \newline $k \geq 2$ & $\ub{1}{\lceil k/4 \rceil} 0 \ub{1}{\lfloor k/4\rfloor} 122 \ub{1}{\lfloor k/4 \rfloor} 122 \ub{1}{\lfloor k/4 \rfloor} 122$\\
    \hline
    $k \equiv 3 \mod 4$ \newline $k \geq 2$ & $\ub{1}{\lceil k/4 \rceil} 122 \ub{1}{\lceil k/4 \rceil} 0 \ub{1}{\lfloor k/4 \rfloor} 122 \ub{1}{\lfloor k/4 \rfloor} 122$\\
    \hline
\end{tabular}

\bigskip
\begin{tabular}{|>{\centering\arraybackslash}m{2.25cm}|c|}
    \hline
    \multicolumn{2} {|c|} {$\alpha_0 + \alpha_1 + 2\alpha_2 + k\delta$}\\
    \hline
    $k \equiv 0 \mod 2$ & $\ub{1}{k/2}122\ub{1}{k/2}0$\\
    \hline
    $k \equiv 1 \mod 2$ & $\ub{1}{\lceil k/2 \rceil} 0 \ub{1}{\lfloor k/2 \rfloor} 122$\\
    \hline
\end{tabular}
\begin{tabular}{|c|c|}
    \hline
    \multicolumn{2} {|c|} {$2\alpha_1 + 3\alpha_2+ k\delta$}\\
    \hline
    $k$ & $12122\ub{1}{k}$\\
    \hline
\end{tabular}

\bigskip
\begin{tabular}{|>{\centering\arraybackslash}m{2.25cm}|c|}
    \hline
    \multicolumn{2} {|c|} {$\alpha_0 + \alpha_1 + 3\alpha_2 + k\delta$}\\
    \hline
    $k = 0$ & 12220\\
    \hline
    $k \equiv 0 \mod 5$\newline $k \geq 1$ & $\ub{1}{k/5} 122 \ub{1}{k/5} 0\ub{1}{k/5} 122 \ub{1}{k/5} 0 \ub{1}{k/5 - 1} 122$\\
    \hline
    $k \equiv 1\mod 5$ \newline $k \geq 1$& $\ub{1}{\lfloor k/5 \rfloor} 122 \ub{1}{\lfloor k/5 \rfloor} 122   \ub{1}{\lfloor k/5 \rfloor}0 \ub{1}{\lfloor k/5 \rfloor} 122 \ub{1}{\lfloor k/5 \rfloor} 0$\\
    \hline
    $k\equiv 2 \mod 5$ \newline $k \geq 1$ & $\ub{1}{\lceil k/5 \rceil} 0 \ub{1}{\lfloor k/5 \rfloor} 122 \ub{1}{\lfloor k/5 \rfloor} 122 \ub{b1}{\lfloor k/5 \rfloor} 122 \ub{1}{\lfloor k/5 \rfloor} 0$\\
    \hline
    $k \equiv 3\mod 5$ \newline $k \geq 1$ & $\ub{1}{\lceil k/5 \rceil} 0 \ub{1}{\lfloor k/5 \rfloor} 122 \ub{1}{\lceil k/5 \rceil} 0 \ub{1}{\lfloor k/5 \rfloor} 122 \ub{1}{\lfloor k/5 \rfloor} 122$\\
    \hline
    $k \equiv 4 \mod 5$ \newline $k \geq 1$ & $\ub{1}{\lceil k/5 \rceil} 122 \ub{1}{\lceil k/5 \rceil} 0 \ub{1}{\lfloor k/5 \rfloor} 122 \ub{1}{\lfloor k/5 \rfloor} 122 \ub{1}{\lceil k/5 \rceil} 0$ \\
    \hline
\end{tabular}

\bigskip
\begin{tabular}{|>{\centering\arraybackslash}m{2.25cm}|c|}
    \hline
    \multicolumn{2} {|c|} {$\alpha_0 + 2\alpha_1 + 2\alpha_2 + k\delta$}\\
    \hline
    $k = 0$ & 12120\\
    \hline
    $k \geq 1$ & $12122 \ub{1}{k-1} 120120$\\
    \hline
\end{tabular}

\bigskip
\begin{tabular}{|c|c|}
    \hline
    \multicolumn{2} {|c|} {$(k\delta,1)$}\\
    \hline
    $k>1$ & $120 \ub{1}{k-1} 122$\\
    \hline
\end{tabular}
\begin{tabular}{|c|c|}
    \hline
    \multicolumn{2} {|c|} {$(k\delta,2)$}\\
    \hline
    $k>1$ & $12122 \ub{1}{k-1} 0$\\
    \hline
\end{tabular}
\end{center}


\raggedright
\subsection{ Order $1 < 2 < 0$}\hfill\\

\bigskip
\begin{center}

\begin{tabular}{|c|c|}
    \hline
    & $\SL(\cdot)$\\
    \hline
    $(\delta,1)$ & $122210$\\
    \hline
    $(\delta,2)$ & $122102$\\
    \hline
\end{tabular}

\bigskip
\begin{tabular}{|>{\centering\arraybackslash}m{2.25cm}|c|}
    \hline
    \multicolumn{2} {|c|} {$\alpha_0 + k\delta$}\\
    \hline
    $k = 0$ & $0$\\
    \hline
    $k \equiv 0 \mod 5$ \newline $k \geq 1$ & $\ub{1}{k/5} 10 \ub{1}{k/5} 2 \ub{1}{k/5} 2 \ub{1}{k/5} 2 \ub{1}{k/5-1} 10$\\
    \hline
    $k \equiv 1 \mod 5$\newline $k \geq 1$ & $\ub{1}{\lfloor k/5 \rfloor} 10 \ub{1}{\lfloor k/5 \rfloor} 2 \ub{1}{\lfloor k/5 \rfloor} 10 \ub{1}{\lfloor k/5 \rfloor} 2 \ub{1}{\lfloor k/5 \rfloor} 2$\\
    \hline
    $k \equiv 2 \mod 5$ \newline $k \geq 1$ & $\ub{1}{\lceil k/5 \rceil} 2 \ub{1}{\lfloor k/5 \rfloor} 10 \ub{1}{\lfloor k/5 \rfloor} 2 \ub{1}{\lfloor k/5 \rfloor} 2 \ub{1}{\lfloor k/5 \rfloor} 10$\\
    \hline
    $k \equiv 3 \mod 5$ \newline $k \geq 1$ & $\ub{1}{\lceil k/5 \rceil} 2 \ub{1}{\lfloor k/5 \rfloor} 10 \ub{1}{\lceil k/5 \rceil} 2 \ub{1}{\lfloor k/5 \rfloor} 10 \ub{1}{\lfloor k/5 \rfloor} 2$\\
    \hline
    $k \equiv 4 \mod 5$\newline $k \geq 1$ & $\ub{1}{\lceil k/5 \rceil} 2 \ub{1}{\lceil k/5 \rceil} 2 \ub{1}{\lfloor k/5 \rfloor} 10 \ub{1}{\lceil k/5 \rceil} 2 \ub{1}{\lfloor k/5 \rfloor} 10$\\
    \hline
\end{tabular}

\bigskip
\begin{tabular}{|>{\centering\arraybackslash}m{2.25cm}|c|}
    \hline
    \multicolumn{2} {|c|} {$\alpha_1 + k\delta$}\\
    \hline
    $k =0$ & 1\\
    \hline
    $k =1$ & $1212210$\\
    \hline
    $k = 2$ & $1221221012210$\\
    \hline
    $k \equiv 0 \mod 3$ \newline $k \geq 3$ & $12210 \ub{1}{k/3 - 1} 12210 \ub{1}{k/3 - 1} 12210 \ub{1}{k/3 - 1} 1222$\\
    \hline
    $k \equiv 1 \mod 3$ \newline $k \geq 3$ & $12210 \ub{1}{\lfloor k/3 \rfloor -1} 12210 \ub{1}{\lfloor k/3 \rfloor -1} 122212210 \ub{1}{\lceil k/3 \rceil -1}$\\
    \hline
    $ k \equiv 2 \mod 3$ \newline $k \geq 3$ & $12210 \ub{1}{\lfloor k/3 \rfloor -1} 122212210 \ub{1}{ \lceil k/3 \rceil -1} 12210 \ub{1}{\lceil k/3 \rceil -1 }$\\
    \hline

\end{tabular}

\bigskip
\begin{tabular}{|c|c|}
    \hline
    \multicolumn{2} {|c|} {$\alpha_2 + k\delta$}\\
    \hline
     $k$& $\ub{1}{k}2$\\
     \hline
\end{tabular}
\begin{tabular}{|>{\centering\arraybackslash}m{2.25cm}|c|}
    \hline
    \multicolumn{2} {|c|} {$\alpha_1 + \alpha_2 + k\delta$}\\
    \hline
    $k = 0$ & 12\\
    \hline
    $k =1 $& 12212210\\
    \hline
    $k \equiv 0 \mod 2$\newline $k \geq 2$ & $12210 \ub{1}{k/2-1} 12210 \ub{1}{k/2 -1} 1222$\\
    \hline
    $k \equiv 1 \mod 2$ \newline $k \geq 2$ & $12210 \ub{1}{\lfloor k/2 \rfloor -1} 122212210 \ub{1}{\lceil k/2 \rceil -1}$\\
    \hline
\end{tabular}

\bigskip
\begin{tabular}{|c|c|}
    \hline
    \multicolumn{2} {|c|} {$\alpha_0 + \alpha_1 + k\delta$}\\
    \hline
     $k$& $\ub{1}{k}10$\\
     \hline
\end{tabular}
\begin{tabular}{|c|c|}
    \hline
    \multicolumn{2} {|c|} {$\alpha_1 + 2\alpha_2 + k\delta$}\\
    \hline
     $k = 1$& $122$\\
     \hline
     $k \geq 2$ & $12210 \ub{1}{k-1} 1222$\\
     \hline
\end{tabular}

\bigskip
\begin{tabular}{|c|c|}
    \hline
    \multicolumn{2} {|c|} {$\alpha_0 + \alpha_1 + \alpha_2 + k\delta$}\\
    \hline
    $k \equiv 0 \mod 2$ & $\ub{1}{k/2} 10 \ub{1}{k/2} 2$\\
    \hline
    $k \equiv 1 \mod 2$ & $\ub{1}{\lceil k/2 \rceil} 2 \ub{1}{\lfloor k/2 \rfloor} 10$\\
    \hline
\end{tabular}
\begin{tabular}{|c|c|}
    \hline
    \multicolumn{2} {|c|} {$\alpha_1 + 3\alpha_2 + k\delta$}\\
    \hline
    $k$ & $1222\ub{1}{k}$\\
    \hline
\end{tabular}

\bigskip
\begin{tabular}{|c|c|}
    \hline
    \multicolumn{2} {|c|} {$\alpha_0 + \alpha_1 + 2\alpha_2 + k\delta$}\\
    \hline
    $k \equiv 0 \mod 3$ & $\ub{1}{k/3} 10 \ub{1}{k/3} 2 \ub{1}{k/3} 2$\\
    \hline
    $k\equiv 1 \mod 3$ & $\ub{1}{\lceil k/3 \rceil} 2 \ub{1}{\lfloor k/3 \rfloor} 10 \ub{1}{\lfloor k/3 \rfloor} 2$\\
    \hline
    $k \equiv 2 \mod 3$ & $\ub{1}{\lceil k/3 \rceil} 2 \ub{1}{\lceil k/3 \rceil} 2 \ub{1}{\lfloor k/3 \rfloor} 10$\\
    \hline
\end{tabular}

\bigskip
\begin{tabular}{|>{\centering\arraybackslash}m{2.25cm}|c|}
    \hline
    \multicolumn{2} {|c|} {$2\alpha_1 + 3\alpha_2 + k\delta$}\\
    \hline
     $k =0$&12122\\
     \hline
     $k = 1$ & 12212212210\\
     \hline
     $k = 2$ & $12212210122101222$\\
     \hline
     $k \equiv 0 \mod 3$ \newline $k \geq 3$ & $12210 \ub{1}{k/3-1} 12210 \ub{1}{k/3-1} 122212210 \ub{1}{k/3-1} 1222$\\
     \hline
     $k \equiv 1\mod 3$ & $12210 \ub{1}{\lfloor k/3 \rfloor -1}122212210 \ub{1}{\lfloor k/3 \rfloor - 1} 122212210 \ub{1}{\lceil k/3 \rceil -1}$\\
     \hline
     $k \equiv 2 \mod 3$ & $12210 \ub{1}{\lfloor k/3 \rfloor -1 }122212210 \ub{1}{\lceil k/3 \rceil - 1} 12210 \ub{1}{\lceil k/3 \rceil -1} 1222$\\
     \hline
\end{tabular}

\bigskip
\begin{tabular}{|c|c|}
    \hline
    \multicolumn{2} {|c|} {$\alpha_0 + \alpha_1 + 3\alpha_2 + k\delta$}\\
    \hline
    $k \equiv 0 \mod 4$ & $\ub{1}{k/4} 10 \ub{1}{k/4} 2 \ub{1}{k/4} 2 \ub{1}{k/4} 2$\\
    \hline
    $k \equiv 1 \mod 4$ & $\ub{1}{\lceil k/4 \rceil} 2 \ub{1}{\lfloor k/4 \rfloor} 10 \ub{1}{\lfloor k/4 \rfloor}  2 \ub{1}{\lfloor k/4 \rfloor} 2$\\
    \hline
    $k \equiv 2 \mod 4$ & $\ub{1}{\lceil k/4 \rceil} 2 \ub{1}{\lceil k/4 \rceil} 2 \ub{1}{\lfloor k/4 \rfloor}  10 \ub{1}{\lfloor k/4 \rfloor}  2$\\
    \hline
    $k \equiv 3 \mod 4$ & $\ub{1}{\lceil k/4 \rceil} 2 \ub{1}{\lceil k/4 \rceil} 2 \ub{1}{\lceil k/4 \rceil} 2 \ub{1}{\lfloor k/4 \rfloor}  10$\\
    \hline
\end{tabular}

\bigskip
\begin{tabular}{|c|c|}
    \hline
    \multicolumn{2} {|c|} {$\alpha_0 + 2\alpha_1 + 2\alpha_2 + k\delta$}\\
    \hline
    $k$ & $12210\ub{1}{k}$\\
    \hline
\end{tabular}
\begin{tabular}{|c|c|}
    \hline
    \multicolumn{2} {|c|} {$(k\delta,1)$}\\
    \hline
    $k>1$ & $1222 \ub{1}{k-1} 10$\\
    \hline
\end{tabular}
\begin{tabular}{|c|c|}
    \hline
    \multicolumn{2} {|c|} {$(k\delta,2)$}\\
    \hline
    $k>1$ & $12210 \ub{1}{k-1} 2$\\
    \hline
\end{tabular}
\end{center}


\raggedright
\subsection{Order $2 < 0 < 1$}\hfill\\

\bigskip
\begin{center}

\begin{tabular}{|c|c|}
    \hline
    & $\SL(\cdot)$\\
    \hline
    $(\delta,1)$ & $221021$\\
    \hline
    $(\delta,2)$ & $221210$\\
    \hline
\end{tabular}
\begin{tabular}{|c|c|}
    \hline
    \multicolumn{2} {|c|} {$\alpha_0 + k \delta$}\\
    \hline
    $k$ & $\ub{1}{k}0$\\
    \hline
\end{tabular}

\bigskip
\begin{tabular}{|>{\centering\arraybackslash}m{2.25cm}|c|}
    \hline
    \multicolumn{2} {|c|} {$\alpha_1 + k \delta$}\\
    \hline
     $k =0$ & $1$\\
     \hline
     $k\equiv 0 \mod 4$ \newline $k \geq 1$ & $\ub{1}{k/4} 21 \ub{1}{k/4} 21 \ub{1}{k/4} 0 \ub{1}{k/4 - 1} 21$\\
     \hline
     $k \equiv 1 \mod 4$ \newline $k \geq 1$ & $\ub{1}{\lfloor k/4 \rfloor} 21 \ub{1}{\lfloor k/4 \rfloor} 21 \ub{1}{\lfloor k/4 \rfloor} 21 \ub{1}{\lfloor k/4 \rfloor} 0$\\
     \hline
     $k \equiv 2 \mod 4$ \newline $k \geq 1$ & $\ub{1}{\lceil k/4 \rceil} 0 \ub{1}{\lfloor k/4 \rfloor} 21 \ub{1}{\lfloor k/4 \rfloor} 21 \ub{1}{\lfloor k/4 \rfloor} 21$\\
     \hline
     $k \equiv 3 \mod 4$ \newline $k \geq 1$ & $\ub{1}{\lceil k/4 \rceil} 21 \ub{1}{\lceil k/4 \rceil} 0 \ub{1}{\lfloor k/4 \rfloor} 21 \ub{1}{\lfloor k/4 \rfloor} 21$\\
     \hline
\end{tabular}

\bigskip
\begin{tabular}{|>{\centering\arraybackslash}m{2.25cm}|c|}
    \hline
    \multicolumn{2} {|c|} {$\alpha_2 + k \delta$}\\
    \hline
    $k=0$ & 2\\
    \hline
    $k =1$ & 2212210\\
    \hline
    $k \geq 2$ & $22121 \ub{1}{k-2} 22102210$\\
    \hline
\end{tabular}
\begin{tabular}{|c|c|}
    \hline
    \multicolumn{2} {|c|} {$\alpha_1  + \alpha_2 + k \delta$}\\
    \hline
    $k$ & $\ub{1}{k}21$\\
    \hline
\end{tabular}

\bigskip
\begin{tabular}{|>{\centering\arraybackslash}m{2.25cm}|c|}
    \hline
    \multicolumn{2} {|c|} {$\alpha_0 + \alpha_1 + k \delta$}\\
    \hline
    $k =0 $ & 01\\
    \hline
    $k \equiv 0 \mod 5$ \newline $k \geq 1$ & $\ub{1}{k/5} 21 \ub{1}{k/5} 0 \ub{1}{k/5} 21 \ub{1}{k/5} 0 \ub{1}{k/5 -1} 21$\\
    \hline
    $k \equiv 1 \mod 5$ \newline $k \geq 1$ & $\ub{1}{\lfloor k/5 \rfloor} 21 \ub{1}{\lfloor k/5 \rfloor} 21 \ub{1}{\lfloor k/5 \rfloor} 0 \ub{1}{\lfloor k/5 \rfloor} 21 \ub{1}{\lfloor k/5 \rfloor} 0$\\
    \hline
    $k \equiv 2 \mod 5$ \newline $k \geq 1$ & $\ub{1}{\lceil k/5 \rceil} 0 \ub{1}{\lfloor k/5 \rfloor} 21 \ub{1}{\lfloor k/5 \rfloor} 21 \ub{1}{\lfloor k/5 \rfloor} 21 \ub{1}{\lfloor k/5 \rfloor} 0$\\
    \hline
    $k \equiv 3 \mod 5$ \newline $k \geq 1$ & $\ub{1}{\lceil k/5 \rceil} 0 \ub{1}{\lfloor k/5 \rfloor} 21 \ub{1}{\lceil k/5 \rceil} 0 \ub{1}{\lfloor k/5 \rfloor} 21 \ub{1}{\lfloor k/5 \rfloor} 21$\\
    \hline
    $k \equiv 4 \mod 5$ \newline $k \geq 1$ & $\ub{1}{\lceil k/5 \rceil} 21 \ub{1}{\lceil k/5 \rceil} 0 \ub{1}{\lfloor k/5 \rfloor} 21 \ub{1}{\lfloor k/5 \rfloor} 21 \ub{1}{\lceil k/5 \rceil} 0$\\
    \hline
\end{tabular}

\bigskip
\begin{tabular}{|c|c|}
    \hline
    \multicolumn{2} {|c|} {$\alpha_1 + 2\alpha_2 + k \delta$}\\
    \hline
     $k=0$ & $221$\\
     \hline
     $k \geq 1$ & $22121 \ub{1}{k-1} 2210$\\
     \hline
\end{tabular}
\begin{tabular}{|c|c|}
    \hline
    \multicolumn{2} {|c|} {$\alpha_0 + \alpha_1 + \alpha_2 + k \delta$}\\
    \hline
     $k\equiv 0 \mod 2$ & $\ub{1}{k/2} 21 \ub{1}{k/2} 0$\\
     \hline
     $k \equiv 1 \mod 2$ & $\ub{1}{\lceil k/2 \rceil} 0 \ub{1}{\lfloor k/2 \rfloor} 21$\\
     \hline
\end{tabular}

\bigskip
\begin{tabular}{|>{\centering\arraybackslash}m{2.25cm}|c|}
    \hline
    \multicolumn{2} {|c|} {$\alpha_1 + 3\alpha_2 +  k \delta$}\\
    \hline
    $k = 0$ & $2221$\\
    \hline
    $k =1$ & 2212212210\\
    \hline
    $k=2$ & 2212210221022121\\
    \hline
    $k \equiv 0 \mod 2$\newline $k \geq 3$ & $22121 \ub{1}{k/2 - 2} 22102210221022121 \ub{1}{k/2 - 1}$\\
    \hline
    $k \equiv 1 \mod 2$ \newline $k \geq 3$ & $22121 \ub{1}{\lfloor k/2 \rfloor -1 } 221022121 \ub{1}{\lfloor k/2 \rfloor -1} 22102210$\\
    \hline
\end{tabular}

\bigskip
\begin{tabular}{|c|c|}
    \hline
    \multicolumn{2} {|c|} {$\alpha_0 + \alpha_1 + 2\alpha_2 + k \delta$}\\
    \hline
     $k$ & $2210\ub{1}{k}$\\
     \hline
\end{tabular}
\begin{tabular}{|c|c|}
    \hline
    \multicolumn{2} {|c|} {$2\alpha_1 + 3\alpha_2 + k \delta$}\\
    \hline
     $k$ & $22121\ub{1}{k}$\\
     \hline
\end{tabular}

\bigskip
\begin{tabular}{|c|c|}
    \hline
    \multicolumn{2} {|c|} {$\alpha_0 + \alpha_1 + 3\alpha_2 +  k \delta$}\\
    \hline
    $k = 0$ & 22210\\
    \hline
    $k=1$ & 22122102210\\
    \hline
    $k\geq 2$ & $22121 \ub{1}{k-2} 221022102210$\\
    \hline
\end{tabular}

\bigskip
\begin{tabular}{|c|c|}
    \hline
    \multicolumn{2} {|c|} {$\alpha_0 + 2\alpha_1 + 2\alpha_2 +  k \delta$}\\
    \hline
    $k\equiv 0 \mod 3$ & $\ub{1}{k/3} 21 \ub{1}{k/3} 21 \ub{1}{k/3} 0$\\
    \hline
    $k \equiv 1 \mod 3$ & $\ub{1}{\lceil k/3 \rceil} 0 \ub{1}{\lfloor k/3 \rfloor} 21 \ub{1}{\lfloor k/3 \rfloor} 21$\\
    \hline
    $k \equiv 2 \mod 3$ & $\ub{1}{\lceil k/3 \rceil} 21 \ub{1}{\lceil k/3 \rceil} 0 \ub{1}{\lfloor k/3 \rfloor} 21$\\
    \hline
\end{tabular}

\bigskip
\begin{tabular}{|c|c|}
    \hline
    \multicolumn{2} {|c|} {$(k\delta,1)$}\\
    \hline
    $k>1$ & $2210 \ub{1}{k-1} 21$\\
    \hline
\end{tabular}
\begin{tabular}{|c|c|}
    \hline
    \multicolumn{2} {|c|} {$(k\delta,2)$}\\
    \hline
    $k>1$ & $22121 \ub{1}{k-1} 0$\\
    \hline
\end{tabular}
\end{center}


\raggedright
\medskip
\subsection{Order $2 < 1 < 0$}\hfill\\

\bigskip
\begin{center}

 \begin{tabular}{|c|c|}
    \hline
    & $\SL(\cdot)$\\
    \hline
    $(\delta,1)$ & $221021$\\
    \hline
    $(\delta,2)$ & $221210$\\
    \hline
\end{tabular}
\begin{tabular}{|c|c|}
    \hline
    \multicolumn{2} {|c|} {$\alpha_0 + k \delta$}\\
    \hline
    $k$ & $\ub{1}{k}0$\\
    \hline
\end{tabular}

\bigskip
\begin{tabular}{|>{\centering\arraybackslash}m{2.25cm}|c|}
    \hline
    \multicolumn{2} {|c|} {$\alpha_1 + k \delta$}\\
    \hline
     $k =0$ & $1$\\
     \hline
     $k\equiv 0 \mod 4$ \newline $k \geq 1$ & $\ub{1}{k/4} 21 \ub{1}{k/4} 21 \ub{1}{k/4} 0 \ub{1}{k/4 - 1} 21$\\
     \hline
     $k \equiv 1 \mod 4$ \newline $k \geq 1$ & $\ub{1}{\lfloor k/4 \rfloor} 21 \ub{1}{\lfloor k/4 \rfloor} 21 \ub{1}{\lfloor k/4 \rfloor} 21 \ub{1}{\lfloor k/4 \rfloor} 0$\\
     \hline
     $k \equiv 2 \mod 4$ \newline $k \geq 1$ & $\ub{1}{\lceil k/4 \rceil} 0 \ub{1}{\lfloor k/4 \rfloor} 21 \ub{1}{\lfloor k/4 \rfloor} 21 \ub{1}{\lfloor k/4 \rfloor} 21$\\
     \hline
     $k \equiv 3 \mod 4$ \newline $k \geq 1$ & $\ub{1}{\lceil k/4 \rceil} 21 \ub{1}{\lceil k/4 \rceil} 0 \ub{1}{\lfloor k/4 \rfloor} 21 \ub{1}{\lfloor k/4 \rfloor} 21$\\
     \hline
\end{tabular}

\bigskip
\begin{tabular}{|>{\centering\arraybackslash}m{2.25cm}|c|}
    \hline
    \multicolumn{2} {|c|} {$\alpha_2 + k \delta$}\\
    \hline
    $k=0$ & 2\\
    \hline
    $k =1$ & 2212210\\
    \hline
    $k \geq 2$ & $22121 \ub{1}{k-2} 22102210$\\
    \hline
\end{tabular}
\begin{tabular}{|c|c|}
    \hline
    \multicolumn{2} {|c|} {$\alpha_1  + \alpha_2 + k \delta$}\\
    \hline
    $k$ & $\ub{1}{k}21$\\
    \hline
\end{tabular}

\bigskip
\begin{tabular}{|>{\centering\arraybackslash}m{2.25cm}|c|}
    \hline
    \multicolumn{2} {|c|} {$\alpha_0 + \alpha_1 + k \delta$}\\
    \hline
    $k =0 $ & 10 \\
    \hline
    $k \equiv 0 \mod 5$ \newline $k \geq 1$ & $\ub{1}{k/5} 21 \ub{1}{k/5} 0 \ub{1}{k/5} 21 \ub{1}{k/5} 0 \ub{1}{k/5 -1} 21$\\
    \hline
    $k \equiv 1 \mod 5$ \newline $k \geq 1$ & $\ub{1}{\lfloor k/5 \rfloor} 21 \ub{1}{\lfloor k/5 \rfloor} 21 \ub{1}{\lfloor k/5 \rfloor} 0 \ub{1}{\lfloor k/5 \rfloor} 21 \ub{1}{\lfloor k/5 \rfloor} 0$\\
    \hline
    $k \equiv 2 \mod 5$ \newline $k \geq 1$ & $\ub{1}{\lceil k/5 \rceil} 0 \ub{1}{\lfloor k/5 \rfloor} 21 \ub{1}{\lfloor k/5 \rfloor} 21 \ub{1}{\lfloor k/5 \rfloor} 21 \ub{1}{\lfloor k/5 \rfloor} 0$\\
    \hline
    $k \equiv 3 \mod 5$ \newline $k \geq 1$ & $\ub{1}{\lceil k/5 \rceil} 0 \ub{1}{\lfloor k/5 \rfloor} 21 \ub{1}{\lceil k/5 \rceil} 0 \ub{1}{\lfloor k/5 \rfloor} 21 \ub{1}{\lfloor k/5 \rfloor} 21$\\
    \hline
    $k \equiv 4 \mod 5$ \newline $k \geq 1$ & $\ub{1}{\lceil k/5 \rceil} 21 \ub{1}{\lceil k/5 \rceil} 0 \ub{1}{\lfloor k/5 \rfloor} 21 \ub{1}{\lfloor k/5 \rfloor} 21 \ub{1}{\lceil k/5 \rceil} 0$\\
    \hline
\end{tabular}

\bigskip
\begin{tabular}{|c|c|}
    \hline
    \multicolumn{2} {|c|} {$\alpha_1 + 2\alpha_2 + k \delta$}\\
    \hline
     $k=0$ & $221$\\
     \hline
     $k \geq 1$ & $22121 \ub{1}{k-1} 2210$\\
     \hline
\end{tabular}
\begin{tabular}{|c|c|}
    \hline
    \multicolumn{2} {|c|} {$\alpha_0 + \alpha_1 + \alpha_2 + k \delta$}\\
    \hline
     $k\equiv 0 \mod 2$ & $\ub{1}{k/2} 21 \ub{1}{k/2} 0$\\
     \hline
     $k \equiv 1 \mod 2$ & $\ub{1}{\lceil k/2 \rceil} 0 \ub{1}{\lfloor k/2 \rfloor} 21$\\
     \hline
\end{tabular}

\bigskip
\begin{tabular}{|>{\centering\arraybackslash}m{2.25cm}|c|}
    \hline
    \multicolumn{2} {|c|} {$\alpha_1 + 3\alpha_2 +  k \delta$}\\
    \hline
    $k = 0$ & $2221$\\
    \hline
    $k =1$ & 2212212210\\
    \hline
    $k=2$ & 2212210221022121\\
    \hline
    $k \equiv 0 \mod 2$\newline $k \geq 3$ & $22121 \ub{1}{k/2 - 2} 22102210221022121 \ub{1}{k/2 - 1}$\\
    \hline
    $k \equiv 1 \mod 2$ \newline $k \geq 3$ & $22121 \ub{1}{\lfloor k/2 \rfloor -1 } 221022121 \ub{1}{\lfloor k/2 \rfloor -1} 22102210$\\
    \hline
\end{tabular}

\medskip
\begin{tabular}{|c|c|}
    \hline
    \multicolumn{2} {|c|} {$\alpha_0 + \alpha_1 + 2\alpha_2 + k \delta$}\\
    \hline
     $k$ & $2210\ub{1}{k}$\\
     \hline
\end{tabular}
\begin{tabular}{|c|c|}
    \hline
    \multicolumn{2} {|c|} {$2\alpha_1 + 3\alpha_2 + k \delta$}\\
    \hline
     $k$ & $22121\ub{1}{k}$\\
     \hline
\end{tabular}
\begin{tabular}{|c|c|}
    \hline
    \multicolumn{2} {|c|} {$\alpha_0 + \alpha_1 + 3\alpha_2 +  k \delta$}\\
    \hline
    $k = 0$ & 22210\\
    \hline
    $k=1$ & 22122102210\\
    \hline
    $k\geq 2$ & $22121 \ub{1}{k-2} 221022102210$\\
    \hline
\end{tabular}

\medskip
\begin{tabular}{|c|c|}
    \hline
    \multicolumn{2} {|c|} {$\alpha_0 + 2\alpha_1 + 2\alpha_2 +  k \delta$}\\
    \hline
    $k\equiv 0 \mod 3$ & $\ub{1}{k/3} 21 \ub{1}{k/3} 21 \ub{1}{k/3} 0$\\
    \hline
    $k \equiv 1 \mod 3$ & $\ub{1}{\lceil k/3 \rceil} 0 \ub{1}{\lfloor k/3 \rfloor} 21 \ub{1}{\lfloor k/3 \rfloor} 21$\\
    \hline
    $k \equiv 2 \mod 3$ & $\ub{1}{\lceil k/3 \rceil} 21 \ub{1}{\lceil k/3 \rceil} 0 \ub{1}{\lfloor k/3 \rfloor} 21$\\
    \hline
\end{tabular}

\medskip
\begin{tabular}{|c|c|}
    \hline
    \multicolumn{2} {|c|} {$(k\delta,1)$}\\
    \hline
    $k>1$ & $2210 \ub{1}{k-1} 21$\\
    \hline
\end{tabular}
\begin{tabular}{|c|c|}
    \hline
    \multicolumn{2} {|c|} {$(k\delta,2)$}\\
    \hline
    $k>1$ & $22121 \ub{1}{k-1} 0$\\
    \hline
\end{tabular}
\end{center}


\end{document}